\documentclass[12pt, reqno]{amsart}

\usepackage{amsmath,amsthm,amsfonts,amscd,amssymb,bbm,enumerate}
\usepackage{cases}
\usepackage{times,euler,eucal,graphicx}
\usepackage{hyperref}
\usepackage[linecolor=white,backgroun dcolor=white,bordercolor=white,textsize=tiny]{todonotes}
\usepackage{cancel}

\usepackage{soul}
\makeatletter

\def\SOUL@ulstunderline#1{{%
    \setbox\z@\hbox{#1}%
    \dimen@=\wd\z@
    \dimen@i=\SOUL@uloverlap
    \advance\dimen@2\dimen@i
    \rlap{
        \null
        \kern-\dimen@i
        \SOUL@ulcolor{\SOUL@ulleaders\hskip\dimen@}%
    }%
    \SOUL@stpreamble
    \rlap{%
        \null
        \kern-\dimen@i
        \SOUL@ulcolor{\SOUL@ulleaders\hskip\dimen@}%
    }%
    \unhcopy\z@
}}
\def\SOUL@ulsteverysyllable{%
    \SOUL@ulstunderline{%
        \the\SOUL@syllable
        \SOUL@setkern\SOUL@charkern
    }%
}
\def\SOUL@ulstsetup{%
  \SOUL@ulsetup
  \let\SOUL@everysyllable\SOUL@ulsteverysyllable
}
\DeclareRobustCommand*\textulst{\SOUL@ulstsetup\SOUL@}

\usepackage{caption} 
\usepackage{subcaption}

\usepackage{IEEEtrantools}

\usepackage{color}

\renewcommand\labelenumi{(\roman{enumi})}
\renewcommand\theenumi\labelenumi

\numberwithin{equation}{section}
\newtheorem{theorem}{Theorem}[section]
\newtheorem{lemma}[theorem]{Lemma}
\newtheorem{corollary}[theorem]{Corollary}
\newtheorem{proposition}[theorem]{Proposition}
\newtheorem{assumption}{Assumption}

\theoremstyle{definition}
\newtheorem{definition}[theorem]{Definition}

\newtheorem{remark}[theorem]{Remark}

\renewcommand{\epsilon}{\varepsilon}

\newcommand{\Y}{Y}

\newcommand{\X}{X}

\newcommand{\A}{\mathcal{A}}
\newcommand{\B}{\mathcal{B}}
\newcommand{\C}{\mathbb{C}}
\newcommand{\D}{\mathcal{D}}
\newcommand{\E}{\mathbb{E}}
\newcommand{\F}{\mathcal{F}}
\newcommand{\G}{\mathcal{G}}
\newcommand{\bH}{\mathbb{H}}

\newcommand{\bN}{\mathbb{N}}

\newcommand{\R}{\mathbb{R}}
\renewcommand{\S}{\mathcal{S}}

\newcommand{\NC}{\operatorname{NC}}
\newcommand{\bcX}{\mathcal{X}}

\newcommand{\uy}{\mathbf{y}}
\newcommand{\uz}{\mathbf{z}}

\newcommand{\ux}{\mathbf{x}}

\newcommand{\Zb}{\mathfrak{b}}

\newcommand{\x}{\textbf{x}}

\newcommand{\id}{\operatorname{id}}

\newcommand{\tr}{\operatorname{tr}}

\newcommand{\sa}{{\operatorname{sa}}}

\newcommand{\1}{\mathbf{1}}

\newcommand{\changelocaltocdepth}[1]{%
  \addtocontents{toc}{\protect\setcounter{tocdepth}{#1}}%
  \setcounter{tocdepth}{#1}%
}

\makeatletter
\def\moverlay{\mathpalette\mov@rlay}
\def\mov@rlay#1#2{\leavevmode\vtop{%
\baselineskip\z@skip \lineskiplimit-\maxdimen
\ialign{\hfil$#1##$\hfil\cr#2\crcr}}}
\makeatother

\makeatletter
\def\@settitle{\begin{center}%
  \baselineskip14\p@\relax
    \normalfont \Large \uppercase{ \textbf{\@title}}
  \end{center}%
 }

\makeindex

\title[Berry-Esseen Bounds]{
Berry-Esseen bounds for the multivariate $\B$-free CLT \\ and operator-valued matrices
}

\author[M. Banna]{Marwa Banna}
\address{New York University Abu Dhabi, Division of Science, Mathematics, Abu Dhabi, UAE}
\email{marwa.banna@nyu.edu}

\author[T. Mai]{Tobias Mai}
\address{Saarland University, Department of Mathematics, D-66123 Saarbr\"{u}cken, Germany}
\email{mai@math.uni-sb.de}

\date{\today}

\thanks{This work has been partially supported by the ERC Advanced Grant NCDFP 339760 held by Roland Speicher. The authors would like to thank Guillaume C\'ebron and Roland Speicher for helpful discussions,  and the referee for the substantial suggestions/comments that helped improve on the paper.}

\begin{document}

\begin{abstract}
We provide bounds of Berry-Esseen type for fundamental limit theorems in operator-valued free probability theory such as the operator-valued free Central Limit Theorem and the asymptotic behaviour of distributions of operator-valued matrices. Our estimates are on the level of operator-valued Cauchy transforms and the L\'evy distance. We address the single-variable as well as the multivariate setting for which we consider linear matrix pencils and noncommutative polynomials as test functions. The estimates are in terms of operator-valued moments and yield the first quantitative bounds on the L\'evy distance for the operator-valued free Central Limit Theorem. Our results also yield quantitative estimates on joint noncommutative distributions of operator-valued matrices having a general covariance profile. In the scalar-valued multivariate case, these estimates could be passed to explicit bounds on the order of convergence under the Kolmogorov distance.
\end{abstract}

\maketitle
\vspace{-1cm}
\section{Introduction}
Since its foundation in the 1980's by Voiculescu, the development of free probability theory has drawn much inspiration from its deep and far reaching analogy with classical probability theory. The same is true for operator-valued free probability, where the fundamental notion of free independence is generalized to free independence with amalgamation as a kind of conditional version of the former. Its development naturally led to operator-valued free analogues of key and fundamental limiting theorems such as the operator-valued free Central Limit Theorem (CLT) due to Voiculescu \cite{Voi-95, Speicher-98, BPV2013, Je-Li-19} and results about the asymptotic behaviour of distributions of matrices with operator-valued entries \cite{Voi-91,Shl-96, Sh-97, Ryan-98,Liu-18, Ba-Ce-18}. In this paper, we give quantitative versions of such limit theorems by providing bounds on the level of the operator-valued Cauchy transform and the L\'evy distance.

When applied to the particular case of the operator-valued free CLT, our results lead to operator-valued free analogues of the classical Berry-Esseen bounds(see Section \ref{section:free-CLT}); this continues \cite{Chi-Got-08} and \cite{Sp-07,Mai-Speicher-13}. Let $(\A, \varphi, E , \B)$ be an operator-valued $W^*$-probability space and  $x=\{x_1, \ldots, x_n\}$ be a family of selfadjoint elements in $\A$ that are free with amalgamation over $\B$ and are such that $E[x_j] =0$. Set $X_n= n^{-1/2} \sum_{j=1}^n x_j$ and let $S_n$ be an operator-valued semicircular element over $\B$ whose variance is given by the completely positive map $
\eta_n : \B \rightarrow \B$, $\eta_n(b) = n^{-1} \sum_{j=1}^n E[x_j b x_j]
$. Denoting by $\G^\B_a(\Zb) = E[(\Zb-a)^{-1}]$ the $\B$-valued Cauchy transform of a selfadjoint element $a$ in $\A$,
the bound which we will derive reads then as follows: for any $\Zb$ in the upper half plane $\bH^+(\B)$,
\[
\| \G^\B_{X_n}(\Zb) - \G^\B_{S_n}(\Zb) \| \leq \frac{1}{\sqrt{n}} \|\Im(\Zb)^{-1}\|^4 A_1(x),
\]
where $\Im(b)$ is the imaginary part of $\Zb$ and $A_1(x)$ depends only on the second and fourth $\B$-valued moments of the elements of the family $x$. As this bound also holds on the level of the fully matricial extensions of Cauchy transforms, it is sufficient to capture convergence in distribution over $\B$; see Remark \ref{rem:Cauchy-convergence}.
Furthermore, we give a first quantitative bound for the operator-valued CLT in terms of the L\'evy distance. Indeed, we prove that there exists a universal constant $c>0$ such that
\[
L(\mu_{X_n}, \mu_{\S_n}) \leq c A_1(x)^{1/7} n^{-1/14}
\]
where $L(\mu_{X_n}, \mu_{\S_n})$ denotes the L\'evy distance between the analytic distributions $\mu_{X_n}$ and $\mu_{S_n}$ of $X_n$ and $S_n$ respectively. It is worth mentioning that this bound on the L\'evy distance holds without requiring any regularity conditions on $\mu_{S_n}$.
Moreover, it extends to the setting of unbounded operators as it does not depend on the norms of the operators involved; see Proposition \ref{prop:Op-freeCLT_unbounded}.

A natural step afterwards is to consider extensions of the multivariate setting to the operator-valued realm and investigate joint distributions. Indeed, noncommutative distributions transfer the well established notion of joint distributions known from classical probability theory to the realm of noncommutative probability. However, in the noncommutative setting,
the definition of these joint distributions is purely combinatorial in nature, in contrast to classical probability theory where the joint distribution of $n$ commuting random variables can be described as a Borel probability measure on $\R^n$. Apart the particular case of commuting variables, including that of a single selfadjoint or normal noncommutative random variable, one cannot encode a noncommutative distribution analytically via a compactly supported real probability measure. Nonetheless, much work has been done in recent years to uncover the still existing rich analytic structure of noncommutative distributions. A typical and also  successful approach is to consider noncommutative test functions and study for each evaluation the distribution of the single noncommutative random variable produced by means of measure theory.

Our aim is to study noncommutative joint distributions of correlated sums in freely independent elements over $\B$ that are not necessarily identically distributed. More precisely, our object of interest is of the form
$
\big(  \sum_j x^{(1)}_j, \dots ,  \sum_j x^{(d)}_j  \big)
$
for which general correlations are allowed between different sums. We consider two classes of test functions, namely selfadjoint linear matrix pencils and noncommutative polynomials, and give explicit estimates on the associated Cauchy transforms and on the L\'evy distance. The case of linear matrix pencils follows from the single-variable operator-valued setting described above and hence the bounds are in terms of the second and fourth $\B$-valued moments. However, this is not the case for noncommutative polynomials which thus need to be treated separately. In this case, the operator norm of the variables appears in the estimates but only as a non-leading term. The leading terms are again in terms of the second and fourth $\B$-valued moments.

We emphasize in this paper on two immediate consequences of this general framework: the operator-valued multivariate free CLT and joint distributions of operator-valued matrices having a general covariance profile.  

The novelties that our approach provides on the $\B$-valued multivariate free CLT, proved in \cite{Voi-95, Speicher-98}, are the explicit quantitative estimates on the operator-valued Cauchy transforms in terms of moments with the optimal rate of convergence. In particular, the estimate on the operator-valued Cauchy transform yields convergence in $*$-distribution over $\B$. Moreover, the estimate on the scalar-valued Cauchy transform establishes the first quantitative bound on the L\'evy distance that holds without demanding any regularity conditions on the analytic distribution of the limiting object. In the scalar-valued multivariate case, i.e. when $\B=\C$, our results also yield explicit bounds on the order of convergence in terms of the Kolmogorov distance. 

As for operator-valued matrices, we are interested in studying joint noncommutative distributions of families of correlated matrices having a general covariance profile. Indeed, Voiculescu proved in his fundamental paper \cite{Voi-91} that a family of independent GUE matrices is asymptotically free and converges in $*$-distribution to free semicircular elements. This result was then extended to families of free operator-valued Wigner matrices \cite{Shl-96, Sh-97, Ryan-98,Liu-18}. Our results provide immediately the first quantitative estimates on scalar-valued Cauchy transforms when considering linear matrix pencils and noncommutative polynomials as test functions. The entries of each individual matrix are free with amalgamation over $\B$ but correlations are allowed between the matrices themselves. These explicit bounds give, when $\B= \mathbb{C}$, quantitative estimates on the order of convergence under the Kolmogorov distance.

To obtain quantitative bounds on Cauchy transforms, we extend and refine an operator-valued Lindeberg method by blocks. In the noncommutative setting, this method was employed for approximation purposes on mixed moments in \cite{Kargin} to generalize Voiculescu's free CLT and in \cite{Deya-Nourdin, Simone} to prove an invariance principle for multilinear homogeneous sums in free elements. It was recently extended for approximations on Cauchy transforms in \cite{Ba-Ce-18} to study distributions of operator-valued matrices with free or exchangeable entries.  
However, using the same machinery in the setting of noncommutative polynomials is not directly applicable since linearity is an essential ingredient for the Lindeberg method. For this purpose, we use the linearization trick that allows to pass from a problem involving arbitrary polynomials to a problem about linear polynomials but with variables living in an amplified probability space. Then our operator-valued extension of the Lindeberg method, together with the linearization trick and essential estimates on the linearization matrix, provide quantitative and explicit estimates on Cauchy transforms that can be passed onto the L\'evy distance. Our results can be extended to study Chebyshev sums; this would be the subject of a future project. 
 
The paper is organized as follows: The general estimates in the operator setting are stated in Theorem \ref{theo1:Lin-Lin:freeness} 
followed by applications to the operator-valued free CLT in Theorem \ref{theo:Op-freeCLT} and operator-valued Wigner matrices in Theorem \ref{theo:op-matrices}. An introduction to noncommutative joint distributions and how they can be studied with the help of suitable noncommutative test functions is given in Section \ref{section:Lin,-Lin} before stating our main results on the multivariate setting in Theorems \ref{theo:Linear-pencil:freeness} and \ref{theo:Lin-Lin:freeness} for linear matrix pencils and noncommutative polynomials respectively. Applications to the multivariate $\B$-free CLT and distributions of families of operator-valued matrices are then given in Sections \ref{section:MCLT}, \ref{section:matrices} and \ref{section:Wigner}. The proofs of the main theorems are postponed to Sections \ref{section:proof-operator-valued} and \ref{section:proof-Polynomials}, whereas Section \ref{section:useful-lemmas} is dedicated to prove useful lemmas that yield key estimates and that could be of independent interest. Finally, we recall in Section \ref{section:preliminaries} preliminary results from free probability theory that are essential for our purposes and present in Section \ref{section:linearizations} some background on the linearization trick. 

\setcounter{tocdepth}{1}
\tableofcontents

\section{Preliminaries and notations}\label{section:preliminaries}

In this section, we give a brief introduction to some basic concepts of free probability theory, both in the scalar and the operator-valued settings. Our exposition relies mainly on \cite{Ni-Di-Sp-operator-valued, Nica-Speicher, Speicher-98}.

\subsection{Scalar and operator-valued probability spaces}

In the most basic and purely algebraic setting, a \emph{noncommutative probability space} means a pair $(\A,\varphi)$ of a unital complex algebra $\A$ and a unital linear functional $\varphi: \A \to \C$. Elements of $\A$ are considered as \emph{noncommutative random variables} and $\varphi$ is referred to as the \emph{expectation} on $\A$.

For our purpose, we need to impose in addition some topological structure. Such an analytic setting is provided for instance by \emph{$C^\ast$-probability spaces}, which are pairs $(\A,\varphi)$ consisting of a unital $C^\ast$-algebra $\A$ and a distinguished state $\varphi: \A \to \C$ on $\A$.
Another framework of this kind are tracial $W^\ast$-probability spaces. Recall that $(\A,\varphi)$ is said to be a \emph{tracial $W^*$-probability space} if $\A$ is a von Neumann algebra and $\varphi: \A \rightarrow \mathbb{C}$ is a faithful normal tracial state on it.
The \emph{$L^p$-norms} are defined, for all $p\geq 1$ and all $x\in \A$, by $\|x\|_{L^p(\A,\varphi)}=[\varphi(|x|^p)]^{1/p}$, where $|x|=(x^*x)^{1/2}$. We simply denote $\|\cdot\|_{L^p(\A,\varphi)}$ by $\|\cdot\|_{L^p}$ when the context is sufficiently clear.
If $(\A_1, \varphi_1)$ and $(\A_2,\varphi_2)$ are two tracial $W^*$-probability spaces, the tracial $W^*$-probability space $(\A_1 \otimes \A_2,\varphi_1 \otimes \varphi_2)$ is the tensor product von Neumann algebra $\A_1 \otimes \A_2$ endowed with the unique faithful normal tracial state $\varphi_1 \otimes \varphi_2$ such that $\varphi_1 \otimes \varphi_2(x\otimes y)=\varphi_1(x)\varphi_2(y)$ for all $x\in \A_1$ and $y\in \A_2$.
For all $p\in [1,\infty]$ we have
$$\|x\otimes y\|_{L^p(\A_1 \otimes \A_2,\varphi_1 \otimes \varphi_2)} = \|x\|_{L^p(\A_1, \varphi_1)}\cdot \|y\|_{L^p(\A_2, \varphi_2)}.$$

An \emph{operator-valued probability space $(\A,E,\B)$} consists of a unital complex algebra $\A$, a unital complex subalgebra $\B$ of $\A$, which is unitally embedded, and a \emph{conditional expectation $E: \A\to\B$}, i.e., a unital linear map $E: \A\to\B$ satisfying:
\begin{itemize}
  \item $E[b] = b$ for all $b\in\B$ and
  \item $E[b_1 x b_2] = b_1 E[x] b_2$ for all $x\in\A$, $b_1,b_2\in\B$.
\end{itemize}
Operator-valued conditional expectations can be seen as natural noncommutative analogues of conditional expectations known in classical probability.

An \emph{operator-valued $C^\ast$-probability space} means an operator-valued probability space $(\A,E,\B)$ which consists of a unital $C^\ast$-algebra $\A$, a unital $C^\ast$-subalgebra $\B$ of $\A$ which is unitally embedded in $\A$, and a conditional expectation $E: \A\to \B$ which is moreover positive in the sense that $E[a^\ast a]$ is a positive element in $\B$ for each $a\in\A$.

If $(\A,\varphi)$ is a tracial $W^\ast$-probability space and $\B$ a von Neumann subalgebra of $\A$, then there exists a unique conditional expectation $E: \A \to \B$ which is trace preserving in the sense that $\varphi \circ E = \varphi$. Then, the quadruple $(\A,\varphi,E,\B)$ is referred to as an \emph{operator-valued $W^\ast$-probability space}.

\subsection{Scalar and operator-valued Cauchy transforms}\label{subsec:Cauchy_transform}

Let $\A$ be a unital $C^\ast$-algebra. For any element $x\in \A$, we denote by $G_x$ the \emph{resolvent} of $x$ given by $G_x(\Zb)=(\Zb -x)^{-1}$ for any $\Zb \in \A$ such that $\Zb -x$ is invertible in $\A$. If $x=x^*$, then $z \1 - x$ is invertible in $\A$ for all $z\in \mathbb{C}^+$, where $\C^+ := \{z\in \mathbb{C} \mid \Im (z)>0\}$, and if $(\A,\varphi)$ is a $C^\ast$-probability space, we have
$$\G_x(z) := \varphi[G_x(z)] = \int_\mathbb{R}\frac{1}{z-\lambda} \text{d} \mu_x (\lambda),$$
where $\mu_x$ is the \emph{analytic distribution} of $x$, i.e., the unique probability measure on $\mathbb{R}$ with the same moments as $x$.
Note that $\G_x = \varphi \circ G_x: \C^+ \to \C^-$, where $\mathbb{C}^- := \{z\in \mathbb{C} \mid \Im (z)<0\}$, is the \emph{Cauchy transform} of the analytic distribution $\mu_x$, which determines $\mu_x$ completely. As a consequence, the pointwise convergence of $(\G_{x_n})_{n\in\bN}$ to the Cauchy transform  of a measure $\nu$ implies the weak convergence of the sequence of analytic distributions $(\mu_{x_n})_{n\in\bN}$ to $\nu$. We refer to $\G_x$ as the \emph{scalar-valued Cauchy transform}, or simply the \emph{Cauchy transform}, of $x$.

For later use, we record here that, for each $x=x^\ast\in\A$ and every $\epsilon>0$,
\begin{equation}\label{eq:Cauchy-integral}
\int_\R \|G_x(t+i\epsilon)\|^2_{L^2}\, \mathrm{d} t = \frac{\pi}{\epsilon}.
\end{equation}
Indeed, the resolvent identity yields that, for every $z\in\C^+$,
\[
\|G_x(z)\|^2_{L^2} = \varphi\big((\overline{z}-x)^{-1} (z-x)^{-1}\big) = - \frac{\Im(\G_x(z))}{\Im(z)},
\]
from which \eqref{eq:Cauchy-integral} follows, using that $-\frac{1}{\pi} \Im(\G_x(t+i\epsilon))\, \mathrm{d} t$ is a probability measure for every $\epsilon>0$.

Cauchy transforms also play an important role in the analytic treatment of operator-valued free probability theory. Let $(\A,E,\B)$ be an operator-valued $C^\ast$-probability space. We call
$$\bH^\pm(\B) := \{\Zb\in\B \mid \exists \varepsilon>0:\ \pm\Im(\Zb) \geq \varepsilon \1\}$$
the upper and lower half-plane of $\B$, respectively, where we use the notation $\Im(\Zb) := \frac{1}{2i}(\Zb-\Zb^\ast)$. The \emph{$\B$-valued Cauchy transform} of $x$ is the function $\G_x^\B: \bH^+(\B) \to \bH^-(\B)$ defined by
$$\G_x^\B(\Zb) := E[G_x(\Zb)] = E[(\Zb-x)^{-1}] \qquad\text{for all $\Zb\in\bH^+(\B)$.}$$

\subsection{Scalar and operator-valued noncommutative distributions}\label{section:nc-distributions}

The joint noncommutative distribution $\mu_x$ of a family $x=(x_i)_{i\in I}$ of noncommutative random variables in the noncommutative probability space $(\A,\varphi)$ is given as the collection of all joint moments, i.e.,
\[
\mu_x=\big\{\varphi(x_{i_1} \cdots x_{i_k}) \mathrel{\big|} k\in\mathbb{N}_0,\ i_1,\dots,i_k\in I\big\}.
\]
Let $(\A_n, \varphi_n)$ and $(\A, \varphi)$ be noncommutative probability spaces and consider, for each $i\in I$, the random variables $x_n^{(i)} \in \A_n$ and $x_i \in \A$. We say that $(x_n^{(i)})_{i\in I}$ converges in distribution to $(x_i)_{i\in I}$ and  write $(x_n^{(i)})_{i\in I} \stackrel{d}{\longrightarrow} (x_i)_{i\in I}$ if all joint moments of $(x_n^{(i)})_{i\in I}$  converge to the corresponding joint moments of $(x_i)_{i\in I}$ ; i.e., if for any $k\in \mathbb{N}_0$ and $i_1,\dots,i_k\in I$
\[
\lim_{n\rightarrow\infty} \varphi(x_n^{(i_1)} \cdots x_n^{(i_k)}) =
  \varphi(x_{i_1} \cdots x_{i_k}).
\]
We say that $(x_n^{(i)})_{i\in I}$ converges in $*$-distribution to $(x_i)_{i\in I}$ and write $(x_n^{(i)})_{i\in I} \stackrel{~^*d}{\longrightarrow} (x_i)_{i\in I}$ whenever $(x_n^{(i)}, x_n^{(i)*})_{i\in I} \stackrel{d}{\longrightarrow} (x_i, x_i^*)_{i\in I}$. 

Similarly, the $\B$-valued joint distribution $\mu_x^\B$ of a family $x=(x_i)_{i\in I}$ in the operator-valued probability space $(\A,E,\B)$ is given as the collection of all $\B$-valued joint moments, i.e.,
$$ \mu_x^\B=\big\{E[x_{i_1} b_1 x_{i_2} \cdots b_{k-1} x_{i_k} ] \mathrel{\big|} k\in\mathbb{N}_0,\ i_1,\dots,i_k\in I, b_1, \dots , b_{k-1} \in \B \big\}. $$
For any $k\geq 1$, we denote by $m_k^{x_{1}, \dots , x_k}$ the multilinear map given by 
\begin{equation}\label{moment-map}
m_k^{x_{1}, \dots , x_k}: \B^{k-1} \to \B, 
 \qquad 
(b_1, \dots , b_{k-1}) \mapsto E[x_{1} b_1 x_{2} \cdots x_{{n-1}} b_{k-1} x_k].
 \end{equation}
If $(\A,E,\B)$ is an operator-valued $C^\ast$-probability space, then the norm of $m_k^{x_{1}, \dots , x_k}: \B^{k-1} \to \B$ is trivially bounded by $\|m_k^{x_{1}, \dots , x_k}\| \leq \prod_{j=1}^k  \|x_{j}\|$ where we recall that the norm $\|\Phi\|$ of a multilinear map $\Phi: \B^{k-1} \to \B$ is given by
\begin{equation}\label{eq:multilinear-map-norm} 
\|\Phi\| := \sup_{\|b_1\| \leq 1,\ \dots,\ \|b_{k-1}\|\leq 1} \|\Phi(b_1,\dots,b_{k-1})\|.
\end{equation}
We shall simply write $m_k^x$ instead of $m_k^{x_{1}, \dots , x_k}$ whenever $x_{1}=\dots =x_k=x$.

Let $(\A_n,E,\B)$ and $(\A,E,\B)$ be operator-valued $C^\ast$-probability spaces and consider, for each $i\in I$, the random variables $x_n^{(i)} \in \A_n$ and $x_i \in \A$. We say that $(x_n^{(i)})_{i\in I}$ converges in distribution over $\B$ to $(x_i)_{i\in I}$ and write $(x_n^{(i)})_{i\in I} \stackrel{\B-d}{\longrightarrow} (x_i)_{i\in I}$ if all joint $\B$-moments of $(x_n^{(i)})_{i\in I}$  converge in norm to the corresponding joint $\B$-moments of $(x_i)_{i\in I}$ ; i.e., if for any $k\in \mathbb{N}_0$ and $i_1,\dots,i_k\in I$, $b_1,\dots,b_{k-1} \in \B$
\[
\lim_{n\rightarrow\infty} \big\| m_k^{x_n^{(i_1)}, \dots , x_n^{(i_k)}}(b_1,\dots,b_{k-1}) - m_k^{x_{i_1}, \dots , x_{i_k}}(b_1,\dots,b_{k-1}) \big\|=0. 
\]
We say that $(x_n^{(i)})_{i\in I}$ converges in $*$-distribution over $\B$  to $(x_i)_{i\in I}$ and write $(x_n^{(i)})_{i\in I} \stackrel{\B-^*d}{\longrightarrow} (x_i)_{i\in I}$ whenever $(x_n^{(i)},x_n^{(i)*})_{i\in I} \stackrel{\B-d}{\longrightarrow} (x_i,x_i^*)_{i\in I}$.

\begin{remark}\label{rem:Cauchy-convergence}
Consider an operator-valued $C^\ast$-probability space $(\A,E,\B)$ and $X=X^\ast\in\A$. For every $k\in\bN$, we can construct out of $(\A,E,\B)$ the $C^\ast$-probability space $(M_k(\A),\id_k \otimes E,M_k(\B))$; thus, besides the $\B$-valued Cauchy-transform $\G^\B_X: \bH^+(\B) \to \bH^-(\B)$, we have at our disposal also their matricial extensions $\G^{M_k(\B)}_{\1_k \otimes X}: \bH^+(M_k(\B)) \to \bH^-(M_k(\B))$ for every $k\in\bN$.
The relevance of the so-called \emph{fully matricial extension $(\G^{M_k(\B)}_{\1_k \otimes X})_{k\in\bN}$ of $\G^\B_X$} comes from the fact that it can be used to detect convergence of $\B$-valued distributions. To make this more precise, we take operator-valued $C^\ast$-probability spaces $(\A_n,E_n,\B)$ for $n\in\bN$ and $(\A,E,\B)$ over some common $C^\ast$-algebra $\B$, and we consider operators $X_n=X_n^\ast\in\A_n$ for every $n\in\bN$ and $X=X^\ast\in\A$. If $\sup_{n\in\bN} \|X_n\| < \infty$ and if, for every $k\in\bN$, the sequence $(\G^{M_k(\B)}_{\1_k \otimes X_n})_{n\in\bN}$ converges to $\G^{M_k(\B)}_{\1_k\otimes X}$ uniformly on every ball in $\bH^+(M_k(\B))$ which lies at positive distance from $\partial \bH^+(M_k(\B))$, then necessarily
\begin{equation}\label{eq:norm-convergence_moments}
\lim_{n\to\infty} \|m_k^{X_n} - m_k^{X}\| = 0 \qquad\text{for each $k\in\bN_0$}
\end{equation}
and in particular $X_n \stackrel{\B-d}{\longrightarrow} X$; this follows from \cite[Proposition 2.11]{BPV2012}.

If the assumption of selfadjointness is dropped, we may apply the previous consideration to the \emph{hermitizations} $(\tilde{X}_n)_{n\in\bN}$ and $\tilde{X}$ of $(X_n)_{n\in\bN}$ and $X$ which are defined by
$$\tilde{X}_n := \begin{bmatrix} 0 & X_n\\ X_n^\ast & 0 \end{bmatrix} \qquad\text{and}\qquad \tilde{X} := \begin{bmatrix} 0 & X\\ X^\ast & 0 \end{bmatrix}$$
and which belong to the operator-valued $C^\ast$-probability spaces $(M_2(\A_n), \id_2 \otimes E_n, M_2(\B))$ and $(M_2(\A), \id_2 \otimes E, M_2(\B))$, respectively. Notice that $\tilde{X}_n \stackrel{M_2(\B)-d}{\longrightarrow} \tilde{X}$ if and only if $X_n \stackrel{\B-^\ast d}{\longrightarrow} X$.

Now, suppose that we are given a sequence $(\ux_n)_{n\in\bN}$ of $d$-tuples $\ux_n=(x_n^{(1)},\dots,x_n^{(d)}) \in \A_n^d$ and $\ux = (x_1,\dots,x_d) \in \A^d$. Consider the noncommutative random variables
$$X_n := \operatorname{diag}(\ux_n) \qquad\text{and}\qquad X := \operatorname{diag}(\x)$$
which are living in the operator-valued $C^\ast$-probability spaces $(M_d(\A_n), \id_d \otimes E_n, M_d(\B))$ and $(M_d(\A), \id_d \otimes E, M_d(\B))$, respectively. Then $X_n \stackrel{M_d(\B)-^\ast d}{\longrightarrow} X$ if and only if $\ux_n \stackrel{\B-^\ast d}{\longrightarrow} \ux$.
\end{remark}

\subsection{Positivity of conditional expectations}

We recall now some properties that will be useful in the sequel. Let $(\A,E,\B)$ be an operator-valued $C^\ast$-probability space. Since the conditional expectation $E$ is positive, it induces a $\B$-valued pre-inner product
$$\langle \cdot,\cdot\rangle:\ \A\times\A \rightarrow \B,\ (x,y) \mapsto E[x^\ast y]$$
with respect to which $\A$ becomes a right pre-Hilbert $\B$-module. In particular, we have the following analogue of the Cauchy-Schwarz inequality:
\begin{equation}\label{Cauchy-Schwarz}
\|E[x^\ast y]\|^2 \leq \|E[x^\ast x]\| \|E[y^\ast y]\|
\end{equation}
More generally, for any $n\in\bN$, we can turn $\A^n$ into a right pre-Hilbert $\B$-module by endowing it with the $\B$-valued pre-inner product
$$\langle \cdot,\cdot\rangle:\ \A^n \times\A^n \rightarrow \B,\ \bigg(\begin{bmatrix} x_1\\ \vdots\\ x_n \end{bmatrix}, \begin{bmatrix} y_1\\ \vdots\\ y_n \end{bmatrix} \bigg) \mapsto \sum^n_{j=1} E[x^\ast_j y_j].$$
If we impose the ordinary rules of matrix multiplication, we can write $\langle x, y \rangle = E[x^\ast y]$ for all $x,y\in\A^n$. Therefore, the Cauchy-Schwarz inequality formulated in \eqref{Cauchy-Schwarz} holds verbatim for (column vectors) $x,y\in\A^n$.
The positivity of $E$ implies moreover the following important inequality
\begin{equation}\label{positivity}
\| E[x^\ast w x] \| \leq \|w\| \|E[x^\ast x]\|
\end{equation}
which holds for all $x\in\A^n$ and $w\in M_n(\A)$ satisfying $w\geq 0$.

\begin{remark}\label{rem:Cauchy-moment}
Consider a selfadjoint operator $x$ in some operator-valued $C^\ast$-probability space $(\A,E,\B)$. For each $N\in\bN_0$ and for all $z\in\C^+$, we have the expansion
\begin{equation}\label{eq:Cauchy-moment-expansion}
\G_x^\B(z\1) = E[G_x(z\1)] = \sum^N_{k=0} \frac{1}{z^{k+1}} E[x^k] + \frac{1}{z^{N+1} }E[G_x(z\1) x^{N+1}].
\end{equation} 
With the help of \eqref{Cauchy-Schwarz}, we infer from \eqref{eq:Cauchy-moment-expansion} that
\begin{equation}\label{eq:Cauchy-moment-expansion-remainder-bound}
\Big\|\G_x^\B(z\1) - \sum^N_{k=0} \frac{1}{z^{k+1}} \E[x^k]\Big\| \leq \frac{1}{|z|^{N+1}} \|E[G_x(z\1) x^{N+1}]\| \leq \frac{1}{|z|^{N+1}\Im(z)} \|E[x^2]\|^{1/2} \|E[x^{2N}]\|^{1/2}.
\end{equation}
Inductively, it follows from the latter estimates that if $(x_n)_{n\in\bN}$ is a sequence of selfadjoint operators in $\A$ such that $\sup_{n\in\bN} \|x_n\| < \infty$ and such that $(\G_{x_n}^\B)_{n\in\bN}$ converges pointwise on $\bH^+(\B)$ to $\G_x^\B$, then $E[x^k_n] \to E[x^k]$ as $n\to \infty$ for every $k\in\bN_0$.
Now, suppose more specifically that we have
$$\|\G_{x_n}^\B(z\1) - \G_x^\B(z\1)\| \leq \epsilon_n \Big(1+\frac{1}{\Im(z)}\Big)^4 \qquad\text{for all $z\in\C^+$}$$
for some sequence $(\epsilon_n)_{n\in\bN}$ of positive real numbers satisfying $\lim_{n\to \infty} \epsilon_n = 0$; see Remark \ref{rem:convergence_polynomial}. Take $\sigma>0$ such that $\sigma^2 \geq \sup_{n\in\bN}\|E[x_n^2]\|$; then $\|E[x^2]\| \leq \sigma^2$ and \eqref{eq:Cauchy-moment-expansion-remainder-bound} with $N=1$ yields for $z=iy$ with $y>0$ that
$$\|E[x_n] - E[x]\| \leq \frac{2\sigma^2}{y} + \epsilon_n \Big(1+\frac{1}{y}\Big)^4 y^2.$$
Take any $r>0$ and let $n_0\in\bN$ be such that $\epsilon_n < \sigma^2 r^{-3} (1+\frac{1}{r})^{-4}$ for all $n\in\bN$ satisfying $n\geq n_0$; then $\|E[x_n] - E[x]\| \leq f(y)$ for all $y>r$ where $f:(r,\infty)\to\R$ is defined by $f(y) := \frac{2\sigma^2}{y} + \epsilon_n (1+\frac{1}{r})^4 y^2$. The function $f$ attains its minimum $f(y_0) = 3(1+\frac{1}{r})^{4/3} \sigma^{4/3} \epsilon_n^{1/3}$ at the point $y_0 = \sigma^{2/3} (1+\frac{1}{r})^{-4/3} \epsilon_n^{-1/3} \in (r,\infty)$. Thus, for all $n\geq n_0$,
$$\|E[x_n] - E[x]\| \leq 3 \Big(1+\frac{1}{r}\Big)^{4/3} \sigma^{4/3} \epsilon_n^{1/3}.$$
\end{remark}

\subsection{Freeness and freeness with amalgamation}\label{subsec:freeness}

Let $(\A,\varphi)$ be a noncommutative probability space. We say that a family $(\A_i)_{i\in I}$ of unital subalgebras of $\A$ is \emph{freely independent}, if
$$\varphi(x_1 \cdots x_n) = 0$$
whenever we take any finite number $n\in\bN$ of the elements $x_1,\dots,x_n$ which satisfy $\varphi(x_j) = 0$ for $j=1,\dots,n$ and $x_j \in \A_{i_j}$ where $i_j \in I$ and $i_1 \neq i_2, \dots, i_{n-1} \neq i_n.$
Elements $(x_i)_{i\in I}$ are called \emph{freely independent} if the unital subalgebras generated by the $x_i$'s are freely independent.

Similarly we define freeness with amalgamation over $\B$ in the framework of an operator-valued noncommutative probability space $(\A,E,\B)$. We say that the unital subalgebras $(\mathcal{A}_i)_{i\in I}$ of $\A$ with $\B \subseteq \A_i$ for each $i\in I$, are \emph{free with amalgamation over $\B$} if
$$E[x_1 \cdots x_n] = 0$$
holds whenever we take finitely many elements $x_1,\dots,x_n$ in $\A$ satisfying $E[x_j] = 0$ for $j=1,\dots,n$ and $x_j \in \A_{i_j}$ where $i_j \in I$ and $i_1 \neq i_2, \dots, i_{n-1} \neq i_n.$
Elements $(x_i)_{i\in I}$ are called \emph{free with amalgamation over $\B$} if the algebras generated by $\B$ and the $x_i$'s are also so.

\subsection{Scalar and operator-valued moment-cumulant formulas}

The characteristic combinatorial structure behind free probability theory are non-crossing partitions \cite[Lecture 9]{Nica-Speicher}. 
For a finite ordered set $S$, we denote by $\NC(S)$ the set of non-crossing partitions of $S$ and simply write $\NC(n)$ if $S=[n]:=\{1 , \dots , n\}$. 
For any  disjoint sets $S_1$ and $ S_2$ such that $S_1 \cup S_2 = [n]$, we denote by $\NC(S_1,S_2)$ the set  consisting of all non-crossing partitions in  $\NC(n)$ whose blocks are either subsets of $S_1$ or  $S_2$; i.e.,
\[
\NC(S_1,S_2) = \{ \pi \in \NC(n) \mid \pi = \pi_1 \cup \pi_2,\ \pi_1 \in \NC(S_1),\ \pi_2 \in \NC(S_2) \}.
\]
Let $(\A,\varphi)$ be a noncommutative probability space. For $n\in\bN$, the \emph{free cumulants} $\kappa_n: \A^n \to \C$ are multilinear functionals, defined inductively by the \emph{moment cumulant formula}
\begin{equation}\label{moment-cumulant-for}
\varphi(x_1 \cdots x_n) = \sum_{\pi\in\NC(n)} \kappa_\pi(x_1,\dots,x_n)
\end{equation}
where, for $\pi= \{V_1, \dots , V_r\}$, 
\[
\kappa_\pi(x_1,\dots,x_n) : = \prod_{\substack{V \in \pi \\ V=(i_1 , \dots , i_\ell)} }\kappa_\ell (x_{i_1} , \dots , x_{i_\ell}) . 
\]

There exists an operator-valued analogue of the moment cumulant formula. It is also based on the lattice of non-crossing partitions but the definition of cumulants gets slightly more involved as one has to take care now of the order of the variables. In fact, if $(\A,E,\B)$ is an operator-valued probability space, the \emph{operator-valued free cumulants} $\kappa^\B_n: \A^n \rightarrow \B$ for $n\in\bN$ are defined inductively by the moment cumulant formula
\[
E[ x_1 \cdots x_n] = \sum_{\pi \in \NC(n)} \kappa^\B_\pi (x_1, \dots , x_n),
\]
where the arguments of $\kappa^\B_\pi$ are distributed according to the blocks of $\pi$ by nesting the cumulants inside each other according to the nesting of the blocks of $\pi$. In other words, the $\B$-valued cumulant $\kappa^\B_\pi$ needs to remember the position of any block of the non-crossing partition $\pi$. We refer to \cite[Section 9.2]{Mingo-Speicher} for a more detailed presentation of operator-valued free cumulants. 

Free independence, both in the scalar and in the operator-valued case, provides rules to compute mixed moments; \cite[Section 3.4]{Speicher-98} makes this explicit:  
Let $x_1,\dots,x_n$ be elements in $\A_1 \cup \A_2 \subset \A$  with $\B\subset \A_1$, $\B\subset \A_2$ and $\A_1,\A_2 $ are freely independent with amalgamation over $\B$. Setting $S_1 := \{i \mid x_i \in \A_1 , 1\leq i \leq n\}$ and $S_2 := \{i \mid x_i \in \A_2, 1\leq i \leq n\}$ then
\begin{equation}\label{operator-moment-cumulant-for}
E[x_1 \cdots x_n] 
= \sum_{\pi \in\NC(S_1)} (  \kappa^\B_{\pi} \cup E_{\pi^c} )[x_1, \dots , x_n],
\end{equation}
where  $\pi^c$ denotes the maximal element $\sigma$ in $\NC(S_2)$ satisfying the condition $\pi\cup \sigma \in \NC(S_1,S_2)$, and where $ \kappa^\B_{\pi} \cup E_{\pi^c}$ acts on blocks of $\pi$ as $\kappa^\B$ and on blocks of $\pi^c$ as $E$. For instance, if $\{x_1, x_2 , x_3 \}$ and $\{y_1, y_2\}$ freely independent with amalgamation over $\B$ with $E[y_1] =E[y_2] = 0$ then 
\begin{equation}\label{operator-moment-cumulant-examples}
E[x_1 y_1 x_2] = E[x_1 E[y_1] x_2]= 0
\quad
\text{and} 
\quad
 E[x_1 y_1 x_2 y_2 x_3] = E[x_1 E[y_1 E[x_2] y_2] x_3]  .
 \end{equation}
 For a proof of \eqref{operator-moment-cumulant-for} and a complete presentation of the combinatorial aspect of freeness with amalgamation, we refer to \cite[Chapter 3]{Speicher-98}.

\subsection{Scalar and operator-valued circular and semicircular families}

We say that $\{s_1, \dots , s_d \}$ $ \subset \A$ is a \emph{centered semicircular family} of covariance $C = (c_{k\ell})_{k,\ell =1}^d$ if for any $n \geq 1$ and any $k_1, \dots ,k_n \in [d]$:
\[
\kappa_n[s_{k_1}, \dots ,s_{k_n}] = 0 \text{  if  } n >2, \quad \kappa_1[s_k]= \varphi(s_k)=0  \quad \text{and}  \quad \kappa_2[s_k, s_\ell]=  \varphi(s_ks_\ell) =: c_{k\ell}.
\]
If $C$ is diagonal then $\{s_1, \dots , s_d \}$ is a free semicircular family.

We say that a family $\{S_1 , \dots , S_d\} \subset \A$ is a \emph{centered operator-valued semicircular family} over $\B$, or simply a centered $\B$-valued  semicircular family, with covariance given by the completely positive map 
\[
\eta: \B \rightarrow M_{d}(\B), \qquad b \mapsto \big(\eta_{k,\ell} ( b) \big)_{k,\ell=1}^d
\]
 if  for any  $k,\ell, k_1, \dots , k_n \in [d]$ and  $b, b_1 , \dots , b_{n-1} \in \B$:
\[
\kappa_n^\B[S_{k_1}b_1, \dots, S_{k_{n-1}} b_{n-1} ,S_{k_n}] = 0 \text{  if  } n >2,\]
\[
 \kappa_1^\B[S_k]= E(S_k)=0   \qquad \text{and}  \qquad \kappa^\B_2[S_k b , S_\ell]= E(S_k b S_\ell)=:  \eta_{k,\ell}(b).
\]
If $\eta$ is diagonal, i.e. $\eta_{k,\ell} \equiv 0$ for any $k \neq \ell $, then $S_1 , \dots , S_d$ are free with amalgamation over $\B$.
For more details on how to realize operator-valued semicircular families for a given $\eta: \B \to M_{d}(\B)$ on a suitable Fock space, we refer the readers to \cite{Speicher-98,Shl1999}.

It follows from \cite[Theorem 4.1.12]{Speicher-98} that the $\B$-valued Cauchy transform $\G_S^\B$ of an $\B$-valued semicircular element $S$ with mean zero and variance $\eta$ satisfies the equation
$$1+\eta(\G_S^\B(\Zb)) \G^\B_S(\Zb) = \Zb \G^\B_S(\Zb) \qquad \text{for all $\Zb\in\bH^+(\B)$}.$$
In fact, it was shown in \cite{HeltonFarSpeicher2007} as part of a more general statement that this equation has for each completely positive map $\eta:\B \to \B$ a unique solution $\G: \bH^+(\B) \to \bH^-(\B)$. It is an additional feature of the proof given in \cite{HeltonFarSpeicher2007} that (a slight modification of) this solution $\G$ can be obtained by a fixed point iteration; from this, it can be deduced that $\G$ is in fact a locally bounded Fr\'echet holomorphic function and hence analytic.

We say that a family $\{C_1 , \dots , C_d\} \subset \A$  is a \emph{centered  $\B$-valued circular family} with covariance $(\eta, \widetilde{\eta})$ given by the completely positive maps 
\[
\eta: \B \rightarrow M_{d}(\B), \quad b \mapsto [\eta_{k,\ell} ( b)]_{k,\ell=1}^d
\qquad \text{and} \qquad
\widetilde{\eta}: \B \rightarrow M_{d}(\B), \quad b \mapsto [\widetilde{\eta}_{k,\ell} ( b)]_{k,\ell=1}^d
\]
if for any  $k,\ell, k_1, \dots , k_n \in [d]$, $\epsilon_1, \ldots , \epsilon_n \in \{1, *\}$ and  $b, b_1 , \dots , b_{n-1} \in \B$:
\[
\kappa_n^\B[C_{k_1}^{\epsilon_1} b_1, \dots, C_{k_{n-1}}^{\epsilon_{n-1}} b_{n-1} ,C_{k_n}^{\epsilon_n}] = 0 \text{  if  } n >2,\]
 \[
\kappa_2^\B[C_k b , C_\ell] = \kappa_2^\B[C_k^* b , C_\ell^*]=0, \quad 
\kappa_2^\B[C_k^* b , C_\ell]:= \eta_{k,\ell} (b)
\quad \text{and} \quad
 \kappa_2^\B[C_k b , C_\ell^*]:= \widetilde{\eta}_{k,\ell} (b)\,.
\]
The $\B$-valued $*$-moments of operator-valued circular elements can be computed using Speicher's moment-culumant formula \eqref{operator-moment-cumulant-for}. We refer to \cite{Dyk-05} for a nice review and an extensive study of $\B$-valued circular elements, especially regarding the existence of non-trivial hyperinvariant subspaces.
Of particular interest in this context is the so-called triangular operator $T$, which was introduced by Dykema and Haagerup as a particular instance of their DT-operators in \cite{Dyk-Haa-04-2}. It was studied further in \cite{Sni-03} and \cite{Dyk-Haa-04}, for which purpose it was crucial that $T$ also falls into the class of operator-valued circular elements.

\subsection{The L\'evy and Kolmogorov distances}\label{section:Levy-Kolmogorov}

On the set of all Borel probability measures on the real line $\R$, there are several well-established and useful notions of distance. We recall here the L\'evy and the Kolmogorov distances that will be used in the context of this paper. For a Borel probability measure $\mu$ on $\R$, we denote by $\F_\mu$ the \emph{cumulative distribution function} of $\mu$ defined by $\F_\mu: \R \to [0,1]$, $\F_\mu(t) := \mu((-\infty,t])$.

If $\mu$ and $\nu$ are two Borel probability measures on $\R$, then
\begin{itemize}
 \item the \emph{L\'evy distance} is defined by
$$L(\mu,\nu) := \inf\{\epsilon>0 \mid \forall t\in\R:\ \F_\mu(t-\epsilon) - \epsilon \leq \F_\nu(t) \leq \F_\mu(t+\epsilon) + \epsilon\};$$
 \item the \emph{Kolmogorov distance} is defined by
$$\Delta(\mu,\nu) := \sup_{t\in\R} |\F_\mu(t) - \F_\nu(t)|.$$
\end{itemize}

It is well-known that the L\'evy distance provides a metrization of convergence in distribution.
Furthermore, it can be bounded by
\begin{equation}\label{eq:Levy_bound}
L(\mu,\nu) \leq 2\sqrt{\frac{\epsilon}{\pi}} + \frac{1}{\pi} \int_\R | \Im(\G_\mu(t+i\epsilon)) - \Im(\G_\nu(t+i\epsilon)) |\, \mathrm{d} t
\end{equation}
for any choice of $\epsilon>0$. We provide a proof of \eqref{eq:Levy_bound} in Section \ref{sec:Levy_bound}; a slightly weaker form of \eqref{eq:Levy_bound} was obtained recently in \cite{Salazar2020}. If we suppose that $\nu$ has a cumulative distribution function $\F_\nu$ which is H\"older continuous with exponent $\beta\in(0,1]$ and H\"older constant $C>0$, i.e., $|\F_\nu(t) - \F_\nu(s)| \leq C |t-s|^\beta$ holds for all $s,t\in\R$, then \cite[Lemma 12.18]{BS10} says that the L\'evy and the Kolmogorov distance are related by
\begin{equation}\label{eq:Levy-Kolmogorov}
L(\mu,\nu) \leq \Delta(\mu,\nu) \leq (C+1) L(\mu,\nu)^\beta.
\end{equation}

\section{The operator-valued setting}\label{section:operator-valued}

This section is devoted to proving an analogue of the classical Berry-Esseen theorem in the realm of operator-valued free probability theory, namely for the sum of variables which are freely independent with amalgamation.
The estimates we provide are on both the scalar- and operator-valued levels. We will be working in the framework of the operator-valued $C^\ast$-probability space $(\A, E, \B)$, if nothing else is said. Depending on the situation, we will restrict ourselves to the case of an operator-valued $W^\ast$-probability space $(\A, \varphi, E, \B)$.

Fix $n \in \mathbb{N}$ and consider two families  $x=\{x_j \mid 1\leq j\leq n\}$ and $y=\{y_j \mid 1\leq j\leq n\}$ of selfadjoint elements which are freely independent with amalgamation over $\B$. We will study how close the noncommutative distributions of 
\[
\ux_n= \sum_{j=1}^n x_j
\qquad  \text{and} \qquad
\uy_n=  \sum_{j=1}^n y_j
\]  
are under the mere conditions that the first and second $\B$-valued moments match; i.e. $E[ x_j] = E [ y_j]$ and $E \big[ x_j \,  b \, x_j \big] = E \big[ y_j \,  b \,  y_j \big]$  for all $j=1, \dots , n$  and $b \in \B$. This will be measured in terms of the scalar- and operator-valued Cauchy transforms of $\ux_n$ and $\uy_n$ by deriving quantitative estimates on their difference. As direct consequences of such approximations, we obtain quantitative bounds on the $\B$-free CLT, illustrated in Section \ref{section:free-CLT}, and on the distribution of matrices with $\B$-free entries, illustrated in Section \ref{section:Wigner}. The estimates on the operator-valued Cauchy transforms yield convergence in distribution over $\B$ to some operator-valued semicircular element, while the estimates on the scalar-valued Cauchy transforms yield convergence of the scalar-valued distributions, which can be measured even in terms of the L\'evy distance.
 
\vspace*{0.1cm}
Before stating our main theorem, we introduce some further notation: let 
\begin{align}\label{def:norms}
\|x\| = \max_{1\leq j\leq n} \|x_j\| \qquad \text{and} \qquad \|x\|_{L^r} = \max_{1\leq j\leq n}  \|x_j\|_{L^r(\A,\varphi)} \quad \text{for any  } r\geq1.
\end{align}
Recalling the definition of the maps in \eqref{moment-map}, we encode the relevant information about the second and fourth moments by $\alpha_2$ and $\alpha_4$ which are defined as
\begin{equation}\label{def:second-fourth-moment}
\alpha_2(x) := \max_{1\leq j\leq n} \sup \big\| m_2^{x_j} (b) \big\|  = \max_{1\leq j\leq n} \big\| m_2^{x_j} (\1) \big\| \enspace \text{and} \enspace
\alpha_4(x) := \max_{1\leq j\leq n} \sup \big\| m_4^{x_j}(b^*,\1,b) \big\|,
\end{equation}
where the above supremums are taken over all $b \in \B$ such that $\|b\| \leq 1$;
the equality of the two expressions defining $\alpha_2(x)$ relies on the fact that $\sup_{\|b\| \leq 1} \|m_2^{x_j}(b)\| = \|m^{x_j}_2(\1)\|$ for $j=1,\dots,n$, which follows from \cite[Corollary 2.9]{Paulsen} since the positivity of $E$ guarantees that each $m_2^{x_j}: \B \to \B$ is a positive linear map. Finally, we set 
\[
A_1(x,y):= \sqrt{\alpha_2(x)} \Big(\sqrt{\alpha_4(x) +  \alpha_2(x)^2} + \sqrt{\alpha_4(y) +  \alpha_2(x)^2} \Big)
\quad \text{and} \quad
A_2(x,y):=  \|x\|_{L^3}^3 + \|y\|_{L^3}^3.
\]

\begin{theorem}\label{theo1:Lin-Lin:freeness}
Let $(\A, E, \B)$ be an operator-valued $C^\ast$-probability space. Let $n\in\bN$ and consider two sets $\{x_1, \ldots, x_n\}$ and $\{y_1, \ldots, y_n\}$ each of which consists of selfadjoint elements in $\A$ which are freely independent with amalgamation over $\B$. Suppose that $E[ x_j] = E [ y_j]=0$ and $E \big[ x_j \,  b \, x_j \big] = E \big[ y_j \,  b \,  y_j \big]$ for  any $b \in \B$ and  $j=1, \dots , n$. Then, for any $\Zb \in \bH^+(\B)$,
\begin{align}\label{eq:theo1_Lin-Lin:freeness-1}
\big\| E[G_{\ux_n}(\Zb)] - E[G_{\uy_n}(\Zb)] \big\|
\leq \|\Im(\Zb)^{-1}\|^4 A_1(x,y) n.
\end{align}
In the case of an operator-valued $W^\ast$-probability space $(\A, \varphi, E, \B)$, we have in addition
\begin{equation}\label{eq:theo1_Lin-Lin:freeness-2}
\big| \varphi [G_{\ux_n}(z)] - \varphi [G_{\uy_n}(z)] \big| \leq \frac{1}{\Im (z)^4} A_2(x,y) n
\end{equation}
for any $z\in \mathbb{C}^+$, and furthermore for every $\epsilon>0$,
\begin{align}\label{eq:theo1_Lin-Lin:freeness-3}
\frac{1}{\pi} \int_\R \big| \varphi [G_{\ux_n}(t+i\epsilon)] - \varphi [G_{\uy_n}(t+i\epsilon)] \big|\, \mathrm{d} t
\leq \frac{1}{\epsilon^3} A_1(x,y) n.
\end{align}
\end{theorem}

Note that the above estimates are merely in terms of the third and fourth moments and do not depend on the operator norm. Moreover, they hold over all the upper half plane $\mathbb{H}^+(\B)$ and all the upper complex plane, respectively.

The proof relies on an operator-valued Lindeberg method and is postponed to Section \ref{section:proof-operator-valued}. The Lindeberg method is a replacement method, which goes back to Lindeberg \cite{Lindeberg}. It was first employed in the free probability setting by Kargin \cite{Kargin} to approximate polynomials in noncommutative variables with the aim of generalizing Voiculescu's free CLT. It was then implemented by Deya and Nourdin \cite{Deya-Nourdin} to prove an invariance principle for multilinear homogeneous sums in free elements. This method was later extended by Banna and C\'ebron \cite{Ba-Ce-18} to approximate resolvents of linear functions in free or exchangeable noncommutative variables. Theorem \ref{theo1:Lin-Lin:freeness} above extends the approach in \cite{Ba-Ce-18} to the operator-valued realm, derives estimates on both the scalar- and operator-valued Cauchy transforms, and allows controlling the L\'evy distance. For readers not familiar with the Lindeberg method, what is meant by it will become clear in the proof of Theorem \ref{theo1:Lin-Lin:freeness}.

\changelocaltocdepth{2}
\subsection{The $\B$-Free Central Limit Theorem}\label{section:free-CLT}
We show in this section how our result yields immediately Berry-Esseen bounds on the level of the operator-valued free CLT. Let us first recall the following Berry-Esseen bound in the setting of classical probability. If $\{x_j\}$ is a family of independent and identically distributed random variables with mean $0$ and variance $1$, then the distance between the distributions of $X_n = n^{-1/2} \sum_{j=1}^n x_j$ and a standard Gaussian variable $\mathcal{N}$ is bounded in terms of the Kolmogorov distance as follows:
\[
\Delta(\mu_{X_n}, \mu_{\mathcal{N}}) \leq C \frac{m_3}{\sqrt{n}},
\]
where $C$ is a constant and $m_3$ is the absolute third moment of the $x_j$'s. In the free case, we have a similar bound but now the Gaussian variable $\mathcal{N}$ is replaced by its free analogue, the semicircular element $s$, and the fourth moment $m_4$ of the $x_j$'s appears in addition to the third order moment $m_3$. In fact, we have the following bound: 
\[
\Delta(\mu_{X_n}, \mu_s) \leq c \frac{|m_3|+\sqrt{m_4}}{\sqrt{n}}.
\]
This free analogue of the Berry-Esseen bound was proven by Christyakov and G\"otze in \cite{Chi-Got-08}; in fact, they give a bound for $\Delta(\mu^{\boxplus n}_n,\mu_s)$, where $\mu_n$ is a suitable scaling of an arbitrary (not necessarily compactly supported) Borel probability measure $\mu$ on $\R$ and $\mu_n^{\boxplus n}$ denotes its $n$-fold free additive convolution. This remarkable result, when applied to $\mu$ being the distribution of the $x_j$'s, yields the bound for $\Delta(\mu_{X_n},\mu_s)$ as stated above; an earlier version of the Berry-Esseen bound in this case was obtained by Kargin in \cite{Kargin-Berry-Esseen}.
We extend the above bounds to the operator-valued setting.

The operator-valued free CLT is due to Voiculescu \cite[Theorem 8.4]{Voi-95} and states in the case of an operator-valued $C^\ast$-probability space $(\A,E,\B)$ the following: let $x = (x_n)_{n\in\bN}$ be a sequence of selfadjoint elements in $\A$ that are free with amalgamation over $\B$ and are such that $E[x_j] = 0$. Assume that $\sup_{n\in\bN} \|m_k^{x_n}\| < \infty$ for all $k\in\bN_0$ and that there exists a linear map $\eta: \B \rightarrow \B$ such that $\lim_{n\rightarrow \infty} n^{-1} \sum_{i=1}^n E [x_i b x_i] = \eta (b)$ for any $b \in \B$. Then
\[
X_n = \frac{1}{\sqrt{n}} \sum_{j=1}^n x_j \xrightarrow[n\rightarrow \infty]{\B-d} S,
\]  
where $S$ is an operator-valued semicircular element over $\B$ with variance $\eta$. Under the assumption that $\lim_{n\rightarrow \infty} \sup_{\|b\|\leq 1} \|n^{-1} \sum_{i=1}^n E [x_i b x_i] - \eta (b)\| = 0$, Theorem \ref{theo1:Lin-Lin:freeness} (in combination with Theorem \ref{theo:opval_semicirculars_comparison}) proves the above convergence analytically and strengthens it to norm-convergence of the distribution over $\B$. In the case of identical second moments (see Remark \ref{rem:CLT_convergence_in_distribution}), the convergence is quantified by providing Berry-Esseen type bounds for the scalar- and operator-valued Cauchy transforms. It also yields quantitative estimates in terms of the L\'evy distance. More details on previous works in this direction can be found in Section \ref{section:MCLT} on the multivariate setting following Remark \ref{rem:convergence_polynomial}.

With this aim, we let $S_n$ be a $\B$-valued centered semicircular element whose variance is given by the completely positive map 
\[
\eta_n : \B \rightarrow \B, \qquad b \mapsto \eta_n(b) = \frac{1}{n} \sum_{j=1}^n E[x_j b x_j].
\]
Finally, setting
\[
A_1(x) = \sqrt{\alpha_2(x)} \Big(\sqrt{\alpha_4(x) + \alpha_2(x)^2} + \sqrt{3}\alpha_2(x) \Big)\enspace \text{and} \enspace A_2(x) = \|x\|_{L^3}^3 +  \sqrt{2\alpha_2(x)} \|x\|_{L^2}^2,
\]
our result then reads as follows:

\begin{theorem}\label{theo:Op-freeCLT}
Let $x = \{x_1, \ldots, x_n\}$ be a family of selfadjoint elements in $\A$ that are free with amalgamation over $\B$ and are such that $E[x_j] =0$. Then for any $\Zb \in \bH^+(\B)$,
\begin{equation}\label{eq:Op-freeCLT}
\|E [G_{X_n}(\Zb)] - E [G_{S_n}(\Zb)] \| \leq \frac{1}{\sqrt{n}} \| \Im (\Zb)^{-1}\|^4 A_1(x)
\end{equation}
and for any $z\in \mathbb{C}^+$, 
\[
|\varphi [G_{X_n}(z)] - \varphi [G_{\S_n}(z)]| \leq \frac{1}{\sqrt{n}} \frac{1}{\Im (z)^4} A_2(x).
\]
Furthermore, there is a universal positive constant $c<1.672$ such that
\[
L(\mu_{X_n}, \mu_{\S_n}) \leq c A_1(x)^{1/7} n^{-1/14}.
\]
\end{theorem}

The proof of Theorem \ref{theo:Op-freeCLT} is postponed to the end of this section.

\begin{remark}\label{rem:CLT_convergence_in_distribution}
Let us point out that if $x_1,\dots,x_n$ are such that $\eta := m_2^{x_1} = \dots = m_2^{x_n}$, then $S_n$ is an operator-valued semicircular element with zero mean and variance $\eta$ which is independent of $n$. We denote it, in this case, by $S$.

We apply Theorem \ref{theo:Op-freeCLT} to the family $\1_k \otimes x = \{\1_k \otimes x_1, \dots, \1_k \otimes x_n\}$ in the operator-valued $C^\ast$-probability space $(M_k(\A),\id_k \otimes E,M_k(\B))$ and note that $A_1(\1_k \otimes x) \leq k \overline{A}_1(x)$, where
$$\overline{A}_1(x) :=  \sqrt{\alpha_2(x)} \Big(\sqrt{\overline{\alpha}_4(x)+  \alpha_2(x)^2} + \sqrt{3} \alpha_2(x) \Big) \qquad \text{and} \qquad \overline{\alpha}_4(x) := \max_{1 \leq j \leq n} \|m_4^{x_j}\|,$$
which follows by Lemma \ref{lem:tensor_moments_diagonal}. Then from \eqref{eq:Op-freeCLT}, we conclude that for any $\Zb \in \bH^+(M_k(\B))$,
$$\|\G^{M_k(\B)}_{\1_k \otimes X_n}(\Zb) - \G^{M_k(\B)}_{\1_k \otimes S_n}(\Zb)\| \leq \frac{k}{\sqrt{n}} \| \Im (\Zb)^{-1}\|^4 \overline{A}_1(x).$$

If we suppose in addition that $x$ extends to an infinite family of selfadjoint elements which are freely independent with amalgamation over $\B$ and satisfy $m_2^{x_j} = \eta$ for all $j\in\bN$, then $\sup_{n\in\bN} \|X_n\| < \infty$ thanks to \cite[Proposition 7.1]{Junge} and by the above, for every $\epsilon>0$,
$$\lim_{n\to \infty} \sup_{\Zb \in \bH^+(M_k(\B)) \colon \Im(\Zb) \geq \epsilon 1} \|\G^{M_k(\B)}_{\1_k \otimes X_n}(\Zb) - \G^{M_k(\B)}_{\1_k \otimes S_n}(\Zb)\| = 0$$
provided that $\sup_{n\in\bN} \alpha_4(\{x_1,\dots,x_n\}) < \infty$. From Remark \ref{rem:Cauchy-convergence}, we get $\lim_{n\to \infty} \|m_k^{X_n} - m_k^S\| = 0$ for every $k \in \bN_0$ and in particular $X_n \stackrel{\B-d}{\longrightarrow} S$ as $n \to \infty$.
\end{remark}

\begin{remark}
In the situation of the previous Remark \ref{rem:CLT_convergence_in_distribution}, we suppose now that the analytic distribution $\mu_S$ of $S$ has a H\"older continuous cumulative distribution function, say with H\"older exponent $\beta\in(0,1]$ and H\"older constant $C>0$. Then Theorem \ref{theo:Op-freeCLT} combined with \eqref{eq:Levy-Kolmogorov} implies that $\Delta(\mu_{X_n},\mu_S) \leq C_1 n^{-\beta/14}$ for some constant $C_1>0$ (independent of $n$).
On the other hand, Theorem 5.3 in \cite{Banna-Mai-18}, whose boundedness condition is satisfied thanks to \cite[Proposition 7.1]{Junge}, yields in this particular situation the better bound $\Delta(\mu_{X_n},\mu_S) \leq C_2 n^{-\frac{\beta}{2\beta+8}}$ for some constant $C_2>0$. The strength of Theorem \ref{theo:Op-freeCLT} lies more in its universality, namely that it applies also when $\F_{\mu_S}$ is not (known to be) H\"older continuous or even when $\mu_S$ has atoms.
\end{remark}

We continue with the following result by which we recover from Theorem \ref{theo:Op-freeCLT} the operator-valued CLT of Voiculescu \cite[Theorem 8.4]{Voi-95} in the setting of operator-valued $C^\ast$-probability spaces and under the assumption of uniform convergence of the $\eta_n$'s and not only in the case of identical second moments which was discussed in Remark \ref{rem:CLT_convergence_in_distribution}; notably, we get in this case the stronger conclusion that the $\B$-valued distribution of $X_n$ is norm-convergent in the sense of \cite{BPV2012} meaning that \eqref{eq:norm-convergence_moments} holds. Beyond this particular application, the result might also be of independent interest; for instance, in comparison with the results of \cite[Proposition 2.3]{AEK2020}.

\begin{theorem}\label{theo:opval_semicirculars_comparison}
Let $(\A,E,\B)$ be an operator-valued $C^\ast$-probability space and consider two operator-valued semicircular elements $S_0,S_1$ with respective covariance maps $\eta_0,\eta_1: \B \to \B$. Then, for every $k\in\bN$ and each $\Zb\in \bH^+(M_k(\B))$, we have that
\begin{equation}\label{eq:opval_semicirculars_comparison_Cauchy}
\|\G^{M_k(\B)}_{\1_k \otimes S_1}(\Zb) - \G^{M_k(\B)}_{\1_k \otimes S_0}(\Zb)\| \leq k \|\Im(\Zb)^{-1}\|^3 \|\eta_1-\eta_0\|.
\end{equation}
Moreover, if $(\A,\varphi,E,\B)$ is an operator-valued $W^\ast$-probability space, then the scalar-valued Cauchy transforms of $S_1$ and $S_0$ satisfy
\begin{equation}\label{eq:opval_semicirculars_comparison_integral}
\frac{1}{\pi} \int_\R |\G_{S_1}(t+i\epsilon) - \G_{S_0}(t+i\epsilon)|\, \mathrm{d} t \leq \frac{1}{\epsilon^2} \|\eta_1 - \eta_0\|
\end{equation}
for each $\epsilon>0$ and, with the universal positive constant
$c = 5(\frac{1}{4\pi})^{2/5} <  1.817$, we have that
\begin{equation}\label{eq:opval_semicirculars_comparison_Levy}
L(\mu_{S_1},\mu_{S_0}) \leq c \|\eta_1-\eta_0\|^{1/5}.
\end{equation}
\end{theorem}

Like for Theorem \ref{theo1:Lin-Lin:freeness}, the proof of Theorem \ref{theo:opval_semicirculars_comparison} relies on the noncommutative Lindeberg method and is thus postponed to Section \ref{section:proof-operator-valued}.

In view of the remarkable fact that the Berry-Esseen bound provided by Theorem \ref{theo:Op-freeCLT} depends only on a finite number of moments of the selfadjoint operators $x$, it is natural to expect that this result allows some extension to the framework of unbounded operators; see \cite{BV1993,Wil2017}, for instance.

Let us consider an operator-valued $W^\ast$-probability space $(\A,\varphi,E,\B)$. By $\tilde{\A}$, we will denote the $\ast$-algebra of all closed and densely defined linear operators affiliated with $\A$; elements of $\tilde{\A}$ are considered as unbounded noncommutative random variables. We define $\tilde{\A}_\sa := \{x \in \tilde{A} \mid x^\ast = x\}$.

For every $x \in \tilde{\A}_\sa$, we may define the \emph{analytic distribution $\mu_x$} of $x$ as the unique Borel probability measure on $\R$ such that $\varphi(f(x)) = \int_\R f(t)\, \mathrm{d}\mu_x(t)$ for all $f\in B_b(\R)$, where $B_b(\R)$ denotes the algebra of all bounded measurable functions $f: \R \to \C$; this extends the definition of Section \ref{subsec:Cauchy_transform}.

Let $(\bcX_i)_{i\in I}$ be a family of subsets of $\tilde{\A}_\sa$. For $i\in I$, we define $\A_i$ as the subalgebra of $\A$ which is generated by $\{f(x) \mid x\in \bcX_i, f\in B_b(\R)\} \cup \B$, where $f(x) \in \A$ is obtained by the Borel functional calculus. The family $(\bcX_i)_{i\in I}$ is said to be \emph{free with amalgamation over $\B$} if the associated family $(\A_i)_{i\in I}$ of subalgebras of $\A$ is free with amalgamation over $\B$ in the sense of Section \ref{subsec:freeness}.
In particular, a family $(x_i)_{i\in I}$ of elements of $\tilde{\A}_\sa$ is said to be \emph{free with amalgamation over $\B$} if the family $(\{x_i\})_{i\in I}$ is free with amalgamation over $\B$ in the aforementioned sense.

\begin{proposition}\label{prop:Op-freeCLT_unbounded}
Let $x=\{x_1,\dots,x_n\}$ be a finite family in $\tilde{\A}_\sa$ which is free with amalgamation over $\B$ and which has the following property: for $i=1,\dots,n$, there exists a sequence $(p_i^k)_{k\in\bN}$ of projections in $\A_i$ such that $x_i^k := p_i^k x_i p_i^k \in \A_i$ and such that the following conditions are satisfied:
\begin{enumerate}
 \item\label{it:CLT-unbounded-i} We have $E[x_i^k] = 0$ for all $k\in\bN$ and $i=1,\dots,n$.
 \item\label{it:CLT-unbounded-ii} There exists an operator-valued semicircular element $S_n$ in $(\A,E,\B)$ such that $$\lim_{k\to \infty} \sup_{b\in\B\colon \|b\|\leq 1} \bigg\|\frac{1}{n} \sum^n_{j=1} E[x_j^k b x_j^k] - E[S_n b S_n] \bigg\| = 0.$$
 \item\label{it:CLT-unbounded-iii} For the family $x^k := \{x_1^k,\dots,x_n^k\}$, we have that
 $$\alpha_2(x) := \sup_{k\in\bN} \alpha_2(x^k) < \infty \qquad\text{and}\qquad \alpha_4(x) := \sup_{k\in\bN} \alpha_4(x^k) < \infty.$$
 \item\label{it:CLT-unbounded-iv} For each $i=1,\dots,n$, the sequence $(\varphi(p_i^k))_{k\in\bN}$ converges to $1$ as $k\to \infty$.
\end{enumerate}
Consider $X_n := \frac{1}{\sqrt{n}} \sum^n_{j=1} x_i$ in $\tilde{\A}_\sa$. Then $L(\mu_{X_n}, \mu_{S_n}) \leq c A_1(x)^{1/7} n^{-1/14}$ with the constant $c>0$ as in Theorem \ref{theo:Op-freeCLT} and $A_1(x) := \sqrt{\alpha_2(x)} \big(\sqrt{\alpha_4(x) + \alpha_2(x)^2} + \sqrt{3}\alpha_2(x) \big)$ as before.
\end{proposition}

\begin{proof}
For each $k\in\bN$, we put $X_n^k := \frac{1}{\sqrt{n}} \sum^n_{j=1} x_j^k$. Since $x^k$ is a family of selfadjoint elements in $\A$ that are free with amalgamation over $\B$ and centered by \ref{it:CLT-unbounded-i}, Theorem \ref{theo:Op-freeCLT} in combination with \ref{it:CLT-unbounded-iii} yields that $L(\mu_{X_n^k},\mu_{S_n^k}) \leq c A_1(x^k)^{1/7} n^{-1/14} \leq c A_1(x)^{1/7} n^{-1/14}$, where $S_n^k$ is an operator-valued semicircular element with the covariance map $\eta_n^k: \B \to \B$ given by $\eta_n^k(b) := \frac{1}{n} \sum^n_{j=1} E[x_j^k b x_j^k]$.
Now, let us take an arbitrary $\epsilon>0$. Notice that $\lim_{k\to \infty} \|\eta_n^k - \eta_n\| = 0$ by assumption \ref{it:CLT-unbounded-ii}, where $\eta_n: \B \to \B, b \mapsto E[S_n b S_n]$ is the covariance of $S_n$; hence, $\lim_{k\to \infty} L(\mu_{S_n^k},\mu_{S_n}) = 0$ thanks to Theorem \ref{theo:opval_semicirculars_comparison}. Furthermore $\lim_{k\to \infty} \varphi(p_i^k) = 1$ for $i=1,\dots,n$ by assumption \ref{it:CLT-unbounded-iv}. Thus, we may choose $k_0\in \bN$ such that, for all $k\geq k_0$, $L(\mu_{S_n^k},\mu_{S_n}) \leq \epsilon$ and $\varphi(p_i^k) \geq 1 - \epsilon$ for $i=1,\dots,n$.
Consider $p^k := p_1^k \wedge \dots \wedge p_n^k \in \A$ like in \cite[Lemma 4.10]{BV1993}. Then $\varphi(p^k) \geq 1-n\epsilon$ and $p^k X_n p^k = p^k X_n^k p^k$, so that $\Delta(\mu_{X_n^k},\mu_{X_n}) \leq n\epsilon$ by \cite[Theorem 3.9 (i)]{BV1993}. In summary, we get that
$$L(\mu_{X_n},\mu_{S_n}) \leq L(\mu_{X_n},\mu_{X_n^k}) + L(\mu_{X_n^k},\mu_{S_n^k}) + L(\mu_{S_n^k},\mu_{S_n}) \leq (n+1) \epsilon + c A_1(x)^{1/7} n^{-1/14}.$$
Letting $\epsilon \searrow 0$, the latter inequality yields the assertion. 
\end{proof}

\begin{proof}[Proof of Theorem \ref{theo:Op-freeCLT}]  
Let $y=\{y_1 , \ldots , y_n\}$ be a family of $\B$-valued semicircular elements that are free over $\B$ and are such that $E[y_j] = 0$ and $E[x_j b x_j] = E[y_j b y_j]$ for any $j \in [n]$ and $b \in \B$. Then we have by the moment-cumulant formula,
\[
E[y_j b^* y_j^2 b y_j] = E[y_j b^* y_j] E[y_j b y_j] + E[y_j b^* E[y_j^2] b y_j]
\] 
and hence $m_4^{y_j}(b^*,1,b) = m_2^{x_j}(b^*) m_2^{x_j}(b) + m_2^{x_j}( b^* m_2^{x_j}(1) b )$; using \eqref{def:second-fourth-moment}, we deduce from the latter that $\alpha_4(y) \leq 2 \alpha_2(x)^2$.
On the other hand, we have that
\[
\|y_j\|_{L^3}^3 \leq \varphi(y_j^4)^{1/2} \|x_j\|_{L^2} = \sqrt{2} \varphi(E[x_j^2]^2)^{1/2} \|x_j\|_{L^2} \leq \sqrt{2} \|E[x_j^2]\|^{1/2} \|x_j\|_{L^2}^2,
\]
and therefore $\|y\|_{L^3}^3 \leq  \sqrt{2\alpha_2(x)} \|x\|_{L^2}^2$. Now set $S_n = \frac{1}{\sqrt{n}} \sum_{j=1}^n y_j$ and note that $S_n$ is a $\B$-valued centered semicircular element whose variance is given by the completely positive map $\eta_n$. The proof of the first part of the theorem then follows by applying Theorem \ref{theo1:Lin-Lin:freeness} to $X_n$ and $S_n$ and taking into account the normalization.

Moreover, with the help of \eqref{eq:Levy_bound}, we infer from the estimate \eqref{eq:theo1_Lin-Lin:freeness-3} in Theorem \ref{theo1:Lin-Lin:freeness} the following bound on the L\'evy distance between the analytic distributions of $X_n$ and $S_n$
\[
L(\mu_{X_n},\mu_{S_n}) \leq 2\sqrt{\frac{\epsilon}{\pi}} + \frac{A_1(x)}{\epsilon^3\sqrt{n}}.
\]
Optimizing over $\epsilon\in (0,\infty)$, we finally get that
\[
L(\mu_{X_n},\mu_{S_n}) \leq c A_1(x)^{1/7} n^{-1/14}
\]
for a universal positive constant $c = 7 (\frac{1}{9\pi})^{3/7} < 1.672$.
\end{proof}

\subsection{Wigner matrices with variance profile}

In this section, we present another consequence of Theorem \ref{theo1:Lin-Lin:freeness} on operator-valued Wigner matrices. Let $N \in \mathbb{N}$ and consider the $N \times N$ operator-valued matrix $A_N \in M_N(\A)= M_N(\C)\otimes \A$ defined by 
\[
A_N = \sum_{1\leq j \leq i \leq N} \big( e_{ij} \otimes a_{ij} + e_{ij}^* \otimes a_{ij}^* \big)
\]
where we set $e_{ii} = \frac{1}{2 \sqrt{N}} E_{ii}$ and $e_{ij}= \frac{1}{\sqrt{N}} E_{ij}$ for  $j<i$ with $(E_{ij})_{1\leq i,j \leq N}$ denoting the standard matrix units in $M_N(\C)$.
We say that $A_N$ is an \emph{operator-valued Wigner matrix} whenever $a= \{a_{ij} \mid 1\leq j \leq i \leq N\} $ is a family of elements that are freely independent with amalgamation over $\B$ such that $a_{ii}= a_{ii}^*$,  $E[a_{ij}]=0$ and  $E[a_{ij} b a_{ij}]=0$ for all  $b \in \B$ and $1\leq j < i \leq N$. Note that the entries of $A_N$ are assumed to be, up to symmetry, free with amalgamation over $\B$ but do not need to be identically distributed. Applying Theorem \ref{theo1:Lin-Lin:freeness}, in the framework of the operator-valued $W^\ast$-probability space $(M_N(\A), \tr_N \otimes  \varphi, \id_N \otimes E, M_N(\B))$ which is associated with $(\A,\varphi,E,\B)$, we approximate the distribution of $A_N$ by that of an operator-valued semicircular element.

\begin{theorem}\label{theo:op-matrices}
Let $A_N$ be an operator-valued Wigner matrix. Then for any $z\in \mathbb{C}^+$, 
\[
|(\tr_N \otimes \varphi ) [G_{A_N}(z)] - (\tr_N \otimes \varphi )  [G_{S_N}(z)]| \leq \frac{8}{\Im (z)^4 \sqrt{N}} \Big(\|a\|_{L^3}^3 + \sqrt{2\|E[a^*a]\|} \|a\|_{L^2}^2\Big)
\]
where $S_N$ is a centered operator-valued semicircular element over $\D_N^\B:= \D_N \otimes\B$, the subalgebra of diagonal matrices in $M_N(\C) \otimes \B$, and whose covariance is given by the map
\[
\eta_N: \D_N^\B \rightarrow \D_N^\B, \qquad  D \mapsto \eta_N (D), 
\]
where for any $D=(d_{ij})_{i,j=1}^N \in \D_N^\B$, 
\[
\big(\eta_N (D) \big)_{i,j} = \delta_{i,j} \frac{1}{N} \sum_{r=1 }^i E [a_{ir} d_{rr} a^*_{ir} ] + \delta_{i,j} \frac{1}{N} \sum_{r=i+1 }^N E[a^*_{ri} d_{rr} a_{ri}].
\]
\end{theorem}

This theorem extends the result in \cite[Theorem 3.4]{Ba-Ce-18} to the operator-valued setting by relaxing the freeness assumption on the entries to freeness with amalgamation.

\begin{proof}
Let $c=\{c_{ij} \mid 1\leq j \leq i \leq N\}$ be a family of \emph{free elements over $\B$} such that
\begin{itemize}
    \item for any $1  \leq i \leq N$, $c_{ii}$ is a $\B$-valued semicircular element with $E[c_{ii}]=0$ and $E[c_{ii} b c_{ii}]= E[a_{ii} b a_{ii}]$ for any $b \in \B$,
    \item for any $1\leq j < i \leq N$,  $c_{ij}$ is a $\B$-valued circular element with $ E[c_{ij}]=0$ and  $E[c_{ij} b c_{ij}^*] = E[a_{ij} b a_{ij}^*] $ for any $b \in \B$.
\end{itemize}
Without of generality, $c$ can be assumed to be free from $a$ over $\B$. Now set $S_N$ to be the $N\times N$ operator-valued Wigner matrix whose entries are given by the family $c$, i.e.
 \[
 S_N = \sum_{1\leq j \leq i \leq N} \big(e_{ij} \otimes c_{ij}+ e_{ij}^* \otimes c_{ij}^* \big).
 \]
 Finally, set $x_{ij} = e_{ij} \otimes a_{ij} + e_{ij}^* \otimes a_{ij}^* $ and $y_{ij} = e_{ij} \otimes c_{ij} + e_{ij}^* \otimes c_{ij}^*$ for all $1\leq j\leq i\leq N$, where we recall that $e_{ii} = \frac{1}{2 \sqrt{N}} E_{ii}$ and $e_{ij}= \frac{1}{\sqrt{N}} E_{ij}$ for  $j<i$ with $(E_{ij})_{1\leq i,j \leq N}$ denoting the standard matrix units in $M_N(\C)$. As freeness with amalgamation over $\B$ is preserved when lifted to matrices then $\{x_{ij} , y_{ij} \mid 1\leq j \leq i \leq N\}$ are free over $M_N(\B)$ with respect to $\id_N \otimes E$. Moreover, for any $b \in M_N(\B)$,
  \[
(\id_N \otimes E)[x_{ij}] = (\id_N \otimes E)[y_{ij}]=0
\quad \text{and} \quad
(\id_N \otimes E)\big[ x_{ij} \,  b \, x_{ij}\big] = (\id_N \otimes E) \big[ y_{ij} \,  b \,  y_{ij} \big]. 
\]
Setting $n=N(N+1)/2$, then the desired estimate follows from \eqref{eq:theo1_Lin-Lin:freeness-2} in Theorem \ref{theo1:Lin-Lin:freeness} by noting that 
\[
\|x\|_{L^3}^3 \leq \frac{8}{N \sqrt{N}} \|a\|_{L^3}^3 \max_{1\leq j \leq i\leq N} \| E_{ij} \|_{L^3}^3 = \frac{8}{N^2 \sqrt{N}} \|a\|_{L^3}^3
\]
and, with arguments similar to the proof of Theorem \ref{theo:Op-freeCLT}, that
\[
\|y\|_{L^3}^3 \leq \frac{8}{N^2 \sqrt{N}} \|c\|_{L^3}^3 \leq \frac{8\sqrt{2}}{N^2 \sqrt{N}} \sqrt{\|E[a^*a]\|} \|a\|_{L^2}^2.
\]
Finally, we show that $S_N$ is an operator-valued element over $\D^\B_N$. By an operator-valued variation of Proposition 13 in \cite[Chapter 9]{Mingo-Speicher}, we  compute the $M_N(\B)$-valued cumulants of $S_N$ in terms of the $\B$-cumulants of its entries. Indeed, as the $c_{ij}$'s are $\B$-valued circular elements, we then get for $m \in \mathbb{N}$ and for any $i,j \in [N]$ and $B, B_1 , \dots , B_{m-1} \in M_N(\B)$
\[\kappa^{M_N(\B)}_m [S_NB_1, \dots , S_N B_{m-1} , S_N]  =0
\quad \text{for  } m>2,
\]
\[
 \kappa_1^{M_N(\B)}[S_N]= 0,   \quad  \text{and} \quad \kappa^{M_N(\B)}_2[S_N B , S_N] = (\id_N \otimes E) \big[S_N B S_N \big] = \eta_N (B),
\]
with $\eta_N (B)$ as described above. To end the proof, it suffices to remark that the subalgebra of diagonal matrices $\D_N^\B \subset M_N(\C) \otimes \B$ is closed under the covariance map $\eta_N$.
\end{proof}

\section{The multivariate operator-valued setting}\label{section:Lin,-Lin}

In this section, we extend our approach to the multivariate setting to study noncommutative joint distributions. Like in Section \ref{section:operator-valued}, we will mostly be working in an operator-valued $C^\ast$-probability space $(\A,E,\B)$ but we specify our considerations to an operator-valued $W^\ast$-probability space $(\A,\varphi,E,\B)$ when necessary.
Fix $n, d \in \mathbb{N}$ and consider a family $x=\{x_j^{(k)} \mid 1\leq j \leq n, \, 1\leq k\leq d \}$ of  elements in $\A $. Our aim is to study the noncommutative joint distribution of the correlated tuple
\[
\ux_n:=\Big( \sum_{j=1}^n  x_j^{(1)},  \sum_{j=1}^n x_j^{(1)*} , \dots , \sum_{j=1}^n  x_j^{(d)} , \sum_{j=1}^n  x_j^{(d)*} \Big)
\]
under the assumption that the summands of each component are freely independent with amalgamation over $\B$. More precisely, we assume that $\A_{x_1}, \dots, \A_{x_n}$ are free with amalgamation over $\B$, with $\A_{x_j} = \B\langle x_j^{(1)}, x_j^{(1)*},\dots, x_j^{(d)}, x_j^{(d)*}\rangle$ being the algebra generated by $\{ x_j^{(1)}, x_j^{(1)*}, \dots , x_j^{(d)}, x_j^{(d)*} \}$ and $\B$. In order to obtain a measure theoretic description of the joint distribution of $\ux_n$, we consider evaluations of ``noncommutative test functions'' $f$. Each such evaluation $f(\ux_n)$ produces a single noncommutative random variable whose distribution can be studied analytically in terms of Cauchy transforms. The aim of this section is to describe how close the distribution of $f(\ux_n)$ is to the distribution of $f(\uy_n)$, when $\uy_n$ is a tuple of correlated sums constructed in the same way as $\ux_n$ and out of a family $y$ consisting of elements with matching first and second moments. More precisely, we consider a family $y=\{y_j^{(k)} \mid 1\leq j \leq n , \, 1\leq k\leq d \}$ of elements in $\A$ satisfying the following: 

\begin{assumption}\label{A:general}
The family $y$ is free from $x$ with amalgamation over $\B$ and is such that 
 \begin{itemize}
\item   $\A_{y_1}, \dots, \A_{y_n}$  are free with amalgamation over $\B$
\end{itemize}
 and for any $j=1, \dots , n$ and $k, \ell= 1, \dots , d$, 
\begin{itemize}
\item $E[ y_j^{(k)}] = E [ x_j^{(k)}]=0$,
\item $E \big[ y_j^{(k), \epsilon_1} \,  b \, y_j^{(\ell), \epsilon_2}  \big] = E \big[ x_j^{(k), \epsilon_1}\,  b \,  x_j^{(\ell),\epsilon_2}  \big]$ for all $b \in \B$ and $\epsilon_1, \epsilon_2 \in \{1, *\}$.
\end{itemize}
\end{assumption}

With the suitable choice of $y$, one can then describe the distribution of $f(\uy_n)$ and, hence, give an analytic approximation of the distribution of $f(\ux_n)$. However, as the evaluation $f(\ux_n)$ does not contain all the information on the noncommutative joint distribution of $\ux_n$, the guiding idea is then: the larger the considered class of test functions is, the more information one gains about the underlying multivariate noncommutative distribution. 
We consider two classes of noncommutative test functions (see Appendix \ref{section:linearizations} for more details):

\begin{itemize}
\item linear matrix pencils: let $\C\langle x_1, x_1^*, \dots, x_d, x_d^*\rangle$ denote the $\ast$-algebra of all noncommutative polynomials in the non-commuting indeterminates $x=(x_1,x_1^*,\dots,x_d,x_d^*)$; for any selfadjoint linear matrix pencil over $\C\langle x_1, x_1^*, \dots, x_d, x_d^*\rangle$, i.e., an element in $M_m(\C) \otimes \C\langle x_1, x_1^*, \dots,x_d, x_d^*\rangle$ which is of the form
\begin{equation}\label{eq:linear_matrix_pencil_sa}
g = Q_0 \otimes 1 + \sum^d_{k=1} (Q_k \otimes x_k + Q_k^\ast \otimes x_k^\ast)
\end{equation}
for matrix coefficients $Q_0,Q_1,\dots,Q_d$ in $M_m(\C)$ with $Q_0$ being selfadjoint, we consider
\[
\qquad \quad \; \;
 g(\ux_n) := Q_0 \otimes 1 +
 Q_1 \otimes \sum_{j=1}^n x_j^{(1)} +  Q_1^* \otimes \sum_{j=1}^n x_j^{(1)*} + \cdots + Q_d \otimes \sum_{j=1}^n x_j^{(d)} +  Q_d^* \otimes \sum_{j=1}^n x_j^{(d)*}.
\]
\item noncommutative polynomials: denote by $\B\langle x_1,x_1^*,\dots, x_d,x_d^*\rangle$ the $\ast$-algebra of $\B$-valued polynomials in the non-commuting indeterminates $x=(x_1,x_1^*,\dots,x_d,x_d^*)$, i.e., the $\C$-linear span of all $\B$-valued monomials of the form
$$b_0 x_{i_1}^{\epsilon_1} b_1 x_{i_2}^{\epsilon_2} b_2 \cdots b_{k-1} x_{i_k}^{\epsilon_{k}} b_k$$
for integers $k\geq 0$, indices $1\leq i_1,\dots,i_k \leq d$, elements $b_0,b_1,\dots,b_k\in\B$ and $\epsilon_1, \ldots, \epsilon_k \in \{1,*\}$; for any noncommutative polynomial $p \in  \B\langle x_1, x_1^*, \dots,x_d, x_d^*\rangle $ which is selfadjoint and of degree $\geq 1$, we consider 
\[
p(\ux_n):= p \Big( \sum_{j=1}^n  x_j^{(1)},  \sum_{j=1}^n x_j^{(1)*} , \dots , \sum_{j=1}^n  x_j^{(d)} , \sum_{j=1}^n  x_j^{(d)*} \Big). 
\]
\end{itemize}

In fact, Remark \ref{rem:Cauchy-convergence} ensures that each of those classes is rich enough to determine convergence in $\ast$-distribution over $\B$. This will be explained below in Remark \ref{rem:convergence_linear-pencil} and Remark \ref{rem:convergence_polynomial}.

In order to simplify the presentation of our results on noncommutative polynomials, we chose to state and prove them for polynomials in $\B \langle  x_1, x_1^*, \ldots, x_d, x_d^* \rangle$. However, our methods and results can be extended easily to polynomials in $\B \langle \tilde{x}_1, \ldots , \tilde{x}_{d_1}, x_1, x_1^*, \ldots, x_{d_2}, x_{d_2}^* \rangle$, namely polynomials in $d_1$ selfadjoint indeterminates $(\tilde{x}_1, \ldots , \tilde{x}_{d_1}) $ and $2d_2$ indeterminates $(x_1, x_1^*, \ldots, x_{d_2}, x_{d_2}^*)$. This covers in particular the setting in \cite{EKY18}.

We fix all along this section a linear pencil $g$ and a noncommutative polynomial $p$ defined as above. Before stating our main theorems, we need to introduce more notation that will be used in the sequel. We denote by $C_{g}$ and $C_{p}$ positive constants that depend only on $g$ and $p$ respectively and that can change from one line to another. 
We also put
\[
\|x\|= \max_{1\leq j\leq n } \max_{1\leq k \leq d } \|x_j^{(k)}\| \qquad \text{and} \qquad \|x\|_{L^r} = \max_{1\leq j\leq n } \max_{1\leq k \leq d } \|x_j^{(k)}\|_{L^r} \quad \text{for any  } r\geq1,
\]
and moreover
\[
\|E[x x^\ast]\| := \max_{1\leq j \leq n} \max_{1 \leq k \leq d} \big\| E [x_j^{(k)}x_j^{(k)*}]\big\| \qquad \text{and} \qquad \|E[x x^\ast]\| := \max_{1\leq j \leq n} \max_{1 \leq k \leq d} \big\| E [x_j^{(k)*} x_j^{(k)}]\big\|.
\]
Recalling the definition of the maps in \eqref{moment-map}, we set for $\epsilon_1, \epsilon_2, \epsilon_3, \epsilon_4 \in \{1, *\}$ 
\begin{align*}
\| \mathfrak{m}_2^{x^{\epsilon_1},x^{\epsilon_2}} \| &= 
 \max_{1\leq j\leq n } \max_{1\leq \ell_1, \ell_2 \leq d } \sup  \Big\| m_2^{x_j^{(\ell_1),\epsilon_1}, x_j^{(\ell_2), \epsilon_2}} (b ) \Big\| 
 \\ \|\mathfrak{m}_4^{x^{\epsilon_1},x^{\epsilon_2},x^{\epsilon_3},x^{\epsilon_4}}\| &=\max_{1\leq j\leq n } \max_{1\leq \ell_1, \ell_2, \ell_3 , \ell_4 \leq d }   \sup
 \Big\|m_4^{x_j^{(\ell_1),\epsilon_1}, x_j^{(\ell_2),\epsilon_2},x_j^{(\ell_3),\epsilon_3}, x_j^{(\ell_4),\epsilon_4}}(b^*,\1,b) \Big\| 
\end{align*}
where the above supremums are taken over all $b \in \B$ such that $\|b\| \leq 1$.
We finally denote by $\alpha_2^*(x)$ and $\alpha_4^*(x)$ the sum of all possible combinations of the above second and fourth mixed moments of $x$ and $x^*$; i.e.
\begin{align}\label{def:alphas}
\alpha_2^*(x) = \sum_{\epsilon_1,\epsilon_2 \in \{1,*\}} \| \mathfrak{m}_2^{x^{\epsilon_1},x^{\epsilon_2}} \|
\quad \text{and} \quad
\alpha_4^*(x) = \sum_{\epsilon_1,\epsilon_2,\epsilon_3,\epsilon_4 \in \{1,*\}} \| \mathfrak{m}_4^{x^{\epsilon_1},x^{\epsilon_2}, x^{\epsilon_3}, x^{\epsilon_4}} \|.
\end{align}

Recalling the definition of the maps in \eqref{moment-map}, we set
\begin{equation}\label{def:betas}
\begin{aligned}
\beta_2(x) &= \sum_{\epsilon_1,\epsilon_2 \in \{1,*\}} \max_{1\leq j\leq n } \max_{1\leq \ell_1, \ell_2 \leq d } \Big\| m_2^{x_j^{(\ell_1),\epsilon_1}, x_j^{(\ell_2), \epsilon_2}} \Big\|
\qquad \text{and}\\
\beta_4(x) &= \sum_{\epsilon_1,\epsilon_2,\epsilon_3,\epsilon_4 \in \{1,*\}} \max_{1\leq j\leq n } \max_{1\leq \ell_1, \ell_2, \ell_3 , \ell_4 \leq d }  \Big\|m_4^{x_j^{(\ell_1),\epsilon_1}, x_j^{(\ell_2),\epsilon_2},x_j^{(\ell_3),\epsilon_3}, x_j^{(\ell_4),\epsilon_4}} \Big\|.
\end{aligned}
\end{equation}
Finally, consider the quantities
\begin{multline}\label{def:B1-B2}
B_1(x,y):= \sqrt{\beta_2(x)} \Big(\sqrt{\beta_4(x) +  \beta_2(x)^2} + \sqrt{\beta_4(y) + \beta_2(x)^2} \Big)\\ \text{and} \qquad  B_2(x,y):= \|x\|_{L^3}^3 + \|y\|_{L^3}^3. 
\end{multline}

\subsection{Linear matrix pencils}
Linear matrix pencils are the essence of various powerful linearization techniques and constitute a fundamental class of noncommutative functions. They are the most basic matrix-valued expressions, are easy to handle, and can capture at the same time much information about the operators to which they are applied. 

\begin{theorem}\label{theo:Linear-pencil:freeness}
Let $g$ be a linear matrix pencil of the form \eqref{eq:linear_matrix_pencil_sa} and let $x$ be a family in $\A$ that is centered with respect to $E$ and is such that $\A_{x_1}, \dots, \A_{x_n}$ are free with amalgamation over $\B$. Consider the operator-valued $C^\ast$-probability space $\big(M_m(\A), \id_m \otimes E, M_m(\B)\big)$ which is associated with $(\A,E,\B)$. If $y$ is a family in $\A$ satisfying Assumption \ref{A:general}, then for any $\Zb \in \bH^+(M_m(\B))$,
\begin{align}\label{eq:theo_Linearpencil-1}
\big\|( \id_m \otimes  E) [G_{g(\ux_n)}(\Zb)] - ( \id_m \otimes  E)[G_{g(\uy_n)}(\Zb)] \big\|
\leq C_g \|\Im(\Zb)^{-1}\|^4 \ B_1(x,y) n,
\end{align}
with $C_g= m^3\big( \sum^d_{\ell=1} \|Q_\ell\| \big)^3$. In the case of an operator-valued $W^\ast$-probability space $(\A,\varphi,E,\B)$, we consider $\big(M_m(\A), $ $ \tr_m \otimes \varphi, \id_m \otimes E, M_m(\B)\big)$. Then, for any $z\in \mathbb{C}^+$,
\begin{align}\label{eq:theo_Linearpencil-2}
\big| (\tr_m \otimes \varphi) [G_{g(\ux_n)}(z)] - (\tr_m \otimes \varphi)[G_{g(\uy_n)}(z)] \big| 
 \leq  c_g \frac{1}{\Im (z)^4}  B_2(x,y) n,
\end{align}
with $c_g= d^3 \max_{1\leq \ell \leq d} \|Q_\ell\|^3$. Furthermore, for every $\epsilon>0$,
 \begin{align}\label{eq:theo_Linearpencil-3}
\frac{1}{\pi} \int_\R \big| (\tr_m \otimes \varphi) [G_{g(\ux_n)}(t+i\epsilon)] - (\tr_m \otimes \varphi) [G_{g(\uy_n)}(t+i\epsilon)] \big|\, \mathrm{d} t
 \leq C_g \frac{1}{\epsilon^3}  B_1(x,y) n.
\end{align}
\end{theorem}

The estimates hold for all points $\Zb$ and $z$ in the upper planes $\bH^+(M_m(\B))$ and $\mathbb{C}^+$ respectively, while depending on the inverse of their imaginary part. On the level of moments, the estimates are merely in terms of the scalar- or operator-valued moments up to order four. Thanks to the latter fact, Theorem \ref{theo:Linear-pencil:freeness} allows like Theorem \ref{theo:Op-freeCLT} an extension to unbounded operators in the spirit of Proposition \ref{prop:Op-freeCLT_unbounded}; this is straightforward and the details are thus omitted.

Note that linear matrix pencils enter in the framework of Theorem \ref{theo1:Lin-Lin:freeness} and thus the above estimates follow from the operator-valued setting. The proof will be illustrated in Section \ref{section:proof-LP}.

\subsection{Noncommutative polynomials}

Inasmuch as the joint $\B$-valued $\ast$-distribution of the operators under investigation is concerned, evaluations of selfadjoint noncommutative $\B$-valued $\ast$-polynomials provide access to this information in a more direct way. Namely, we consider polynomials in $\B \langle x_1, x_1^*, \ldots , x_d,x_d^*\rangle$, the complex unital algebra of noncommutative $\B$-valued polynomials in the non-commuting indeterminates $x=(x_1,x_1^*,\dots,x_d,x_d^*)$.

Linearity, which is an essential requirement for the Lindeberg method, is obviously lost when dealing with polynomials. This issue can be however resolved using linearization techniques for noncommutative polynomials (see Appendix \ref{section:linearizations}). This allows passing the study to linear polynomials having matrix-valued coefficients, applying an operator-valued Lindeberg method on the matrix level and finally deducing the desired estimates on the associated noncommutative polynomials. Contrary to the linear matrix pencil case, the polynomial case does not follow from Theorem \ref{theo1:Lin-Lin:freeness} even after linearization. It requires a closer analysis on the level of the linearization matrix which will become clearer in the proof. Recalling the definition of $B_1$ and $B_2$ in \eqref{def:B1-B2}, our result then reads as follows.

\begin{theorem}\label{theo:Lin-Lin:freeness}
Let $(\A, E, \B)$ be an operator-valued $C^\ast$-probability space. Let $x$ be a  family of elements in $\A$ that are centered with respect to $E$ and are such that $\A_{x_1}, \dots, \A_{x_n}$ are free with amalgamation over $\B$. If $y$ is a family satisfying Assumption \ref{A:general}
then, for sufficiently large $n$, there exists a positive integer $r_p$ such that:
for any $\Zb \in \bH^+(\B)$,
\begin{align}\label{oper:Lin-Lin:free}
\|  E [G_{p(\ux_n)}(\Zb)] - E [G_{p(\uy_n)}(\Zb)] \|  \leq  C_p  M_{x,y}^{8r_p} \Big(1 +\|\Im (\Zb)^{-1}\| \Big)^4 B_1(x,y)n,
\end{align}
and, in the case of an operator-valued $W^\ast$-probability space $(\A, \varphi, E, \B)$, for any $z\in \mathbb{C}^+$, 
\begin{equation}\label{scalar:Lin-Lin:free}
\big|\varphi [G_{p(\ux_n)}(z)] - \varphi [G_{p(\uy_n)}(z)]\big| \leq C_p  M_{x,y}^{8r_p} \Big(1 +\frac{1}{\Im (z)} \Big)^4 B_2(x,y) n ,
\end{equation}
and for any $\epsilon>0$,
\begin{align}\label{scalar:Lin-Lin:free-Levy}
\frac{1}{\pi} \int_\R \big| \varphi [G_{p(\ux_n)}(t+i\epsilon)] - \varphi [G_{p(\uy_n)}(t+i\epsilon)] \big|\, \mathrm{d} t
\leq C_p M_{x,y}^{6r_p}\frac{(1+\epsilon^2)}{\epsilon^3}   B_1(x,y) n,
\end{align}
where 
\begin{equation}\label{linearization-constant-1}
M_{x,y} := \|x\| + \|y\| +  \sqrt{n} \sqrt{\| E [xx^*] \|}+\sqrt{n} \sqrt{\| E [x^*x] \|}
\end{equation}
and $C_p$ is positive constant that only depends on the polynomial $p$ and that can take different values in each of the above bounds.
\end{theorem}
The constants $C_p$ and $r_p$ do not depend on $n$. They only depend on $d$ and a linear representation $\rho$ associated with the polynomial $p$ (see Appendix \ref{section:linearizations}). For each of the bounds \eqref{oper:Lin-Lin:free}, \eqref{scalar:Lin-Lin:free} and \eqref{scalar:Lin-Lin:free-Levy}, an explicit expression of $C_p$ can be easily tracked in Section \ref{section:proof-Polynomials} in case of interest.

The proof relies on an operator-valued Lindeberg method by blocks on the level of the linearization matrix. The replacement is done in a way that, at each step, all associated correlated elements are replaced at once as a block; hence the name. Having lifted the computations to a linear polynomial with matrix-valued coefficients, one can then apply the Lindeberg method on the matrix level and finally pass the desired estimates back on the initial polynomial. The proof is postponed to Section \ref{section:proof-Polynomials}.

\begin{remark}\label{rem:terms-est}
Note that when considering noncommutative polynomials as test functions, the operator norm of the variables appears in our estimates through $M_{x,y}$ but only in its \emph{non-leading} term. The leading terms only depend on the imaginary part of $\Zb$ or $z$ and on the operator-valued moments of the $x_j^{(k)}$'s up to order $4$. Indeed, the operator norm pops out when controlling the operator norm of the linearization matrix of the polynomial as shown in Lemma \ref{lem:estimates-L-inverse}. 
On the other hand, when considering linear matrix pencils as test functions, the operator norm does not appear in our estimates. Again the bounds in Theorem \ref{theo:Linear-pencil:freeness} only depend on the imaginary part of $\Zb$ or $z$ and the operator-valued moments of the $x_j^{(k)}$'s up to order $4$.
\end{remark}

\subsection{Multivariate $\B$-Free Central Limit Theorem}\label{section:MCLT}

Having obtained, in Section \ref{section:free-CLT}, quantitative bounds on the operator-valued free CLT, a natural continuation is to consider the multivariate setting. The operator-valued multivariate free CLT was proved first by Speicher \cite[Theorem 4.2.4]{Speicher-98} for identically distributed variables. More precisely, Speicher showed that if $(x^{(1)}_1, \dots , x^{(d)}_1)$, $\dots$, $(x^{(1)}_n, \dots , x^{(d)}_n)$ are centered freely independent and  identically distributed over $\B$ then
\begin{equation}\label{MCLT-Speicher}
\Big(  \frac{1}{\sqrt{n}} \sum_{j=1}^n x^{(1)}_j, \dots ,  \frac{1}{\sqrt{n}} \sum_{j=1}^n x^{(d)}_j  \Big) \xrightarrow[n\rightarrow \infty]{\B-^*d} 
(C_1, \dots, C_d)
\end{equation}
where  $\{C_1, \dots, C_d\}$ is a family of  $\B$-valued circular elements that are centered with respect to $E$ and whose covariance $(\eta, \widetilde{\eta})$ is given by the completely positive  maps
\[
\eta: \B \rightarrow M_{d}(\B), \quad b \mapsto [\eta_{k,\ell} ( b)]_{k,\ell=1}^d
\qquad \text{and} \qquad
\widetilde{\eta}: \B \rightarrow M_{d}(\B), \quad b \mapsto [\widetilde{\eta}_{k,\ell} ( b)]_{k,\ell=1}^d
\]
where 
\[
 \eta_{k,\ell} ( b) = E [x_1^{(k)*} \,  b  \, x_1^{(\ell)}] \quad \text{and} \quad \widetilde{\eta}_{k,\ell} ( b) =  E [x_1^{(k)}  \, b \,   x_1^{(\ell)*}] . 
\]
Note that the above convergence is in $*$-distribution over $\B$; see Section \ref{section:nc-distributions}. 
 
To describe the $\B$-valued free CLT analytically, we consider selfadjoint test functions, and quantify the above convergence in terms of Cauchy transforms. This follows respectively from Theorems \ref{theo:Linear-pencil:freeness} and \ref{theo:Lin-Lin:freeness}, which provide quantitative estimates for this convergence in the more general case of free but not necessarily identically distributed variables over $\B$. More precisely, let $x=\{x_j^{(k)} \mid \, 1\leq j \leq n , \, 1\leq k\leq d \}$ be a family in $\A$ with  $\A_{x_j} = \B\langle x_j^{(1)}, x_j^{(1)*},\dots, x_j^{(d)}, x_j^{(d)*}\rangle$ denoting the algebra generated by $\{ x_j^{(1)}, x_j^{(1)*},$ $ \dots , $ $ x_j^{(d)}, x_j^{(d)*} \}$ and $\B$. Assume that $x$ satisfies the following:

\begin{assumption}\label{A:MCLT}
$\A_{x_1}, \dots, \A_{x_n}$  are free  with amalgamation over $\B$ and for any $k,\ell = 1, \dots, d$ and $j=1, \dots, n $, 
\begin{itemize}
\item $ E [x^{(k)}_j]=0 $,
\item $E [x^{(k)}_j \,  b \,  x^{(\ell)}_j ]=0$ for all $b\in \B$.
\end{itemize}
 \end{assumption}
Now, let $\{C^{(1)}_n, \dots ,  C^{(d)}_n\}$ be a family of $\B$-valued circular elements with covariance $(\eta_n, \widetilde{\eta}_n)$ given by the completely positive maps
\[
\eta_n: \B \rightarrow M_{d}(\B), \quad b \mapsto \big[\eta_{k,\ell}^{(n)} ( b)\big]_{k,\ell=1}^d
\qquad \text{and} \qquad
\widetilde{\eta}_n: \B \rightarrow M_{d}(\B), \quad b \mapsto \big[\widetilde{\eta}_{k,\ell}^{(n)} ( b)\big]_{k,\ell=1}^d
\]
where
\begin{equation}\label{cov-MCLT}
 \eta_{k,\ell}^{(n)}  (b) = \frac{1}{n} \sum_{j=1}^n E [x_j^{(k)*} \,  b  \, x_j^{(\ell)}] \qquad \text{and} \qquad \widetilde{\eta}_{k,\ell}^{(n)}  ( b) = \frac{1}{n} \sum_{j=1}^n E [x_j^{(k)}  \, b \,   x_j^{(\ell)*}] . 
\end{equation}  
Again, the aim is to obtain quantitative estimates analytically in terms of Cauchy transforms when considering linear matrix pencils and noncommutative polynomials as test functions. To settle the notation, let us first define
\[
X_ n:= (X^{(1)}_n, X^{(1)*}_n,\dots , X^{(d)}_n, X^{(d)*}_n)
\qquad \text{with} \qquad 
X_n^{(k)}:= \frac{1}{\sqrt{n}} \sum_{j=1}^n x^{(k)}_j
\quad \text{for  } k = 1 , \dots , d,
\]
and moreover
\[
B_1(x) := \sqrt{\beta_2(x)} \Big(\sqrt{\beta_4(x) +  \beta_2(x)^2} + \sqrt{2} \beta_2 (x) \Big) 
\quad\text{and} \quad B_2(x):= \|x\|_{L^3}^3 + \sqrt{\beta_2(x)} \|x\|_{L^2}^2 .
\]

Our estimates are stated in the following two theorems whose proofs are combined and postponed to the end of this section.

\begin{theorem}\label{theo:op-MCLT-LP}
Let $g$ be a linear matrix pencil of the form \eqref{eq:linear_matrix_pencil_sa} and consider a family $x$ satisfying Assumption \ref{A:MCLT}. Consider the operator-valued $C^\ast$-probability space $\big(M_m(\A), \id_m \otimes E, M_m(\B)\big)$ which is associated with $(\A,E,\B)$. Then the following bounds hold: for any $\Zb \in \bH^+(M_m(\B))$,
 \begin{align*}
\big\| ( \id_m \otimes  E) \big[G_{g(X_n)}(\Zb)\big] - ( \id_m \otimes  E)\big[ G_{g(C^{(1)}_n, C^{(1)*}_n, \dots , C^{(d)}_n, C^{(d)*}_n )}(\Zb)\big] \big\|
\leq  C_g \|\Im (\Zb)^{-1}\|^4  B_1(x) \frac{1}{\sqrt{n}}.
\end{align*}
In the case of an operator-valued $W^\ast$-probability space $(\A,\varphi,E,\B)$, we consider the associated operator-valued $W^\ast$-probability space $\big(M_m(\A), \tr_m \otimes \varphi, \id_m \otimes E, M_m(\B)\big)$. Then, it holds for any $z\in \mathbb{C}^+$,
\begin{align*}
\Big|(\tr_m \otimes \varphi ) \big[G_{g(X_n)}(z)\big]  - (\tr_m \otimes \varphi ) \big[G_{g(C^{(1)}_n, C^{(1)*}_n, \dots , C^{(d)}_n, C^{(d)*}_n )}(z)\big]\Big| 
 \leq  c_g \frac{1}{\Im (z)^4}  B_2(x) \frac{1}{\sqrt{n}},
\end{align*}
where $C_g= m^3\big( \sum^d_{\ell=1} \|Q_\ell\| \big)^3$, $c_g= d^3 \max_{1\leq \ell \leq d} \|Q_\ell\|^3$ and  $\{C^{(1)}_n, \dots ,  C^{(d)}_n\}$ is a family of  $\B$-valued circular operators that are centered with respect to $E$ with covariance $(\eta_n, \widetilde{\eta}_n)$ given by \eqref{cov-MCLT}. Furthermore, there is a universal positive constant $c < 1.672$ such that
\begin{align*}
L\big(\mu_{g(X_n)}, \mu_{g(C^{(1)}_n, C^{(1)*}_n, \dots , C^{(d)}_n, C^{(d)*}_n )}\big) \leq c B_1(x)^{1/7} n^{-1/14}.
\end{align*}
\end{theorem}

\begin{remark}\label{rem:convergence_linear-pencil}
Let $x = \{x_n^{(k)} \mid n\in\bN, 1 \leq k \leq d\}$ be a family in $\A$ with the property that each of the finite families $x_n := \{x_j^{(k)} \mid 1 \leq j \leq n, 1 \leq k \leq d\}$ satisfies Assumption \ref{A:MCLT}.
Suppose further that $(\eta_n, \widetilde{\eta}_n)$ given by \eqref{cov-MCLT} is independent of $n$, say $(\eta,\widetilde{\eta})$, and let $\{C^{(1)},\dots,C^{(d)}\}$ be a family of $\B$-valued circular elements with the covariance $(\eta, \widetilde{\eta})$. Finally, we impose the condition that $\sup_{n\in\bN} B_1(x_n) < \infty$.
Then Theorem \ref{theo:op-MCLT-LP} yields for an arbitrary linear pencil $g$ of the form \eqref{eq:linear_matrix_pencil_sa} and for each $\epsilon>0$ that
$$\lim_{n\to \infty} \sup_{\Zb\in\bH^+(M_m(\B))\colon \Im(\Zb) \geq \epsilon \1} \big\|( \id_m \otimes  E) [G_{g(X_n)}(\Zb)] - ( \id_m \otimes  E)[G_{g(C)}(\Zb)] \big\| = 0,$$
where $C := (C^{(1)}, C^{(1)*}, \dots , C^{(d)}, C^{(d)*})$. With the help of the results outlined in Remark \ref{rem:Cauchy-convergence}, we see that this ensures that $(X^{(1)}_n,\dots,X^{(d)}_n) \stackrel{\B-^\ast d}{\longrightarrow} (C^{(1)}, \dots, C^{(d)})$ as $n\to \infty$;
notice that indeed the hermitizations of $\operatorname{diag}(X^{(1)}_n,\dots,X^{(d)}_n)$ and $\operatorname{diag}(C^{(1)}, \dots, C^{(d)})$, as well as all their amplifications with $\1_k \otimes \cdot$, are nothing but selfadjoint linear matrix pencils in the variables $X_n$ and $C$, respectively.
\end{remark}

\begin{theorem}\label{theo:op-MCLT-NCP}
Let $(\A, E, \B)$ be an operator-valued $C^\ast$-probability space. Let $p$ be a selfadjoint noncommutative polynomial in $ \B\langle x_1,$ $x_1^*, \dots,x_d, x_d^*\rangle $ of degree $\geq 1$, and consider a family $x$ satisfying Assumption \ref{A:MCLT}. Then, for sufficiently large $n$, there exist  positive integers $r_p$ and $C_p$, only depending on $p$, such that:
 for any $\Zb \in \bH^+(\B)$,
 \begin{align}\label{est:OVCT-polynomial}
\big\| E \big[G_{p(X_n)}(\Zb)\big] - E\big[ G_{p(C^{(1)}_n, C^{(1)*}_n, \dots , C^{(d)}_n, C^{(d)*}_n )}(\Zb)\big] \big\|
\leq C_p  M_{x}^{8r_p} \Big(1 +\|\Im (\Zb)^{-1}\| \Big)^4  B_1(x) \frac{1}{\sqrt{n}} ,
\end{align}
and, in the case of an operator-valued $W^\ast$-probability space $(\A, \varphi, E, \B)$, for any $z\in \mathbb{C}^+$, 
\begin{align*}
\Big| \varphi  \big[G_{p(X_n)}(z)\big]  -  \varphi  \big[G_{p(C^{(1)}_n, C^{(1)*}_n, \dots , C^{(d)}_n, C^{(d)*}_n )}(z)\Big]\Big|  \leq   C_p  M_{x}^{8r_p} \Big(1 +\frac{1}{\Im (z)} \Big)^4  B_2(x) \frac{1}{\sqrt{n}}, 
\end{align*}
where 
\[
M_x= \frac{1}{\sqrt{n}}\|x\| + \sqrt{\beta_2 (x)},
\]
and $\{C^{(1)}_n, \dots ,  C^{(d)}_n\}$ is a family of  $\B$-valued circular operators that are centered with respect to $E$ and whose covariance  $(\eta, \widetilde{\eta})$ is given by \eqref{cov-MCLT}. Furthermore, there is a universal positive constant $c < 1.846$ such that
\[
L(\mu_{p(X_n)},\mu_{p(C^{(1)}_n, C^{(1)*}_n, \dots , C^{(d)}_n, C^{(d)*}_n )}) \leq c M_x^{6r_\rho/7} B_1(x)^{1/7} n^{-1/14}.
\]
\end{theorem}

\begin{remark}\label{rem:convergence_polynomial}
In the situation of Remark \ref{rem:convergence_linear-pencil}, Theorem \ref{theo:Lin-Lin:freeness} yields for an arbitrary selfadjoint $\B$-valued noncommutative polynomial $p$ and for each $\epsilon>0$ that
$$\lim_{n\to \infty} \sup_{\Zb\in\bH^+(\B)\colon \Im(\Zb) \geq \epsilon \1} \big\|E[G_{p(X_n)}(\Zb)] - E[G_{p(C)}(\Zb)] \big\| = 0.$$
If we apply Theorem \ref{theo:Lin-Lin:freeness} to the variables $\1_k \otimes x$ in the operator-valued $C^\ast$-probability space $(M_k(\A), \id_k \otimes E, M_k(\B))$, then we conclude that in fact for every $k\in\bN$ and for each $\epsilon>0$,
$$\lim_{n\to \infty} \sup_{\Zb\in\bH^+(M_k(\B))\colon \Im(\Zb) \geq \epsilon \1} \big\|(\id_k\otimes E)[G_{\1_k \otimes p(X_n)}(\Zb)] - (\id_k \otimes E)[G_{\1_k\otimes p(C)}(\Zb)] \big\| = 0.$$
With the help of Remark \ref{rem:Cauchy-convergence}, we infer that $p(X_n) \stackrel{\B-d}{\longrightarrow} p(C)$ as $n\to \infty$.
In particular, we have that $E[p(X_n)] \to E[p(C)]$ as $n \to \infty$ for every selfadjoint $\B$-valued noncommutative polynomial $p$ and hence for all $\B$-valued noncommutative polynomials (because we can apply the former to both their real and imaginary parts); in other words, we have that $(X^{(1)}_n,\dots,X^{(d)}_n) \stackrel{\B-^\ast d}{\longrightarrow} (C^{(1)}, \dots, C^{(d)})$ as $n\to \infty$. Using the observation made in Remark \ref{rem:Cauchy-moment}, we can derive $E[p(X_n)] \to E[p(C)]$ as $n \to \infty$ for selfadjoint and hence for all $\B$-valued noncommutative polynomials $p$ without going through matricial amplifications. More precisely, by a direct application of Remark \ref{rem:Cauchy-moment} with $\epsilon_n=n^{-1/2}$ which is guaranteed by \eqref{est:OVCT-polynomial}, we get for $n$ sufficiently large that $\|E[p(X_n)] - E[p(C)]\| \leq  c n^{-1/6}$ for some constant $c>0$ that depends on $p$ but is independent of $n$.
\end{remark}

Our approach produces the first quantitative bounds on the L\'evy distance for the multivariate $\B$-free CLT without requiring any regularity conditions on the analytic distributions of the $\B$-valued circular family. Note that quantitative results on the operator-valued Cauchy transforms for the CLT in the setting of $\mathcal{T}$-free independence were obtained by Jekel and Liu \cite{Je-Li-19}. The same rate of convergence as in our theorem can be obtained as a \emph{non-trivial} consequence of their Theorem 8.10. While their bound would improve on the power of $\|\Im (\Zb)^{-1}\|$, it yields on the other hand estimates in terms of the operator norm instead of the moments; cf. \cite[Remark 8.15]{Je-Li-19}. Our extension of the Lindeberg method to the operator-valued setting allows to obtain bounds depending only on the moments for linear matrix pencils.
Combined with the linearization technique, our approach also allows obtaining bounds for noncommutative polynomials in which the leading term depends again on moments
while the operator norm appears only in the non-leading term; see Remark \ref{rem:terms-est}.

The multivariate free CLT is a particular case of the operator-valued one with $\B= \mathbb{C}$. Note that the family $\{C^{(1)}_n, \dots , C^{(d)}_n\} $ consists in this case of scalar-valued circular elements in $\A$ that are centered with respect to $\varphi$ and whose covariance is given by the matrix $\eta_n =(\eta_{k,\ell}^{(n)})_{k,\ell=1}^d$  with
\begin{equation}\label{cov:circular-MCLT}
\eta_{k,\ell}^{(n)}=\frac{1}{n} \sum_{j=1}^n \varphi [x_j^{(k)*}   x_j^{(\ell)}].
\end{equation}
In this setting and for free, selfadjoint and \emph{identically distributed} elements, Fathi and Nelson \cite{Fat-Nel-17} provided quantitative estimates for the multivariate version of the \emph{entropic} free CLT in terms of the non-microstates free entropy due to Wang \cite{W10}, by which they extended results of Chistyakov and G\"otze \cite{CG13}. On the level of Cauchy transforms, this was studied by Speicher \cite{Sp-07} and by Mai and Speicher \cite{Mai-Speicher-13} who also obtained the same rate of convergence but only for $z$ in a neighborhood of infinity. However, in order to pass these quantitative estimates to the L\'evy distance, the estimates on the difference of the Cauchy transforms should hold near the real axis; while there are general methods to extend these estimates towards the real axis (see the appendix of \cite{SV12}), the resulting bounds are far from being satisfying. In the particular case where $\eta_n$ is independent of $n$ and invertible, we could also push these estimates to the Kolmogorov distance. Indeed, this is done by an application of our previous result \cite[Theorem 5.3]{Banna-Mai-18} which also requires quantitative estimates on the Cauchy transforms over the strip $\{z \in \mathbb{C} \mid 0 < \Im (z) <\rho\}$ for some $\rho > 0$. This requirement is satisfied after Theorems \ref{theo:op-MCLT-LP} and \ref{theo:op-MCLT-NCP} above and thus we can obtain estimates on the Kolmogorov distance as illustrated in Corollary \ref{coro:Kolmogorov-MCLT}.
 
Before we give the precise statements, let us first recall the following terminology (see \cite{AEK2019,MSY18}): a positive linear map $\mathcal{L}: M_m(\C) \to M_m(\C)$ is said to be \emph{semi-flat} if there exists a constant $c>0$ such that $\mathcal{L}(B) \geq c \tr_m(B) \1_m$ holds for all positive semidefinite matrices $B \in M_m(\C)$.
To a linear matrix pencil $g$ of the form \eqref{eq:linear_matrix_pencil_sa} with coefficients $Q_0,Q_1,\dots,Q_d$ in $M_m(\C)$ with $Q_0$ being selfadjoint, we associate the \emph{quantum operator} $\mathcal{L}: M_m(\C) \to M_m(\C)$ which is the (completely) positive linear map defined by
\[
\mathcal{L}:\ M_m(\C) \rightarrow M_m(\C), \qquad B \mapsto \sum_{\ell=1}^{2d} \widetilde{Q}_\ell B \widetilde{Q}_\ell\]
where $\widetilde{Q}_{2k} := Q_k + Q_k^\ast$ and $\widetilde{Q}_{2k-1} := i (Q_k - Q_k^\ast)$ for $k=1,\dots,d$.

\begin{corollary}\label{coro:Kolmogorov-MCLT}
Let the setting be as in Theorems \ref{theo:op-MCLT-LP} and \ref{theo:op-MCLT-NCP} with $\B = \C$.
Suppose in addition that the covariance matrix $\eta_n=(\eta^{(n)}_{k,\ell})_{k,\ell=1}^d$ as defined in \eqref{cov:circular-MCLT} is independent of $n$ and invertible.
If $g$ admits a semi-flat quantum operator $\mathcal{L}$, then there exists a constant $K>0$ such that for all $n \in \bN$, 
\[
\Delta \Big( \mu_{g(X_n)}, \mu_{g(C^{(1)}_n, C^{(1)*}_n,\dots , C^{(d)}_n, C^{(d)*}_n)}\Big) \leq K n^{-1/14}. 
\]
If $p$ has degree $\deg(p)=r \geq 1$, then there exists a constant $K>0$ such that for all $n \in \bN$, 
\begin{equation}\label{Kolm-bound}
\Delta \Big( \mu_{p(X_n)}, \mu_{p(C^{(1)}_n, C^{(1)*}_n,\dots , C^{(d)}_n, C^{(d)*}_n)}\Big) \leq K n^{-1/(2^{r+3}-6)}. 
\end{equation}
\end{corollary}

To the best of our knowledge, this is the first result providing quantitative bounds on the Kolmogorov distance in the multivariate setting. Remarkably, the speed of convergence in \eqref{Kolm-bound} depends only on the degree $r$ of the noncommutative polynomial $p$.

We end this section by giving the proofs of Theorems \ref{theo:op-MCLT-LP} and \ref{theo:op-MCLT-NCP} and of Corollary \ref{coro:Kolmogorov-MCLT}.

\begin{proof}[Proof of Theorems \ref{theo:op-MCLT-LP} and \ref{theo:op-MCLT-NCP}] This is a direct consequence of Theorems \ref{theo:Linear-pencil:freeness} and \ref{theo:Lin-Lin:freeness}. With a particular choice of the family $y$, one then only needs to control the corresponding bounds on the moments and operator norms. 

We consider a family $y=\{y_j^{(k)} \mid 1\leq j\leq n,$ $ 1 \leq k\leq d\}$ of $\B$-valued circular elements that is free from $x$ over $\B$. We choose $y$ such that it satisfies Assumption \ref{A:MCLT} and has the same moments of second order as $x$, i.e. $\A_{y_1}, \dots, \A_{y_n}$ are free over $\B$ and for any choice of $k,\ell  \in [d]$, $j \in [n]$ and $\epsilon_1 , \epsilon_2 \in \{ 1,* \}$
\[
E [ y_{j}^{(k)}] = E[ x_{j}^{(k)}]=0
\quad \text{and} \quad
E [ y_{j}^{(k), \epsilon_1} b \,  y_{j}^{(\ell), \epsilon_2} ] = E [ x_{j}^{(k), \epsilon_1} b \,   x_{j}^{(\ell),\epsilon_2}] \quad \text{for all } b \in \B. 
\]
Denote by $\tilde{x}_j:= \frac{1}{\sqrt{n}}(x_j^{(1)},$ $ x_j^{(1)*}, $  $\dots ,x_j^{(d)},$ $ x_j^{(d)*})$ and $\tilde{y}_j:=  \frac{1}{\sqrt{n}}(y_j^{(1)}, y_j^{(1)*}, \dots ,y_j^{(d)}, y_j^{(d)*})$. Applying Theorems \ref{theo:Linear-pencil:freeness} and \ref{theo:Lin-Lin:freeness} for  $\tilde{\ux}= \sum_{j=1}^n \tilde{x}_j$ and $\tilde{\uy}=  \sum_{j=1}^n \tilde{y}_j$, the bounds on the scalar- and operator-valued Cauchy transforms in Theorems \ref{theo:op-MCLT-LP} and \ref{theo:op-MCLT-NCP} follow after controlling the quantities $B_1(\tilde{x},\tilde{y})$, $B_2(\tilde{x},\tilde{y})$ and  $M_{\tilde{x},\tilde{y}}$. We start by noting that for any $n \in \mathbb{N}$,
\[
\begin{aligned}
& \beta_4(\tilde{x}) = \frac{1}{n^2}\beta_4(x), \quad \beta_2(\tilde{x}) = \frac{1}{n}\beta_2(x),
\quad \|\tilde{x}\|_{L^3}^3 = \frac{1}{n \sqrt{n}} \|x\|_{L^3}^3, \quad \text {and} \quad  \|\tilde{x}\| = \frac{1}{\sqrt{n}}\|x\|.
\end{aligned}
\]
With our choice of the family $y$, we then get by the $\B$-valued moment-cumulant formula,  
\[
\begin{aligned}
&\beta_4(\tilde{y})=\frac{1}{n^2} \beta_4(y) \leq \frac{1}{n^2} \beta_2(y)^2=\frac{1}{n^2} \beta_2(x)^2,
\quad \text{and} \quad
\beta_2 (\tilde{y})= \frac{1}{n} \beta_2(y)= \frac{1}{n} \beta_2(x).
\end{aligned}
\]
Now, with arguments similar to the proof of Theorem \ref{theo:Op-freeCLT}, we get
\[
\|y_j^{(k)}\|_{L_3}^3 \leq \big(2\min \big\{ \|E[y_j^{(k)*}y_j^{(k)}]\| , \|E[y_j^{(k)}y_j^{(k)*}]\|\big\}\big)^{1/2} \|y\|_{L_2}^2 \leq \sqrt{\beta_2(y)} \|y\|_{L_2}^2, 
\]
hence, by taking the maximum over $k \in [d]$ and $j \in [n]$, we get
\[
\|\tilde{y}\|_{L^3}^3 = \frac{1}{n \sqrt{n}} \|y\|_{L^3}^3 \leq \frac{1}{n \sqrt{n}} \sqrt{\beta_2(y)} \|y\|_{L_2}^2 =   \frac{1}{n \sqrt{n}} \sqrt{\beta_2(x)} \|x\|_{L_2}^2.
\]
Putting the above terms together, we obtain  
\[
B_1(\tilde{x},\tilde{y}) = \frac{1}{n \sqrt{n}} B_1(x)
\quad \text{and} \quad 
B_2(\tilde{x},\tilde{y}) = \frac{1}{n \sqrt{n}} B_2(x).
\]
Finally, by \cite[Proposition 4.14]{Dyk-05}, the operator norm of the $\B$-valued circular elements can be controlled as follows 
\[
\|y_j^{(k)}\| \leq 2\big(\max \big\{ \|E[y_j^{(k)*}y_j^{(k)}]\| , \|E[y_j^{(k)}y_j^{(k)*}]\|\big\}\big)^{1/2} \leq 2\sqrt{\beta_2(y)}=2\sqrt{\beta_2(x)} \leq 2\|x\|.
\]
Taking the maximum over $k \in [d]$ and $j \in [n]$, we we infer that $ M_{\tilde{x},\tilde{y}} \leq c M_x $ for some positive constant $c$. To conclude the bounds on the Cauchy transforms, it remains to notice that $\tilde{\uy}\stackrel{d}{=}(C^{(1)}_n, C^{(1)*}_n, \dots , C^{(d)}_n,C^{(d)*}_n )$ with $\{C^{(1)}_n, \dots , C^{(d)}_n\} $  a family of operator-valued circular elements over $\B$ that are centered with respect to $E$ and whose covariance is given by \eqref{cov-MCLT}. 

Finally, to get the estimate on the L\'evy distance for the case of linear matrix pencils, we use the bound \eqref{eq:theo_Linearpencil-3}, together with the L\'evy bound \eqref{eq:Levy_bound}, to get
\[
L(\mu_{g(X_n)},\mu_{g(C^{(1)}_n, C^{(1)*}_n, \dots , C^{(d)}_n, C^{(d)*}_n )}) \leq 2 \sqrt{ \frac{\epsilon}{\pi}} + \frac{B_1(x)}{\epsilon^3\sqrt{n}}.
\]
Optimizing over $\epsilon\in (0,\infty)$, we infer that
\[
L(\mu_{X_n},\mu_{g(C^{(1)}_n, C^{(1)*}_n, \dots , C^{(d)}_n, C^{(d)*}_n )}) \leq c B_1(x)^{1/7} n^{-1/14},
\]
for
the universal constant $c = 7(\frac{1}{9\pi})^{3/7} < 1.672$.
In the same way, we use the bound \eqref{scalar:Lin-Lin:free-Levy} to get
\[
L(\mu_{p(X_n)},\mu_{p(C^{(1)}_n, C^{(1)*}_n, \dots , C^{(d)}_n, C^{(d)*}_n )}) \leq 2 \sqrt{ \frac{\epsilon}{\pi}}  + \frac{1+\epsilon^2}{\epsilon^3}\frac{M_x^{6r_\rho}B_1(x)}{\sqrt{n}} \leq 2 \sqrt{ \frac{\epsilon}{\pi}}  + \frac{2}{\epsilon^3}\frac{M_x^{6r_\rho}B_1(x)}{\sqrt{n}}
\]
for any $\epsilon \in (0,1)$ and thus, after optimizing over $\epsilon\in (0,1) $, we get for $n$ sufficiently large that
\[
L(\mu_{p(X_n)},\mu_{p(C^{(1)}_n, C^{(1)*}_n, \dots , C^{(d)}_n, C^{(d)*}_n )}) \leq c M_x^{6r_\rho/7} B_1(x)^{1/7} n^{-1/14},
\]
for the universal constant $c = 7 (\frac{1}{9\pi})^{3/7} 2^{1/7} < 1.846$.
\end{proof}

\begin{proof}[Proof of Corollary \ref{coro:Kolmogorov-MCLT}]
We start by  noting that when $\B=\C$ and $E=\varphi$, the family $\{C^{(1)}_n, \dots ,$ $ C^{(d)}_n \}$ is a family of centered circular elements with covariance matrix given in \eqref{cov:circular-MCLT}.

If $g$ is a selfadjoint linear matrix pencil of the form \eqref{eq:linear_matrix_pencil_sa}, then $g(C^{(1)}_n, C^{(1)*}_n, \dots , C^{(d)}_n, C^{(d)*}_n )$ can be rewritten as a linear matrix pencil $\tilde{g}$ in $2d$ correlated semicircular elements having finite non-microstates free Fisher information thanks to the invertibility of $\eta_n$. Thus, Theorem~8.1 in \cite{MSY18} ensures that the cumulative distribution function of $\mu_{g(C^{(1)}_n, C^{(1)*}_n,\dots , C^{(d)}_n, C^{(d)*}_n)}$ is H\"older continuous with exponent $\frac{2}{3}$.
Notice that similarly $p (C^{(1)}_n, C^{(1)*}_n, \dots , C^{(d)}_n, C^{(d)*}_n) $ can be written as a selfadjoint polynomial $\tilde{p}$ in correlated semicircular elements with degree $\deg(\tilde{p})=\deg(p)=r$. Now as this semicircular family admits Lipschitz conjugate variables thanks to the invertibility of $\eta_n$, then by Theorem~1.1 in \cite{Banna-Mai-18}, the cumulative distribution function of $\mu_{p(C^{(1)}_n, C^{(1)*}_n,\dots , C^{(d)}_n, C^{(d)*}_n)}$ is H\"older continuous with exponent $1/(2^r-1)$.

Combing these facts with the bound on the scalar-valued Cauchy transform in Theorems \ref{theo:op-MCLT-LP} and \ref{theo:op-MCLT-NCP}, then Theorem~5.3 in \cite{Banna-Mai-18} guarantees thanks to \cite[Proposition 7.1]{Junge} the existence of a numerical constant $K>0$ for which the asserted bounds on the Kolmogorov distance hold.
\end{proof}

\subsection{Matrices with covariance profile} \label{section:matrices}

In his fundamental paper \cite{Voi-91}, Voiculescu proved an asymptotic freeness result that revealed the connection between random matrices and free probability theory. Indeed, he proved that a family of independent Gaussian matrices is asymptotically free and converges in $*$-distribution to free semicircular elements. This  result was then extended to matrices with bosonic and fermionic entries by Shlyakhtenko \cite{Shl-96, Sh-97} who illustrated the connection between random band matrices and freeness with amalgamation \cite[Theorem 4.1]{Sh-97}. Asymptotic freeness extensions of Voiculescu's result were also given by Ryan \cite{Ryan-98} and Liu \cite{Liu-18}, respectively, for Wigner matrices with free and conditionally free entries having identical variance. In this section, we push these results to families of operator-valued matrices with $\B$-free entries that can have a general covariance structure. We consider linear matrix pencils and noncommutative polynomials as test functions and provide quantitative estimates about such convergences in terms of Cauchy transforms. 
 
Fix $d \in \bN$ and $n\in \bN$, and let $A^{(1)}_N , \dots , A^{(d)}_N$ be operator-valued matrices in $M_N(\A)$ that are given, for any $1\leq k\leq d$, by 
\[
A^{(k)}_N := \frac{1}{\sqrt{N}} [ a^{(k)}_{ij}]_{i,j=1}^N =  \frac{1}{\sqrt{N}} \sum_{i,j=1}^N  E_{ij} \otimes a_{ij}^{(k)}
\]
where $(E_{ij})_{1\leq i,j \leq N}$ are the standard matrix units in $M_N(\C)$. All along this section, we will be working in the setting of an operator-valued $W^*$-probability space $(M_N(\A), \tr_N \otimes \varphi, \id_N \otimes E, M_N(\B))$. We assume that $a= \{a_{ij}^{(k)} \mid 1\leq i,j \leq N, 1\leq k\leq d\} $ satisfies Assumption \ref{A:MCLT}, i.e. all $\A_{a_{ij}}$'s are freely independent with amalgamation over $\B$, $E[a_{ij}^{(k)}]=0$ and  $E[a_{ij}^{(k)} b a_{ij}^{(\ell)}]= E[a_{ij}^{(k)*} b a_{ij}^{(\ell)*}]=0$ for all  $b \in \B$,  $1\leq i,j \leq N$ and $1 \leq k, \ell \leq d$. Note that the matrices themselves can be correlated and that their entries do not need to have identical variances.  Our aim is to obtain explicit estimates on the analytic distribution of $(A^{(1)}_N , \dots , A^{(d)}_N )$ in terms of Cauchy transforms by considering as before linear matrix pencils and selfadjoint polynomials in noncommutative variables as test functions. Moreover, we provide quantitative estimates, when $\B= \mathbb{C}$, on the order of convergence under the Kolmogorov distance. 

With this aim, let us consider a family $\{C^{(1)}_N, \dots , C^{(d)}_N\}$ of 
operator-valued circular elements over $\D_N^\B:=  \D_N \otimes \B$, the subalgebra of diagonal matrices in $ M_N(\C) \otimes \B $, whose covariance $(\eta_N, \widetilde{\eta}_N)$ is given by the completely positive maps
\[
\eta_N: \D_N^\B \rightarrow M_{d}(\D_N^\B), \; D \mapsto \big[\eta_{k,\ell}^{(N)} ( D)\big]_{k,\ell=1}^d
\quad \text{and} \quad
\widetilde{\eta}_N: \D_N^\B \rightarrow M_{d}(\D_N^\B), \, D \mapsto \big[\widetilde{\eta}_{k,\ell}^{(N)} ( D)\big]_{k,\ell=1}^d
\]
where for any $k,\ell =1 , \dots, d$ and any $  D=(d_{ij})_{i,j=1}^N \in \D_N^\B$, 
\begin{equation}\begin{aligned}\label{covariance-circular-family-matrices}
\eta_{k,\ell}^{(N)} (D) &= (\id_N \otimes E) \big[ A_N^{(k)*} DA_N^{(\ell)} \big] \quad \text{with }  \quad \big(\eta_{k,\ell}^{(N)} (D) \big)_{i,j} = \delta_{i,j} \frac{1}{N} \sum_{r=1}^N  E \big[a^{(k)*}_{ri} d_{rr} a^{(\ell)}_{ri}\big] ,\\
\widetilde{\eta}_{k,\ell}^{(N)} (D) &= (\id_N \otimes E) \big[ A_N^{(k)} DA_N^{(\ell)*} \big] \quad \text{with }  \quad \big(\widetilde{\eta}_{k,\ell}^{(N)} (D) \big)_{i,j} = \delta_{i,j} \frac{1}{N} \sum_{r=1}^N  E \big[a^{(k)}_{ir} d_{rr} a^{(\ell)*}_{ir}\big] \; .
\end{aligned}\end{equation}
We shall approximate the analytic distribution of $(A^{(1)}_N , \dots , A^{(d)}_N)$ with that of $(C^{(1)}_N, \dots , C^{(d)}_N)$ by providing Berry-Esseen bounds on the scalar-valued Cauchy transforms when considering linear matrix pencils and noncommutative polynomials as test functions. Let us recall first the quantities:
\[
B_2(x)= \|x\|_{L^3}^3 + \sqrt{\beta_2(x)} \|x\|_{L^2}^2
\quad \text{and} \quad
M_a=\frac{1}{\sqrt{N}}\|a\|+ \sqrt{\| E[a^*a] \|} + \sqrt{\| E[aa^*] \|}.
\]
\begin{theorem}\label{theo:Lin-Lin:matrices}
Let $a$	be a family satisfying Assumption \ref{A:MCLT} from which we construct a family of operator-valued matrices 
$
A_N:= \big(A^{(1)}_N, A^{(1)*}_N, \dots , A^{(d)}_N, A^{(d)*}_N\big).
$ The following bounds hold:\\ 
if $g$ is a linear matrix pencil \eqref{eq:linear_matrix_pencil_sa} then for any $z\in \mathbb{C}^+$,
\[
\big| (\tr_{mN}  \otimes \varphi) \big[G_{g(A_N)}(z) 
\big]  - (\tr_{mN}  \otimes \varphi) \big[
G_{g(C^{(1)}_N, C^{(1)*}_N,\dots , C^{(d)}_N,C^{(d)*}_N)}(z)\big] \big| \! \leq  \frac{c_g}{\Im (z)^4} B_2(a) \frac{1}{\sqrt{N}}, 
\]
with $c_g= d^3 \max_{1\leq \ell \leq d} \|Q_\ell\|^3$; whereas, if $p \in  \B\langle x_1, x_1^*,\dots , x_d, x_d^*\rangle $ is a noncommutative polynomial of degree $\geq 1$, then for sufficiently large $N$, there exist positive constants $C_p$ and $r_\rho$ depending only on $p$ such that for any $z\in \mathbb{C}^+$,
\begin{align*}
\big|\! (\tr_N \otimes \varphi) \big[G_{p(A_N)}(z) 
\big] - (\tr_N \otimes \varphi) \big[
G_{p(C^{(1)}_N, C^{(1)*}_N,\dots , C^{(d)}_N,C^{(d)*}_N)}(z)\big]\! \big| 
\leq   C_p \Big( \frac{1}{\Im (z)} +1 \Big)^4 M_a^{8r_\rho} B_2(a) \frac{1}{\sqrt{N}} 
\end{align*}
where $\{C^{(1)}_N, \dots , C^{(d)}_N\}$ is a family of centered operator-valued circular elements over $\D_N^\B:= \B \otimes\D_N$  with covariance $(\eta_N, \widetilde{\eta}_N)$. 
\end{theorem}

This theorem is the first result, to the best of our knowledge, that gives quantitative estimates on the differences of Cauchy transforms for polynomials in operator-valued matrices with $\B$-free entries. Moreover, the matrices are allowed to have a very general covariance profile and the entries can have different distributions. Moreover, our estimates are in terms of the moments up to the third order. Only when considering polynomials as test functions, the operator norm appears in our estimates but merely as a non-leading term. 

Note that, when considering operator-valued matrices, the estimates \eqref{eq:theo_Linearpencil-3} and \eqref{scalar:Lin-Lin:free-Levy} yield bounds for the L\'evy distance which are of order $\sqrt{N}$. Hence, as $N\to \infty$, we cannot deduce convergence on the level of the L\'evy distance without imposing further conditions.
For instance, by Theorem 5.3 in \cite{Banna-Mai-18}, we can obtain quantitative estimates in terms of the Kolmogorov distance whenever the analytic distributions of $g(C^{(1)}_N, C^{(1)*}_N,\dots , C^{(d)}_N,C^{(d)*}_N)$ and  $p(C^{(1)}_N, C^{(1)*}_N,\dots , C^{(d)}_N,C^{(d)*}_N)$ have H\"older continuous cumulative distribution functions. Instances of such a scenario are given in the following corollary.

\begin{corollary}\label{coro:matices}
Let the setting be as in Theorems \ref{theo:Linear-pencil:freeness} and \ref{theo:Lin-Lin:matrices} with $\B= \C$. Assume that $k,\ell=1, \dots, d$
\begin{equation}\label{cond:matrices}
 \sum_{r=1}^N  \varphi (a^{(k)*}_{ri} a^{(\ell)}_{ri})
\quad \text{and} \quad 
\sum_{r=1}^N  \varphi (a^{(k)*}_{ir}  a^{(\ell)}_{ir}) 
\quad \text{are independent of i}.
\end{equation}
 Then  for any $N \in \bN$, $(C^{(1)}_N, \dots , C^{(d)}_N)$  is a family of centered circular elements in $\A$ whose covariance $(\eta, \widetilde{\eta} )$ is given by the matrices $\eta =(\eta_{k,\ell})_{k,\ell=1}^d$ and $\widetilde{\eta} =(\widetilde{\eta}_{k,\ell})_{k,\ell=1}^d$ with 
\[
\eta_{k,\ell}=\frac{1}{N} \sum_{i=1}^N  \varphi (a^{(k)*}_{i1} a^{(\ell)}_{i1})
\quad \text{and} \quad 
\widetilde{\eta}_{k,\ell}=\frac{1}{N} \sum_{j=1}^N  \varphi (a^{(k)}_{1j}  a^{(\ell)*}_{1j}).
\]
Moreover, if $\eta$ and $\widetilde{\eta}$ are independent of $N$ and if  $p$ is of degree $\deg(p)=r \geq 1$, then the same estimate  as in \eqref{Kolm-bound} holds on the Kolmogorov distance.
\end{corollary}

\begin{proof}[Proof of Theorem \ref{theo:Lin-Lin:matrices}]
The proof follows from the estimates \eqref{eq:theo_Linearpencil-2} and \eqref{scalar:Lin-Lin:free} in Theorems \ref{theo:Linear-pencil:freeness} and \ref{theo:Lin-Lin:freeness} in the framework of the operator-valued probability space $ (M_N(\A), \tr_N \otimes \varphi, \id_N \otimes E, M_B(\B))$. With a particular choice of the family $y$, one then only needs to control the corresponding bounds on the moments and operator norms. 
We let $c=\{c_{ij}^{(k)} \mid 1\leq i,j\leq N,  1\leq k\leq d\}$ be a family of $\B$-valued circular elements in $\A$ that satisfies Assumption \ref{A:MCLT} with respect to $E$. Moreover, assume that $c$ is also free from $a$ over $\B$. Define for all $k=1, \ldots, d$, the $N \times N$ matrix $C^{(k)}_N= \frac{1}{\sqrt{N}} [ c^{(k)}_{ij}]_{i,j=1}^N$ and note that $\{A^{(1)}_N , \dots , A^{(d)}_N\}$ and $\{ C^{(1)}_N , \dots , C^{(d)}_N\}$ are free with amalgamation over $M_N(\B)$. Now,  set $n=N^2$ and 
\begin{align*}
\ux &= \Big( \sum_{i,j=1}^N  x_{ij}^{(1)}, \sum_{i,j=1}^N  x_{ij}^{(1)*}, \dots , \sum_{i,j=1}^N x_{ij}^{(d)}, \sum_{i,j=1}^N x_{ij}^{(d)*} \Big),
\\
\uy &= \Big( \sum_{i,j=1}^N  y_{ij}^{(1)}, \sum_{i,j=1}^N  y_{ij}^{(1)*}, \dots , \sum_{i,j=1}^N y_{ij}^{(d)}, \sum_{i,j=1}^N y_{ij}^{(d)*} \Big)
\end{align*}
with $x_{ij}^{(k)}= \frac{1}{\sqrt{N}}   E_{ij} \otimes a_{ij}^{(k)}$ and $y_{ij}^{(k)}=  \frac{1}{\sqrt{N}} E_{ij} \otimes c_{ij}^{(k)}$. Clearly by Assumption \ref{A:MCLT}, the subalgebras
$$\big\{ M_N(\B) \langle x_{ij}^{(1)} , x_{ij}^{(1)*} ,\dots , x_{ij}^{(d)},x_{ij}^{(d)*} \rangle ,\, M_N(\B)\langle y_{ij}^{(1)} ,y_{ij}^{(1)*} , \dots , y_{ij}^{(d)} ,y_{ij}^{(d)*}\rangle \mid 1 \leq i,j \leq N \big\}$$
are free with amalgamation over $M_N(\B)$. Moreover, the covariance structure is preserved on the matrix level in the sense that for any $k,\ell =1, \dots, d$, $1\leq i,j \leq N$,  $\epsilon_1, \epsilon_2 \in \{1,*\}$ and $B \in M_N(\B)$,
\begin{multline*}
(\id_N \otimes E) [ x_{ij}^{(k)}] =(\id_N \otimes E)[ y_{ij}^{(k)}]=0 \; \text{ and }\;
(\id_N \otimes E)[ x_{ij}^{(k),\epsilon_1} B x_{ij}^{(\ell),\epsilon_2}] =(\id_N \otimes E) [ y_{ij}^{(k),\epsilon_1} B y_{ij}^{(\ell),\epsilon_2}]. 
\end{multline*}
To this extend, we have shown that all the requirements of Theorems \ref{theo:Linear-pencil:freeness} and \ref{theo:Lin-Lin:freeness} hold. One only needs to control the terms $B_2(x,y)$ and $M_{x,y}$. With this aim, and using similar arguments as in Theorems \ref{theo:op-MCLT-LP}
 and  \ref{theo:op-MCLT-NCP}, we note that
 \[
 \begin{aligned}
 &\|x\| =\frac{1}{\sqrt{N}}\|a\|, \qquad \quad
 \| x\|_{L^3} \leq \frac{1}{N\sqrt{N}}\|a\|_{L^3(\A, \varphi)}^3 \max_{1\leq i,j \leq N}\| E_{ij}\|_{L^3(M_N(\C),\tr_N)}^3 = \frac{1}{N^2\sqrt{N}} \|a\|_{L^3}^3, 
 \\ 
 &\|y\|  \leq \frac{1}{\sqrt{N}}\|a\|,
 \qquad  \quad
 \| y\|_{L^3} \leq \frac{1}{N^2\sqrt{N}} \|c\|_{L^3}^3 \leq \frac{1}{N^2\sqrt{N}} \sqrt{\beta_2(a)}\|a\|_{L^2}^2,
 \\
& \|E[x^*x]\|= \frac{1}{N} \|E[a^*a]\|
 \quad \text{and}\quad
  \|E[xx^*]\|= \frac{1}{N} \|E[aa^*]\|.
 \end{aligned}
\]
Hence, we infer that 
\[
B_2(x,y) \leq \frac{1}{N^{5/2}} \Big( \|a\|_{L^3}^3 +\sqrt{\beta_2(a)}\|a\|_{L^2}^2\Big) 
\; \text{and} \; 
M_{x,y} \leq c\Big( \frac{1}{\sqrt{N}}\|a\|+ \sqrt{\| E[a^*a] \|} + \sqrt{\| E[aa^*] \|}\Big).
\]
The last step of the proof consists of showing that $\{C^{(1)}_N, \dots , C^{(d)}_N\}$ is a family of 
operator-valued circular elements over $\D_N^\B =  \D_N \otimes \B$ whose covariance $(\eta_N, \widetilde{\eta}_N)$ is given in \eqref{covariance-circular-family-matrices}. Indeed, by an operator-valued variation of  Proposition 13 in  \cite[Chapter 9]{Mingo-Speicher}, we  compute the $M_N(\B)$-valued cumulants of $C^{(1)}_N, \dots , C^{(d)}_N$ in terms of the $\B$-cumulants of their entries. With our choice of  the $c^{(k)}_{ij}$'s, we get for any $m \in \mathbb{N}$,  $k,k_1, \dots , k_m \in [d]$, $\epsilon, \epsilon_1, \dots , \epsilon_m \in \{1, *\}$  and $B, B_1 , \dots , B_{m-1} \in M_N(\B)$,
\[\kappa^{M_N(\B)}_m [C^{(k_1),\epsilon_1}_NB_1, \dots ,C^{(k_{m-1}), \epsilon_{m-1}}_N B_{m-1} ,C^{(k_m),\epsilon_m}_N]  =0
\quad \text{whenever } m>2
\]
and in the other cases $\kappa_1^{M_N(\B)}[C^{(k),\epsilon}_N] = 0$ and $\kappa^{M_N(\B)}_2[C^{(k),\epsilon}_N B , C^{(\ell),\epsilon}_N] =  0$ as well as
\[
\kappa^{M_N(\B)}_2[C^{(k)*}_N B , C^{(\ell)}_N]= (\id_N \otimes E) \big[ C_N^{(k)*} D C_N^{(\ell)} \big] 
\enspace \text{and} \enspace
\kappa^{M_N(\B)}_2[C^{(k)}_N B , C^{(\ell)*}_N]= (\id_N \otimes E) \big[ C_N^{(k)} D C_N^{(\ell)*} \big] 
\]
where for any  $i,j \in [N]$
\begin{align*}  \big(\kappa^{M_N(\B)}_2[C^{(k)*}_N B , C^{(\ell)}_N] \big)_{i,j} &= \delta_{i,j} \frac{1}{N} \sum_{r=1}^N  \kappa_2^\B [c^{(k)*}_{ri} B_{rr}, c^{(\ell)}_{ri}] = \delta_{i,j} \frac{1}{N} \sum_{r=1}^N  E [c^{(k)*}_{ri} B_{rr} c^{(\ell)}_{ri}],\\
 \big(\kappa^{M_N(\B)}_2[C^{(k)}_N B , C^{(\ell)*}_N] \big)_{i,j}& = \delta_{i,j} \frac{1}{N} \sum_{r=1}^N  \kappa_2^\B [c^{(k)}_{ir} B_{rr}, c^{(\ell)*}_{ir}] = \delta_{i,j} \frac{1}{N} \sum_{r=1}^N  E [c^{(k)}_{ir} B_{rr} c^{(\ell)*}_{ir}].
\end{align*}
To end the proof, it suffices to remark that the subalgebra of diagonal matrices $\D_N^\B \subset M_N(\C) \otimes \B$ is closed under the covariance maps.
\end{proof}

\begin{proof}[Proof of Corollary \ref{coro:matices}] The proof  follows from  Theorem \ref{theo:Lin-Lin:matrices}. It suffices to notice that under condition \eqref{cond:matrices}, we have that for any $k$ and $\ell$
\[
\kappa^{\D_N^\B}_2 \big[ C^{(k),\epsilon_1}_N  , C^{(\ell),\epsilon_2}_N\big]=  \eta_{k,\ell}(1)= \kappa^{\D_N}_2 \big[ C^{(k),\epsilon_1}_N  , C^{(\ell),\epsilon_2}_N\big]. 
\]   
Then by \cite[Theorem 3.1]{Ni-Di-Sp-operator-valued}, $C^{(1)}_N, \dots, C^{(d)}_N$  is a family of centered circular elements in $\A$ with covariance is as announced.
\end{proof}

\subsection{Wigner matrices with covariance profile}\label{section:Wigner}

This section is devoted for the study of the noncommutative joint distribution of correlated operator-valued Wigner matrices with $\B$-free entries. Consider the operator-valued $W^*$-probability space $ (M_N(\A), \tr_N \otimes \varphi, \id_N \otimes E, M_B(\B))$. Fix $d \geq 1$ and  let   $A^{(1)}_N , \dots , A^{(d)}_N$  be $N \times N$ operator-valued  Wigner matrices defined by 
\begin{equation}\label{def:map_Wigner-mult}
A^{(k)}_N = \sum_{1\leq j \leq i \leq N} \big(e_{ij} \otimes a_{ij}^{(k)}+ e_{ij}^* \otimes a_{ij}^{(k)*} \big)
\end{equation}
where $e_{ii} = \frac{1}{2 \sqrt{N}} E_{ii}$ and $e_{ij}= \frac{1}{\sqrt{N}} E_{ij}$ for  $j<i$ and $(E_{ij})_{1\leq i,j \leq N}$ are the standard matrix units in $M_N(\C)$. Consider the family $S^{(1)}_N, \dots , S^{(d)}_N$ of centered operator-valued semicircular elements over $\D_N^\B := \D_N \otimes \B$, the subalgebra of diagonal matrices in $M_N(\C) \otimes \B$, and whose covariance is given by the map
\[
\eta: \D_N^\B \rightarrow M_d(\D_N^\B), \qquad  D \mapsto \big[\eta_{k,\ell} (D) \big]_{k,\ell=1}^d
\]
where for any $k,\ell =1 , \dots, d$ and any $  D=(d_{ij})_{i,j=1}^N \in \D_N^\B$, 
\[
\big(\eta_{k,\ell} (D) \big)_{i,j} = \delta_{i,j} \frac{1}{N} \sum_{r=1 }^i E [a^{(k)}_{ir} d_{rr} a^{(\ell)*}_{ir} ] + \delta_{i,j} \frac{1}{N} \sum_{r=i+1 }^N E [a^{(k)*}_{ri} d_{rr} a^{(\ell)}_{ri} ]. \]

\begin{theorem}\label{theo:Lin-Lin:Wigner}
Let $a=\{a_{ij}^{(k)} \mid 1\leq k \leq d , 1\leq j <i \leq N\}$ be a family such that the algebras $\A_{a_{ij}} = \B\langle a_{ij}^{(1)}, a_{ij}^{(1)*},\dots, a_{ij}^{(d)}, a_{ij}^{(d)*}\rangle$ are free with amalgamation over $\B$ for any $1\leq j\leq i\leq N$, and  such that
\begin{itemize}
    \item for any $k$ and $i$, $a^{(k)}_{ii}=a^{(k)*}_{ii}$ with $E[a^{(k)}_{ii}]=0$, 
    \item for any $k$ and  $j<i$,  $ E[a_{ij}^{(k)}]=0$,  $E[a_{ij}^{(k)} b a_{ij}^{(k)}] = 0$, for any $b \in \B$.
\end{itemize}
 We construct from $a$ a family of operator-valued Wigner matrices 
$A_N:= \big(A^{(1)}_N, \dots , A^{(d)}_N\big)$ as described in \eqref{def:map_Wigner-mult}.
Then the bounds in Theorem \ref{theo:Lin-Lin:matrices} hold for $\{A^{(1)}_N, \dots , A^{(d)}_N\}$ and $\{S^{(1)}_N, \dots , S^{(d)}_N\}$.
\end{theorem}
The proof of Theorem \ref{theo:Lin-Lin:Wigner} follows closely the lines of Theorem \ref{theo:Lin-Lin:matrices}. It will hence be omitted.

\begin{corollary}\label{coro:Lin-Lin:matrices-op-sc}
Let the setting be as in Theorem \ref{theo:Lin-Lin:Wigner} with $\B= \C$ and assume that for any $k, \ell=1, \dots, d$ and $i=1 , \dots, N$
\begin{equation}\label{cond:sc-element}
\sum_{j=1}^i \varphi   \big(a_{ij}^{(k)} a_{ij}^{(\ell)*}\big)  + \sum_{j=i+1}^N \varphi   \big(a_{ji}^{(k)*} a_{ji}^{(\ell)}\big) 
\quad
\text{is independent of  } i.  
\end{equation}
 Then  for any $N \in \bN$, $(S^{(1)}_N, \dots , S^{(d)}_N) $ has the same distribution as $ (s_1, \dots , s_d)$ where $s_1, \dots , s_d $ is a family of centered semicircular elements  and having covariance matrix 
$C:=(c_{k\ell})_{k,\ell=1}^d $ with $ c_{k,\ell}= \frac{1}{N} \sum_{j=1}^N\varphi   \big(a_{Nj}^{(k)} a_{Nj}^{(\ell)*}\big) $.  Moreover, if the entries of $C$ do not depend on $N$ and if $p$ is of degree $\deg(p)=r \geq 1$, then  the same estimate  as in \eqref{Kolm-bound} holds on the Kolmogorov distance.
\end{corollary}

Theorem \ref{theo:Lin-Lin:Wigner} follows from Theorem \ref{theo:Lin-Lin:freeness} and its proof a simple adaptation of that of Theorem \ref{theo:Lin-Lin:matrices}. Corollary \ref{coro:Lin-Lin:matrices-op-sc} is a direct consequence of Theorem \ref{theo:Lin-Lin:Wigner}. Therefore, the proofs are omitted.

\section{Useful lemmas and estimates}\label{section:useful-lemmas}

In this section, we illustrate useful lemmas that are essential for the proofs of several results and that could be of independent interest.  We start by the following basic algebraic identity:

\begin{lemma}\label{lem:Taylor_resolvents}
 Let $x$ and $y$ be invertible in some unital complex algebra $\A$, then for each $m\in\mathbb{N}_0$, the following identity holds:
\begin{align*}
x^{-1} - y^{-1} = \sum^m_{k=1} y^{-1} \big[ \big(y - x\big) y^{-1} \big]^k + x^{-1} \big[ \big(y - x\big) y^{-1} \big]^{m+1}.
\end{align*}
\end{lemma}

The following two lemmas will be crucial for bounding the remainder terms resulting from Lemma \ref{lem:Taylor_resolvents} when performing the Lindeberg method; we will use the notation introduced in \eqref{moment-map}.

\begin{lemma}\label{lemma:moment-bounds}
Let $(\A, E, \B)$ be an operator-valued $C^\ast$-probability space and let $x, y_1, y_2,y_3$ and $w$ be elements in $\A$ such that $\{x\}$ and $\{y_1,y_2,y_3\}$ are free with amalgamation over $\B$ with $E[x]=0$. Then,
\begin{multline*}
\| E [ y_1 x w x y_2 x y_3] \|^2 \\ \leq \|w\|^2 \|E[y_1 y_1^*]\| \, \|E[y_2^* y_2]\| \, \|E[y_3^* y_3]\| \|E[x x^*]\| \big( \sup \|m_4^{x^*,x^*,x,x} (b^*,1,b)\| +  \|E[x^* x]\|^2 \big),
\end{multline*}
where the supremum is over all $b \in \B$ with $\|b\|\leq 1$.
\end{lemma}

\begin{proof}
By \eqref{Cauchy-Schwarz} and \eqref{positivity}, we have 
\begin{align*}
\| E [y_1 x w x y_2 x y_3] \|^2 &\leq \| E [ y_3^*x^*y_2^*x^*x y_2 x y_3] \| \, \| E [ y_1 x w w^*x^*y_1^*]\|\\
                          &\leq \| w\|^2  \| E [ y_3^*x^*y_2^*x^*x y_2 x y_3] \| \, \| E [ y_1 x x^*y_1^*] \|. 
\end{align*}
By freeness, we have $E [ y_1 x x^* y_1^*] = E [ y_1 \, E[x x^*] \, y_1^*]$; with the help of \eqref{positivity}, we infer from the latter that
 \[
	\| E [ y_1 x x^* y_1^* ] \| =  \| E [ y_1 \, E[x x^*] \, y_1^*] \| \leq \|E[y_1 y_1^*]\| \, \|E[x x^*]\|. 
 \]
Again by freeness and the moment-cumulant decomposition in \eqref{operator-moment-cumulant-for}, we have 
\begin{align*}
E [ y_3^*x^*y_2^*x^* x y_2 x y_3] &= \sum_{\pi \in \NC(7)} \kappa_\pi^\B [y_3^*,x^*,y_2^*,x^*x,y_2,x,y_3]  \\&
 = \sum_{\pi_1 \in \NC(\{2,4,6\})} ( \kappa^\B_{\pi_1} \cup E_{\pi_1^c} ) [y_3^*,x^*,y_2^*,x^*x,y_2,x,y_3]
\end{align*}
where $\pi$ in the first sum has to decompose as $\pi=\pi_1\cup\pi_2$ with $\pi_1\in\NC(\{2,4,6\})$ and $\pi_2\in\NC(\{1,3,5,7\})$. 
Since we have $E[x]=0$, the only possibilities for $\pi_1$ that contribute to the last sum are $\{(2,6)(4)\}$ and $\{(2,4,6)\}$. Thus, we get
 \begin{align*}
 E [ y_3^*x^*y_2^*x^* x y_2 x y_3] \!
 &\!= E\big[y_3^*  \kappa^\B_3 \big[x^*  E[y_2^*], x^*x, E[y_2]  x \big] y_3 \big]+ 
 E \big[ y_3^*  \kappa^\B_2  \big[ x^*  E [ y_2^*  E[x^*x]  y_2], x\big] y_3\big] 
 \\&\!= E\big[y_3^*  E \big[x^*  E[y_2^*]   x^*x   E[y_2]  x \big] y_3 \big] 
\! -  E\big[y_3^*  E \big[x^*  E[y_2^*]  E[ x^*x]   E[y_2]  x \big] y_3 \big]
 \\ &\qquad +
 E \big[ y_3^* E \big[ x^*  E [ y_2^*  E[x^*x]  y_2]x \big] y_3\big] 
 \\& \leq E\big[y_3^*  E \big[x^*  E[y_2^*]   x^*x   E[y_2]  x \big] y_3 \big]  +
 E \big[ y_3^* E \big[ x^*  E [ y_2^*  E[x^*x]  y_2]x \big] y_3\big],
 \end{align*}
 where the last inequality follows from the fact that  $E\big[y_3^*  E \big[x^*  E[y_2^*]  E[ x^*x]   E[y_2]  x \big] y_3 \big]$ is positive. Therefore, since $\|E[y_2]\|^2 \leq \|E[y_2^* y_2]\|$ by \eqref{Cauchy-Schwarz}, we get
\begin{align*}
\| E [ y_3^*x^*y_2^*x^* x y_2 x y_3] \|
 &\leq \|E[y_3^* y_3]\| \big( \|E[y_2]\|^2  \sup \|m_4^{x^*,x^*,x,x} (b^*,1,b)\| + \|E[y_2^* y_2]\| \|E[x^* x]\|^2 \big)\\
 &\leq \|E[y_2^* y_2]\| \, \|E[y_3^* y_3]\| \big( \sup \|m_4^{x^*,x^*,x,x} (b^*,1,b)\| +  \|E[x^* x]\|^2 \big),
 \end{align*}
where the above supremums are over all $b \in \B$ such that $\|b\|\leq 1$. Putting the above bounds together, we end the proof.
\end{proof}

Further, we will need the following variant of Lemma \ref{lemma:moment-bounds} which holds in the case of an operator-valued $W^\ast$-probability space.

\begin{lemma}\label{lemma:moment-bounds_variant}
Let $(\A, \varphi, E, \B)$ be an operator-valued $W^\ast$-probability space and suppose that $x$, $y_1,y_2,y_3$, and $w$ are elements in $\A$ such that $\{x\}$ and $\{y_1,y_2,y_3\}$ are free with amalgamation over $\B$ with $E[x]=0$. Then
\begin{multline*}
| \varphi ( y_1 x w x y_2 x y_3 ) |^2\\ \leq \| w\|^2  \|y_1\|_{L^2}^2 \|y_3\|_{L^2}^2 \|E[y^*_2 y_2]\| \, \|E[x x^*]\| \, \big( \sup \|m_4^{x^*,x^*,x,x} (b^*,1,b)\| + \|E[x^* x]\|^2 \big),
\end{multline*}
where the supremum is over all $b \in \B$ with $\|b\|\leq 1$.
\end{lemma}

\begin{proof}
The ordinary Cauchy-Schwarz inequality together with the positivity of $\varphi$ yield that
\begin{align*}
| \varphi (y_1 x w x y_2 x y_3) |^2 &\leq \varphi ( y^*_3 x^* y^*_2 x^* x y_2 x y_3 ) \, \varphi ( y_1 x w w^* x^* y^*_1 ) \\
                                    &\leq \| w\|^2 \varphi ( y^*_3 x^* y^*_2 x^* x y_2 x y_3 ) \, \varphi ( y_1 x x^* y^*_1 ).
\end{align*}
Like in the proof of Lemma \ref{lemma:moment-bounds}, we infer with the help of the moment cumulant formula that $E [ y_1 x x^* y^*_1] = E [ y_1 \, E[x x^*] \, y^*_1 ]$, from which we conclude that
\[
\varphi ( y_1 x x^* y^*_1 ) \leq \|y_1\|_{L^2}^2 \|E[x x^*]\|
\]
by using $\varphi \circ E = \varphi$ and the positivity of $\varphi$. In the proof of Lemma \ref{lemma:moment-bounds}, we have also seen that
\begin{align*}
E [ y^*_3 x^* y^*_2 x^* x y_2 x y_3]
&= E\big[y^*_3  E \big[x^*  E[y^*_2] x^* x E[y_2] x \big] y_3 \big] 
+ E \big[ y^*_2 E \big[ x^*  E [ y^*_2  E[x^*x]  y_2]x \big] y_3\big],
 \end{align*}
from which we derive after applying $\varphi$ that
\begin{align*}
\varphi ( y^*_3 x^* y^*_2 x^* x y_2 x y_3 )
&= \varphi\big(y^*_3 E \big[x^*  E[y^*_2]   x^*x   E[y_2]  x \big] y_3 \big)  + \varphi \big( y^*_3 E \big[ x^*  E [ y^*_2  E[x^*x]  y_2] x \big] y_3 \big)
 \end{align*}
and finally, by the triangle inequality and again the positivity of $\varphi$,
\begin{align*}
\varphi ( y^*_3 x^* y^*_2 x^* x y_2 x y_3)
&\leq \|y_3\|_{L^2}^2 \Big( \| E \big[x^*  E[y^*_2]  x^*x  E[y_2]  x\big] \| + \|E \big[ x^*  E [ y^*_2  E[x^*x]  y_2] x \big]\| \Big).
\end{align*}
Proceeding like in the proof of Lemma \ref{lemma:moment-bounds}, we obtain
\[
\varphi ( y^*_3 x^* y^*_2 x^* x y_2 x y_3 ) \leq \|y_3\|_{L^2}^2 \|E[y^*_2 y_2]\| \big( \sup \|m_4^{x^*,x^*,x,x} (b^*,1,b)\| +  \|E[x^* x]\|^2 \big),
\]
where the above supremums are over all $b \in \B$ such that $\|b\|\leq 1$. Putting the bounds together, we arrive at the inequality asserted in the lemma.
\end{proof}

Finally, we need the following result which allows us to control the moments of a selfadjoint linear pencil by the moments of its entries.

\begin{lemma}\label{lem:tensor_moments}
Let $(\A, E, \B)$ be an operator-valued $C^\ast$-probability space and let $m\in\bN$. Then for all $Q_1,\dots,Q_d \in M_m(\C)$ and for each $d$-tuple $x=(x_1,\dots,x_d)$ of operators in $\A$, the linear pencil $Q(x) = \sum^d_{\ell=1} (Q_\ell \otimes x_\ell + Q_\ell^\ast \otimes x_\ell^\ast)$ satisfies for all $k\geq 0$ that
\[
\|m_k^{Q(x)}\| \leq m^{2(k-1)} \Big( \sum^d_{\ell=1} \|Q_\ell\| \Big)^k \sum_{\epsilon_1,\dots,\epsilon_k \in \{1,\ast\}} \max_{1\leq\ell_1,\dots,\ell_k\leq d} \|m_k^{x_{\ell_1}^{\epsilon_1},\dots,x_{\ell_k}^{\epsilon_k}}\|
\]
with norms defined like in \eqref{eq:multilinear-map-norm}.
In particular, we have
\[
\|m_2^{Q(x)} (\1)\| \leq \Big( \sum^d_{\ell=1} \|Q_\ell\| \Big)^2 \sum_{\epsilon_1,\epsilon_2 \in \{1,\ast\}} \max_{1\leq\ell_1,\ell_2\leq d} \|m_2^{x_{\ell_1}^{\epsilon_1},x_{\ell_2}^{\epsilon_2}} (\1)\|.
\]
\end{lemma}

\begin{proof}
We start by recalling from \eqref{eq:multilinear-map-norm} that $\|m_k^{Q(x)}\| = \sup \|m_k^{Q(x)}(B^1,\dots,B^{k-1})\|$ where the supremum is taken over all matrices $B^1,\dots,B^{k-1} \in M_m(\B)$ such that $\|B^1\|, \ldots, \|B^{k-1}\|\leq 1$.
For any $B^1,\dots,B^{k-1} \in M_m(\B) \cong M_m(\C) \otimes \B$, we write
\begin{multline}\label{moments-linear-pencil-expression}
m_k^{Q(x)}(B^1,\dots,B^{k-1})
\\
=  \sum_{1\leq i_1, j_1, \ldots, i_{k-1},j_{k-1}\leq m} \sum_{\substack{1 \leq \ell_1,\dots,\ell_k \leq d\\ \epsilon_1,\dots,\epsilon_k \in \{1,\ast\}}} Q_{\ell_1}^{\epsilon_1} E_{i_1j_1} Q_{\ell_2}^{\epsilon_2} \cdots E_{i_{k-1}j_{k-1}} Q_{\ell_k}^{\epsilon_k} \otimes E[x_{\ell_1}^{\epsilon_1} B^1_{i_1j_1} x_{\ell_2}^{\epsilon_2} \cdots B^{k-1}_{i_{k-1}j_{k-1}} x_{\ell_k}^{\epsilon_k}], 
\end{multline}
where $(E_{ij})_{1\leq i,j \leq m}$ are the $m \times m$ standard matrix units. Hence we get
\begin{align*}
\|m_k^{Q(x)}(B^1,\dots,B^{k-1})\| & \leq \Bigg[ \prod_{s=1}^{k-1} \Big(\sum_{i, j=1}^m \|B^s_{ij} \| \Big)\Bigg] \Big( \sum^d_{\ell=1} \|Q_\ell\| \Big)^k \sum_{\epsilon_1,\dots,\epsilon_k \in \{1,\ast\}}  \max_{1\leq\ell_1,\dots,\ell_k\leq d} \|m_k^{x_{\ell_1}^{\epsilon_1},\dots,x_{\ell_k}^{\epsilon_k}}\|
\\& \leq m^{2(k-1)} \|B^1\| \cdots \|B^{k-1}\| \Big( \sum^d_{\ell=1} \|Q_\ell\| \Big)^k \! \! \! \sum_{\epsilon_1,\dots,\epsilon_k \in \{1,\ast\}}  \max_{1\leq\ell_1,\dots,\ell_k\leq d} \|m_k^{x_{\ell_1}^{\epsilon_1},\dots,x_{\ell_k}^{\epsilon_k}}\|.
\end{align*}
Similarly, we prove the bound for the particular case $\|m_2^{Q(x)} (\1)\|$. 
\end{proof}

The following lemma provides some strengthening of Lemma \ref{lem:tensor_moments} in the particular case of diagonal amplifications.

\begin{lemma}\label{lem:tensor_moments_diagonal}
Let $(\A,E,\B)$ be an operator-valued $C^\ast$-probability space and let $m\in\bN$. Then, for all $x=x^\ast \in \A$, we have
$$
\|m_k^{\1_m \otimes x}\| \leq m^{k-1} \|m_k^x\| \quad\text{all integers $k\geq 1$} \qquad\text{and}\qquad \alpha_4(\1_m \otimes x) \leq m^2 \|m_4^x\|.$$
\end{lemma}

\begin{proof}
Notice that the first inequality is obvious for $k=1$; thus, it suffices to consider the case $k\geq 2$. To this end, take $B^1,\dots,B^{k-1} \in M_m(\B)$. Then, we have that
$$m_k^{\1_m \otimes x}(B^1,\dots,B^{k-1}) = \sum_{1\leq i_1, \ldots, i_k  \leq m} E_{i_1i_k} \otimes E[x B^1_{i_1 i_2} x \cdots B^{k-1}_{i_{k-1}i_k} x]$$
and hence, thanks to \cite[Exercise 3.10 (i)]{Paulsen},
$$\|m_k^{\1_m \otimes x}(B^1,\dots,B^{k-1})\| \leq \Bigg(\sum_{1\leq i_1,i_k\leq m} \bigg\|\sum_{1\leq i_2, \ldots, i_{k-1} \leq m} E[x B^1_{i_1 i_2} x \cdots B^{k-1}_{i_{k-1} i_k} x]\bigg\|^2\Bigg)^{1/2}.$$
Because
$$\bigg\|\sum_{1\leq i_2, \ldots, i_{k-1} \leq m} E[x B^1_{i_1 i_2} x \cdots B^{k-1}_{i_{k-1} i_k} x]\bigg\| \leq m^{k-2} \|m_k^x\| \|B^1\| \cdots \|B^{k-1}\|,$$
we infer from the latter that $\|m_k^{\1_m \otimes x}(B^1,\dots,B^{k-1})\| \leq m^{k-1} \|m_k^x\| \|B^1\| \cdots \|B^{k-1}\|$. This yields $\|m_k^{\1_m \otimes x}\| \leq m^{k-1} \|m_k^x\|$, as asserted.\\

Similarly, we prove for $B\in M_m(\B)$ that $\|m_4^{\1_m \otimes x}(B^\ast,\1_m \otimes \1,B)\| \leq m^2 \|m_4^x\| \|B\|^2$, which leads us to the asserted bound $\alpha_4(\1_m \otimes x) \leq m^2 \|m_4^x\|$.
\end{proof}

\section{Proof of the operator-valued setting}\label{section:proof-operator-valued}

The aim of this section is to prove Theorem \ref{theo1:Lin-Lin:freeness} and Theorem \ref{theo:opval_semicirculars_comparison}. The proofs rely on an operator-valued version of the Lindeberg method extending the one in \cite{Ba-Ce-18} to operator-valued Cauchy transforms and refining some estimates to get bounds in terms of moments instead of the operator norm.

\begin{proof}[Proof of Theorem \ref{theo1:Lin-Lin:freeness}]
We start by writing, for any $\Zb \in \bH^+(\B)$, the difference as a telescoping sum:  
\begin{align*}
G_{\ux_n}  (\Zb)  - G_{\uy_n}(\Zb)
 &= \sum^n_{i=1} \big(G_{\uz_i}(\Zb) - G_{\uz_{i-1}}(\Zb)\big) 
= \sum^n_{i=1} \Big( \big( G_{\uz_i}(\Zb) -G_{\uz^0_i}(\Zb) \big)
 -  \big( G_{\uz_{i-1}}(\Zb)   -G_{\uz^0_i}(\Zb) \big) \Big)
\end{align*}
 where for any $i= 1, \ldots, n$,
 \[
 \uz_i =\sum_{j=1}^i x_j + \sum_{j=i+1}^n y_j
\qquad \text{and} \qquad 
  \uz^0_i =\sum_{j=1}^{i-1} x_j + \sum_{j=i+1}^n y_j.
 \]
Noting that $\uz_n=\ux_n$, $\uz_0=\uy_n$, $\uz_i-  \uz^0_i = x_i$ and $\uz_{i-1}- \uz^0_i = y_i$ for $i=1,\dots,n$. This shows that in the above telescoping sum, we are replacing only one variable at a time; at each step $i$, we replace $x_i$ by $y_i$. We will then control the  error terms and sum over $i=1,\dots,n$. This is the so-called Lindeberg method which is also known in probability theory as the replacement trick.

Applying the algebraic identity in Lemma \ref{lem:Taylor_resolvents} up to order $3$, we get 
\begin{align*}
G_{\uz_i}(\Zb) - G_{\uz^0_i}(\Zb)  
&= G_{\uz^0_i}(\Zb) \ x_i \  G_{\uz^0_i}(\Zb)
 +G_{\uz^0_i}(\Zb) \big( x_i \  G_{\uz^0_i}(\Zb)\big)^2
+G_{\uz_i}(\Zb) \big( x_i \  G_{\uz^0_i}(\Zb)\big)^3 \qquad\text{and}\\
G_{\uz_{i-1}}(\Zb) - G_{\uz^0_i}(\Zb)  
&= G_{\uz^0_i}(\Zb) \ y_i \ G_{\uz^0_i}(\Zb)
+G_{\uz^0_i}(\Zb) \big( y_i \  G_{\uz^0_i}(\Zb)\big)^2
+G_{\uz_{i-1}}(\Zb)\big(  y_i \  G_{\uz^0_i}(\Zb)\big)^3. 
\end{align*}
We assume without loss of generality that $\{x_1, \ldots , x_n\}$ and $\{y_1, \ldots , y_n\}$ are free over $\B$. Now fixing $i \in [n]$, we note that $E[x_i]= E[y_i]=0$ and that $x_i$ and $y_i$ are free from $\uz^0_i$ with amalgamation over $\B$. Therefore, by freeness, we get
\begin{align*}
E \big[G_{\uz^0_i}(\Zb) \ x_i \  G_{\uz^0_i}(\Zb)  \big] 
&= E \big[G_{\uz^0_i}(\Zb) \ E[ x_i ] \  G_{\uz^0_i}(\Zb)  \big] =0 \qquad\text{and}\\
E  \big[G_{\uz^0_i}(\Zb)  \ y_i \  G_{\uz^0_i}(\Zb) \big] 
&= E \big[G_{\uz^0_i}(\Zb)  \ E [ y_i ] \  G_{\uz^0_i}(\Zb)  \big] =0. 
\end{align*}
Moreover, as  $E[x_i b x_i] = E[y_i b y_i]$ for any $b \in \B$, we get again, by $E[x_i]= E[y_i]=0$ and the moment-cumulant formula in \eqref{operator-moment-cumulant-for}, that
\begin{multline*}
E \big[ G_{\uz_i^0}(\Zb)\  x_i \  G_{\uz_i^0}(\Zb) \ x_i \  G_{\uz_i^0}(\Zb) \big]
  = E \big[ G_{\uz_i^0}(\Zb) \ E \big[ x_i \  E[ G_{\uz_i^0}(\Zb)] \ x_i  \big]\  G_{\uz_i^0}(\Zb)  \big]
\\  = E \big[ G_{\uz_i^0}(\Zb) \ E \big[\  y_i \  E[ G_{\uz_i^0}(\Zb) ] \ y_i    \big]\  G_{\uz_i^0}(\Zb)  \big]
 =E \big[ G_{\uz_i^0}(\Zb) \ y_i \  G_{\uz_i^0}(\Zb) \ y_i \  G_{\uz_i^0}(\Zb)\big].
\end{multline*}
Therefore, for any $\Zb\in \bH^+(\B)$,
\[
E \big[ G_{\ux_n}(\Zb) \big] - E \big[ G_{\uy_n}(\Zb) \big] = \sum^n_{i=1} \big( E[W_i(\Zb)] - E[\widetilde{W}_i(\Zb)] \big),
\]
where we abbreviate
$
W_i(\Zb) := G_{\uz_i}(\Zb) \big( x_i \  G_{\uz^0_i}(\Zb)\big)^3$ and $\widetilde{W}_i(\Zb) := G_{\uz_{i-1}}(\Zb) \big( y_i \  G_{\uz^0_i}(\Zb)\big)^3.
$

\vspace{0.2cm}

\paragraph*{\bf Proof of \eqref{eq:theo1_Lin-Lin:freeness-1} } To prove the estimate on the norm of the operator-valued Cauchy transform, we write for any $\Zb \in \bH^+(\B)$,
\[
\| E [G_{\ux_n}(\Zb)] - E [G_{\uy_n}(\Zb)] \| \leq \sum^n_{i=1} \big( \|E [W_i(\Zb)] \| + \|E [\widetilde{W}_i(\Zb)]\| \big).
\]
Noting that $G_{\uz_i}(\Zb) \ x_i \ G_{\uz^0_i}(\Zb) = G_{\uz^0_i}(\Zb) \ x_i \ G_{\uz_i}(\Zb)$, we get by Lemma \ref{lemma:moment-bounds} and using \eqref{def:second-fourth-moment} that
\[
\|E [W_i(\Zb)] \|  \leq \|\Im(\Zb)^{-1}\|^4 \sqrt{\alpha_2(x)\big(\alpha_4 (x) +  \alpha_2(x)^2\big)}\, .
\]
Similarly, we bound the second term and get
\[
\|E [\widetilde{W}_i(\Zb)] \|   \leq \| \Im (\Zb)^{-1}\|^4  \sqrt{  \alpha_2(y) \big(  \alpha_4(y)+  \alpha_2(y)^2 \big)} 
 = \| \Im (\Zb)^{-1}\|^4  \sqrt{  
 \alpha_2(x) \big( \alpha_4(y)+  \alpha_2(x)^2 \big) } \, .
\]
Collecting the above bounds and summing over $i = 1, \dots, n$, we end the proof of \eqref{eq:theo1_Lin-Lin:freeness-1}, the first estimate of the theorem.

\vspace{0.2cm}

\paragraph*{\bf Proof of \eqref{eq:theo1_Lin-Lin:freeness-2}} Now to prove the estimate on the scalar-valued Cauchy transform, we set $\Zb=z\1$ with $z \in \mathbb{C}^+$ (but we will omit the $\1$ in the following for better legibility) and note that
\[
\varphi [G_{\ux_n}(z)] - \varphi [G_{\uy_n}(z)] = \sum_{i=1}^n  \varphi[W_i(z)] + \varphi[\widetilde{W}_i(z)].
\]
Then by H\"older inequality we get 
\begin{align*}
|\varphi[W_i(z)]| & =\big|\varphi\big[ G_{\uz_i}(\Zb)  x_i   G_{\uz^0_i}(\Zb)  x_i   G_{\uz^0_i}(\Zb)x_i G_{\uz^0_i}(\Zb)\big] \big|
\\&\leq \| G_{\uz_i}(\Zb)  x_i   G_{\uz^0_i}(\Zb)\|_{L^3} \| x_i   G_{\uz^0_i}(\Zb)\|_{L^3} \| x_i G_{\uz^0_i}(\Zb)\|_{L^3}
\\& \leq \big \| G_{\uz_i}(z) \|  \| G_{\uz^0_i}(z) \|^3 \|x_i\|_{L^3}^3 \leq \frac{1}{\Im (z)^4} \|x_i\|_{L^3}^3.
\end{align*}
Similarly we bound the term $\varphi[\widetilde{W}_i(z)]$, collect the above bounds and sum over $i$ to get our second estimate \eqref{eq:theo1_Lin-Lin:freeness-2}.

\vspace{0.2cm}

\paragraph*{\bf Proof of \eqref{eq:theo1_Lin-Lin:freeness-3} } Again, noting that $G_{\uz_i}(\Zb) \ x_i \ G_{\uz^0_i}(\Zb) = G_{\uz^0_i}(\Zb) \ x_i \ G_{\uz_i}(\Zb)$, we finally get by a direct application of Lemma \ref{lemma:moment-bounds_variant},
\begin{align*}
|\varphi[W_i(z)]| &\leq \frac{1}{\Im(z)^2} \|G_{\uz^0_i}(z)\|_{L^2}^2 \sqrt{\alpha_2(x) \big( \alpha_4(x)+  \alpha_2(x)^2 \big)}\, ,\\
|\varphi[\widetilde{W}_i(z)]| &\leq \frac{1}{\Im(z)^2} \|G_{\uz^0_i}(z)\|_{L^2}^2 \sqrt{\alpha_2(x) \big( \alpha_4(y)+  \alpha_2(x)^2 \big)}\, .
\end{align*}
Putting these facts together and involving \eqref{eq:Cauchy-integral}, we arrive at \eqref{eq:theo1_Lin-Lin:freeness-3} and end the proof.
\end{proof}

\begin{proof}[Proof of Theorem \ref{theo:opval_semicirculars_comparison}]
Fix $n\in\bN$ and let $x=\{x_j \mid 1\leq j \leq n\}$ and $y = \{y_j \mid 1 \leq j \leq n\}$ consist of $\B$-freely independent copies of the given operator-valued semicircular elements $\frac{1}{\sqrt{n}} S_0$ and $\frac{1}{\sqrt{n}} S_1$, respectively, such that $x$ and $y$ are also free. We follow now the lines of the proof of Theorem \ref{theo1:Lin-Lin:freeness}. 
As the families $x$ and $y$ are centered, then clearly the first order terms are zero as before. However, as they do not have matching second moments, the second order terms will not cancel and we get 
\begin{align*}
E &\big[ G_{\uz_i^0}(\Zb) \ x_i E[ G_{\uz_i^0}(\Zb)]   x_i \  G_{\uz_i^0}(\Zb) \big]
-E \big[ G_{\uz_i^0}(\Zb) \ y_i E[ G_{\uz_i^0}(\Zb)]   y_i \  G_{\uz_i^0}(\Zb) \big]
\\& =E \Big[ G_{\uz_i^0}(\Zb) \Big( E \big[ x_i E[ G_{\uz_i^0}(\Zb)]   x_i \big]  - E \big[ y_i E[ G_{\uz_i^0}(\Zb)]  y_i \big] \Big) G_{\uz_i^0}(\Zb) \Big]
\\& =
\frac{1}{n}E \big[ G_{\uz_i^0}(\Zb) \ (\eta_0-\eta_1)\big( E[ G_{\uz_i^0}(\Zb)] \big) \  G_{\uz_i^0}(\Zb) \big],
\end{align*}
where we have used the fact that for any $b \in \B$,  $E[x_i b x_i]=n^{-1}E[S_0 b S_0]= n^{-1} \eta_0(b)$ and  $E[y_i b y_i]=n^{-1}E[S_1 b S_1]= n^{-1} \eta_1(b)$.
As for the third order terms, we have by Lemma \ref{lemma:moment-bounds} that
\begin{IEEEeqnarray*}{lClCl}
\|E [W_i(\Zb)] \| &\leq &\|\Im(\Zb)^{-1}\|^4 \sqrt{\alpha_2(x)\big(\alpha_4 (x) +  \alpha_2(x)^2\big)} &\leq &C n^{-3/2} \|\Im(\Zb)^{-1}\|^4,\\
\|E [\widetilde{W}_i(\Zb)] \| &\leq &\|\Im(\Zb)^{-1}\|^4 \sqrt{\alpha_2(y)\big(\alpha_4 (y) +  \alpha_2(y)^2\big)} &\leq &C n^{-3/2} \|\Im(\Zb)^{-1}\|^4,
\end{IEEEeqnarray*}
for all $\Zb\in\bH^+(\B)$, where $C>0$ is some constant independent of $n$. Hence
\begin{align*}
\lefteqn{\| \G^\B_{S_0}(\Zb) - \G^\B_{S_1}(\Zb)\| = \| E [G_{\ux_n}(\Zb)] - E [G_{\uy_n}(\Zb)] \|}\nonumber\\
&\leq \frac{1}{n} \sum^n_{i=1} \big\|E \big[ G_{\uz_i^0}(\Zb) \ (\eta_0-\eta_1)\big( E[ G_{\uz_i^0}(\Zb)] \big) \  G_{\uz_i^0}(\Zb) \big] \big\| + \sum^n_{i=1} \big( \|E [W_i(\Zb)] \| + \|E [\widetilde{W}_i(\Zb)]\| \big)\nonumber \\
&\leq \|\Im(\Zb)^{-1}\|^3 \|\eta_0-\eta_1\| + \frac{2C}{\sqrt{n}} \|\Im(\Zb)^{-1}\|^4.
\end{align*}
By letting $n\to \infty$, we obtain the asserted bound \eqref{eq:opval_semicirculars_comparison_Cauchy} in the case $k=1$, namely
\begin{equation}\label{eq:opval_semicirculars_comparison_Cauchy_k=1}
\| \G^\B_{S_0}(\Zb) - \G^\B_{S_1}(\Zb)\| \leq \|\Im(\Zb)^{-1}\|^3 \|\eta_0-\eta_1\| \qquad\text{for all $\Zb\in\bH^+(B)$}.
\end{equation}
To get the assertion \eqref{eq:opval_semicirculars_comparison_Cauchy} for general $k\in\bN$, we apply
\eqref{eq:opval_semicirculars_comparison_Cauchy_k=1}
to the operator-valued semicircular elements $\1_k \otimes S_1$ and $\1_k \otimes S_0$ in the operator-valued $C^\ast$-probability space $(M_k(\A), \id_k \otimes E, M_k(\B))$; 
this yields that for all $\Zb \in \bH^+(M_k(\B))$
\[
\| \G^\B_{\1_k \otimes S_0}(\Zb) - \G^\B_{\1_k \otimes S_1}(\Zb)\| \leq  \|\Im(\Zb)^{-1}\|^3 \|\id_k \otimes \eta_0 - \id_k \otimes \eta_1\|.
\]
By involving that $\| \id_k \otimes \eta_0 - \id_k \otimes \eta_1 \| \leq k \| \eta_0 - \eta_1 \|$, which follows from \cite[Exercise 3.10]{Paulsen}), we arrive at the bound given in \eqref{eq:opval_semicirculars_comparison_Cauchy} for general $k$.

To prove \eqref{eq:opval_semicirculars_comparison_integral}, we use Lemma \ref{lemma:moment-bounds_variant} that yields for $z\in\C^+$ that
\begin{IEEEeqnarray*}{lClCl}
|\varphi[W_i(z)]| &\leq &\frac{1}{\Im(z)^2} \|G_{\uz^0_i}(z)\|_{L^2}^2 \sqrt{\alpha_2(x) \big( \alpha_4(x) + \alpha_2(x)^2 \big)} &\leq &C n^{-3/2} \frac{1}{\Im(z)^2} \|G_{\uz^0_i}(z)\|_{L^2}^2,\\
|\varphi[\widetilde{W}_i(z)]| &\leq &\frac{1}{\Im(z)^2} \|G_{\uz^0_i}(z)\|_{L^2}^2 \sqrt{\alpha_2(y) \big( \alpha_4(y) + \alpha_2(y)^2 \big)} &\leq &C n^{-3/2} \frac{1}{\Im(z)^2} \|G_{\uz^0_i}(z)\|_{L^2}^2,
\end{IEEEeqnarray*}
with the same constant $C>0$ as above, and hence
\begin{align*}
\lefteqn{|\G_{S_0}(z) - \G_{S_1}(z)| = |\G_{x_n}(z) - \G_{y_n}(z)|}\\
&\leq \frac{1}{n} \sum^n_{i=1} \varphi\big( G_{\uz_i^0}(\Zb) \ (\eta_0-\eta_1)\big( E[ G_{\uz_i^0}(\Zb)] \big) \  G_{\uz_i^0}(\Zb) \big) + \sum^n_{i=1} \big(|\varphi[W_i(z)]| +  |\varphi[\widetilde{W}_i(z)]| \big)\\
&\leq \Big(\frac{1}{\Im(z)} \|\eta_0-\eta_1\| + \frac{1}{\sqrt{n}} \frac{2C}{\Im(z)^2} \Big) \frac{1}{n} \sum^n_{i=1} \|G_{\uz^0_i}(z)\|_{L^2}^2.  
\end{align*}
Using this for $z=t+i\epsilon$, integrating over $t\in\R$ by taking \eqref{eq:Cauchy-integral} into consideration, and letting $n\to \infty$, we obtain \eqref{eq:opval_semicirculars_comparison_integral}.
Finally, we involve \eqref{eq:Levy_bound} to deduce from the latter that for all $\epsilon>0$
$$L(\mu_{S_0},\mu_{S_1}) \leq 2\sqrt{\frac{\epsilon}{\pi}} + \frac{1}{\epsilon^2} \|\eta_0-\eta_1\|.$$
Optimizing over $\epsilon \in (0,\infty)$ yields \eqref{eq:opval_semicirculars_comparison_Levy}.
\end{proof}

\subsection{Application to linear matrix pencils: proof of Theorem \ref{theo:Linear-pencil:freeness}}\label{section:proof-LP}
In this section, we prove the multivariate case when linear matrix functions are chosen as test functions. We shall see that it follows directly from the operator-valued setting in Theorem \ref{theo1:Lin-Lin:freeness} by simply adjusting the setting to the operator-valued $W^\ast$-probability space $(M_m(\A), \tr_m \otimes \varphi, M_m(\B), E_m)$ with $E_m := \id_m \otimes E$.

\begin{proof}[Proof of Theorem \ref{theo:Linear-pencil:freeness}]
Note that without loss of generality, we assume $Q_0=0$. For $ Q_1, \ldots, Q_d \in M_m(\mathbb{C})$, the linear matrix pencil in $\ux_n$ and $\uy_n$ are given by
\begin{align*}
 \begin{aligned}
 g(\ux_n) &= 
 Q_1 \otimes \sum_{j=1}^n x_j^{(1)} +  Q_1^* \otimes \sum_{j=1}^n x_j^{(1)*} + \cdots + Q_d \otimes \sum_{j=1}^n x_j^{(d)} +  Q_d^* \otimes \sum_{j=1}^n x_j^{(d)*}
 := \sum_{j=1}^n X_j,
 \\  g(\uy_n) &= 
 Q_1 \otimes \sum_{j=1}^n y_j^{(1)} +  Q_1^* \otimes \sum_{j=1}^n y_j^{(1)*} + \cdots + Q_d \otimes \sum_{j=1}^n y_j^{(d)} +  Q_d^* \otimes \sum_{j=1}^n y_j^{(d)*}
 := \sum_{j=1}^n Y_j,
 \end{aligned}
\end{align*}
where $X_j= \sum_{\ell=1}^d \big(Q_\ell \otimes  x_j^{(\ell)} +  Q_\ell^* \otimes  x_j^{(\ell)*}\big)$ and $Y_j= \sum_{\ell=1}^d \big(Q_\ell \otimes  y_j^{(\ell)} +  Q_\ell^* \otimes  y_j^{(\ell)*}\big)$. Now as freeness lifts up to matrices, the free independence assumption on the family $x$ passes to $X=\{X_j \mid 1\leq j\leq n\}$. More precisely,  $X_1,\ldots , X_n$ are free with amalgamation over $M_m(\B)$ while the summands of each $X_j$ may be correlated. Furthermore, Assumption \ref{A:general} lifts to the family $Y=\{Y_j \mid 1\leq j\leq n\}$. Namely, $Y_1,\ldots , Y_n$ are free with amalgamation over $M_m(\B)$ and for any $j=1 , \ldots , n$,
\[
E_m[X_j]=E_m[Y_j]=0 
\quad \text{and} \quad 
E_m[X_j b X_j]=E_m[Y_jb Y_j] 
\quad \text{for all $b \in M_m(\B)$}.
\]
Now that all assumptions of Theorem \ref{theo1:Lin-Lin:freeness} hold for $X$ and $Y$, we directly deduce the desired estimates by adjusting them to the $W^\ast$-probability space $(M_m(\A), \tr_m \otimes \varphi,  E_m, M_m(\B))$.

\vspace{0.2cm}

\paragraph*{\bf Proof of \eqref{eq:theo_Linearpencil-1}} By \eqref{eq:theo1_Lin-Lin:freeness-1}, we have for any $\Zb \in \bH^+(M_m(\B))$,
\begin{align*}
\big\| E_m [G_{g(\ux_n)}(\Zb)] - E_m[G_{g(\uy_n)}(\Zb)] \big\| 
&= \big\| E_m [G_{ \sum_{j=1}^n X_j}(\Zb)] - E_m[G_{ \sum_{j=1}^n Y_j}(\Zb)] \big\| 
\\&
\leq \|\Im(\Zb)^{-1}\|^4 A_1(X,Y) n,
\end{align*}
with $A_1(X,Y)=\sqrt{\alpha_2(X)} \big(\sqrt{\alpha_4(X) +  \alpha_2(X)^2} + \sqrt{\alpha_4(Y) + \alpha_2(X)^2}\, \big).$
Now by Lemma \ref{lem:tensor_moments}, we have 
\[
\begin{aligned}
 &\alpha_2(X)= \max_{1 \leq j \leq n}  \|E_{m}[X_j^2]\| \leq \Big( \sum^d_{\ell=1} \|Q_\ell\| \Big)^2 \beta_2(x), 
 \\&\alpha_4(X)=\max_{1 \leq j \leq n} \sup \|m_4^{X_j}(b^*,\1,b)\| \leq m^6 \Big( \sum^d_{\ell=1} \|Q_\ell\| \Big)^4 \beta_4(x),
 \\ & \alpha_4(Y)= \max_{1 \leq j \leq n} \sup \|m_4^{Y_j}(b^*,\1,b)\|  \leq m^6 \Big( \sum^d_{\ell=1} \|Q_\ell\| \Big)^4 \beta_4(y),
\end{aligned}
\]
with the supremum taken over all $b \in M_m(\B)$ such that $\| b\| \leq 1$ and $\beta_2(x)$, $\beta_4(x)$ and $\beta_4(y)$ as defined in \eqref{def:betas}. Putting these bounds together, we get that $A_1(X,Y)\leq C_g B_1(x,y)$ with $C_g=m^3\big(\sum_{\ell=1}^d\|Q_\ell\|\big)^3$; hence, the desired bound in \eqref{eq:theo_Linearpencil-1}.

\vspace{0.2cm}

\paragraph*{\bf Proof of \eqref{eq:theo_Linearpencil-2}} In the operator-valued $W^\ast$-probability space $(M_m(\A), \tr_m \otimes \varphi, E_m, M_m(\B))$, we have by \eqref{eq:theo1_Lin-Lin:freeness-2} for any $z\in \mathbb{C}^+$ that 
\begin{align*}
\big| (\tr_m \otimes \varphi) [G_{g(\ux_n)}(z)] - (\tr_m \otimes \varphi)[G_{g(\uy_n)}(z)] \big| &\leq \frac{1}{\Im (z)^4} \big( \|X\|_{L^3}^3 + \|Y\|_{L^3}^3 \big) n
\\ & \leq d^3 \max_{1\leq \ell \leq d} \|Q_\ell\|^3 \, \frac{1}{\Im (z)^4}  \big( \|x\|_{L^3}^3 + \|y\|_{L^3}^3 \big) n.
\end{align*}

\vspace{0.2cm}

\paragraph*{\bf Proof of \eqref{eq:theo_Linearpencil-3}} By applying \eqref{eq:theo1_Lin-Lin:freeness-3} with respect to $\tr_m \otimes \varphi$, we get for every $\epsilon>0$, 
\[
\frac{1}{\pi} \int_\R \big| (\tr_m \otimes \varphi) [G_{g(\ux_n)}(t+i\epsilon)] - (\tr_m \otimes \varphi) [G_{g(\uy_n)}(t+i\epsilon)] \big|\, \mathrm{d} t
\leq \frac{1}{\epsilon^3} A_1(X,Y) n
  \leq C_g\frac{1}{\epsilon^3} B_1(x,y) n,
\]
which proves the last estimate in the theorem.
\end{proof}

\changelocaltocdepth{1}

 \section{Proof of the multivariate setting: noncommutative polynomials}\label{section:proof-Polynomials}
This section is dedicated to prove Theorem~\ref{theo:Lin-Lin:freeness} in which we consider noncommutative polynomials as test functions. The proof is based on linearization techniques for operator-valued noncommutative polynomials. An overview with definitions, properties and main results on linearizations can be found in Appendix \ref{section:linearizations}.

Consider a polynomial $p\in \B\langle x_1,x_1^*,\dots,x_d, x_d^*\rangle$ for which we fix, all along this section, a selfadjoint linear representation $\rho= (v^*, Q, v)$ of dimension $m$, where $v$ is a column vector in $\C^m$ and $Q \in  M_{m}(\B\langle x_1,x_1^*,\dots,x_d,x_d^*\rangle )$  is an $m\times m$ $\B$-valued affine linear pencil of the particular form 
\[
Q( x_1,x_1^*,\dots,x_d,x_d^*) = Q_0 \otimes \1 + \sum_{\ell=1}^d \big(Q_\ell \otimes x_\ell + Q_\ell^* \otimes x_\ell^*\big) 
\]
with $Q_0=Q_0^* \in M_m(\B)$ and $Q_1,\dots,Q_d \in M_{m}(\C) \subseteq M_m(\B)$ are complex-valued matrices. Then consider the linearization matrix and write for any $\Zb \in \bH^+(\B)$,  
\begin{equation}\label{Linearisation-formula}
L_{p(x_1, x_1^*, \dots ,x_d, x_d^*)}(\Zb) = (\Lambda(\Zb) + \widehat{Q}_0) \otimes \1 + \widehat{Q}_1 \otimes x_1 + \widehat{Q}_1^* \otimes x_1^*+ \dots + \widehat{Q}_d \otimes x_d + \widehat{Q}_d^* \otimes x_d^*
\end{equation}
with $\Lambda(\Zb)\in M_{m+1}(\mathbb{\B})$ and $ \widehat{Q}_0, \dots , \widehat{Q}_d \in M_{m+1}(\C)$  given by 
\[
\Lambda (\Zb) = \begin{bmatrix} \Zb & 0\\ 0 & 0 \end{bmatrix}, \quad
\widehat{Q}_0 = \begin{bmatrix} 0 & v^*\\ v & -Q_0 \end{bmatrix} \quad\text{and}\quad \widehat{Q}_j  = \begin{bmatrix} 0 & 0\\ 0 & -Q_j\end{bmatrix} \quad
\text{for all  } j=1, \ldots,d.
\]
By Theorem \ref{thm:linearization}, we have for any $\Zb \in \bH^+(\B)$,  
\[
G_{p(\cdot)}(\Zb)=(\Zb - p(\cdot))^{-1} = \big( L_{p(\cdot)}(\Zb)^{-1}\big)_{11},
\]
and therefore, 
\begin{equation}\label{eq:Cauchy-linearization}
E \big[G_{p(\cdot)}(\Zb)\big] = \big( E_{m+1} \big[L_{p(\cdot)}(\Zb)^{-1}\big]\big)_{11},
\end{equation}
where we recall that $E_{m+1} := (\id_{m+1} \otimes E)$.

Note that $E_{m+1}[L_{p(\cdot)}(\Zb)^{-1}]$ is not an operator-valued Cauchy-transform as $\Lambda(\Zb)$ does not belong to $\bH^+(M_{m+1}(\B))$ but rather lies on its boundary; this prevents a direct reduction of the multivariate to the single-variable operator-valued free CLT. To overcome this hurdle, the approach of \cite{Mai-Speicher-13} originating in \cite{Sp-07} was to move $\Lambda(\Zb)$ into $\Lambda_\epsilon(\Zb) \in \bH^+(M_{m+1}(\B))$ by filling up the diagonal of $\Lambda(\Zb)$ by $i\epsilon\1$ for sufficiently small $\epsilon>0$. Bounds on how much $E_{m+1} [L_{p(\cdot)}(\Zb)^{-1}]$ deviates from the Cauchy-transform $E_{m+1} [G_{\hat{p}_\rho(\cdot)}(\Lambda_\epsilon(\Zb))]$ in terms of $\epsilon$ were used to build a bridge to the operator-valued free CLT, but this resulted in bounds far from the optimal order $\frac{1}{\sqrt{n}}$ and valid only at points near $\infty$. While improvements of the linearization techniques would allow an extension of the said bounds to the entire upper half-plane and also the comparison result between $E_{m+1} [L_{p(\cdot)}(\Zb)^{-1}]$ and $E_{m+1}[G_{\hat{p}_\rho(\cdot)}(\Lambda_\epsilon(\Zb))]$ has been strengthened after \cite{Mai-Speicher-13} (see the appendix of \cite{Banna-Mai-18}), the problem of getting a rate of convergence below the optimal one remained open until now. Our solution rests on the operator-valued Lindeberg method by blocks and a careful examination of the terms emerging by linearization.

\subsection{Operator-valued Lindeberg method by blocks} With linearizations, linearity is recovered and freeness with amalgamation over $\B$ is lifted to $M_{m+1}(\B)$. By lifting the operator-valued Lindeberg method to the level of the linearization matrix $L_{p}$ of $p$, we recover the desired estimates on the Cauchy transforms of $p(\cdot)$ by controlling the $(1,1)$ entry of  $E_{m+1} \big[L_{p(\cdot)}(\Zb)^{-1}\big]$. However, one still needs to group the correlated components together and replace them all at once. This is shown in the following proposition which is also of independent interest.

\begin{proposition}\label{prop:Lin-Lin}
Let $x=(x_j^{(k)})_{1\leq j\leq n, \, 1\leq k\leq d}$ and $y=(y_j^{(k)})_{1\leq j\leq n, \, 1\leq k\leq d}$ be two families in $\A$. For all  $i =1, \dots, n$,  let $x_i:= (x_i^{(1)},$ $ x_i^{(1)*}, $  $\dots ,x_i^{(d)},$ $ x_i^{(d)*})$ and $y_i:= (y_i^{(1)}, y_i^{(1)*}, \dots ,y_i^{(d)}, y_i^{(d)*})$ and set
\[
\uz_i= \sum_{j=1}^ix_j + \sum_{j=i+1}^n y_j
\qquad \text{and} \qquad
\uz_i^0= \sum_{j=1}^{i-1}x_j + \sum_{j=i+1}^n y_j  .
\]
Then for any $\Zb\in \bH^+(\B)$,
\[
E [G_{p(\uz_n)}(\Zb)] - E [G_{p(\uz_0)}(\Zb)] = \sum_{i=1}^n 
\big((A_i)_{11}+(B_i)_{11}+(C_i)_{11}\big),
\]
where, for any $i \in [n]$,
\begin{align*}
A_i &=  E_{m+1} \Big[ L_{p(\uz_i^0)}(\Zb)^{-1} \  \X_i \ L_{p(\uz_i^0)}(\Zb)^{-1} \Big] - E_{m+1} \Big[ L_{p(\uz_i^0)}(\Zb)^{-1} \  \Y_i \ L_{p(\uz_i^0)}(\Zb)^{-1} \Big], \\
B_i &= E_{m+1} \Big[ L_{p(\uz_i^0)}(\Zb)^{-1} \ \big(\X_i \  L_{p(\uz_i^0)}(\Zb)^{-1} \big)^2\Big] -E_{m+1} \Big[ L_{p(\uz_i^0)}(\Zb)^{-1} \ \big( \Y_i \  L_{p(\uz_i^0)}(\Zb)^{-1}\big)^2 \Big] ,
\\
C_i &= E_{m+1} \Big[ L_{p(\uz_i)}(\Zb)^{-1} \ \big(\X_i \  L_{p(\uz_i^0)}(\Zb)^{-1} \big)^3\Big] -E_{m+1} \Big[ L_{p(\uz_{i-1})}(\Zb)^{-1} \ \big( \Y_i \  L_{p(\uz_i^0)}(\Zb)^{-1}\big)^3 \Big] ,
\end{align*}
with $\X_i:= \sum_{\ell=1}^d \big(\widehat{Q}_\ell \otimes x_i^{(\ell)} + \widehat{Q}_\ell^* \otimes x_i^{(\ell)*} \big)$ and  $\Y_i:= \sum_{\ell=1}^d  \big(\widehat{Q}_\ell \otimes y_i^{(\ell)} + \widehat{Q}_\ell^* \otimes y_i^{(\ell)*}\big)$.
\end{proposition}

\begin{proof}[Proof of Proposition \ref{prop:Lin-Lin}]
For the fixed linearization matrix, we have by \eqref{eq:Cauchy-linearization} for any $\Zb \in \bH^+(\B)$,  
\[
E \big[ G_{p(\uz_n)}(\Zb) \big] - E \big[ G_{p(\uz_0)}(\Zb) \big] = \big(E_{m+1} \big[L_{p(\uz_n)}(\Zb)^{-1} - L_{p(\uz_0)}(\Zb)^{-1} \big]\big)_{11}.
\]
We work now on the matrix level and apply the operator-valued Lindeberg method to control $ E_{m+1} \big[L_{p(\uz_n)}(\Zb)^{-1}\big] - E_{m+1}\big[L_{p(\uz_0)}(\Zb)^{-1} \big]$ where we replace, at each step $i$,  $X_i$ by $Y_i$. This allows to replace, at one step, all the correlated elements $\widehat{Q}_\ell \otimes x_i^{(\ell)}$ and their adjoints by the corresponding $\widehat{Q}_\ell \otimes y_i^{(\ell)}$ and their adjoints. This is why we refer to this replacement method as an operator-valued Lindeberg method by blocks. To do so, we start by writing the difference as a telescoping sum: 
\begin{align*}
L_{p(\uz_n)} (\Zb)^{-1}  - L_{p(\uz_0)}(\Zb)^{-1} 
 &= \sum^n_{i=1} \big( L_{p(\uz_i)}(\Zb)^{-1} - L_{p(\uz_{i-1})}(\Zb)^{-1} \big) 
\\&= \sum^n_{i=1} \Big( \big( L_{p(\uz_i)}(\Zb)^{-1} - L_{p(\uz^0_i)}(\Zb)^{-1}  \big)
 -  \big( L_{p(\uz_{i-1})}(\Zb)^{-1}  -L_{p(\uz^0_i)}(\Zb)^{-1}\big) \Big)
\end{align*}
Noting that  $L_{p(\uz_i)}(\Zb) -  L_{p(\uz^0_i)}(\Zb) = \X_i$ and $L_{p(\uz_{i-1})}(\Zb) -  L_{p(\uz^0_i)}(\Zb) = \Y_i$, then applying the algebraic identity in Lemma \ref{lem:Taylor_resolvents} up to order $3$, we get 
\begin{multline*}
L_{p(\uz_i)}(\Zb)^{-1} - L_{p(\uz^0_i)}(\Zb)^{-1}  
= L_{p(\uz^0_i)}(\Zb)^{-1} \ \X_i \  L_{p(\uz^0_i)}(\Zb)^{-1}
\\  +L_{p(\uz^0_i)}(\Zb)^{-1} \Big( \X_i \  L_{p(\uz^0_i)}(\Zb)^{-1}\Big)^2
+L_{p(\uz_i)}(\Zb)^{-1} \Big( \X_i \  L_{p(\uz^0_i)}(\Zb)^{-1}\Big)^3
\end{multline*}
and 
\begin{multline*}
L_{p(\uz_{i-1})}(\Zb)^{-1} - L_{p(\uz^0_i)}(\Zb)^{-1}  
= L_{p(\uz^0_i)}(\Zb)^{-1} \ \Y_i \ L_{p(\uz^0_i)}(\Zb)^{-1}
\\+L_{p(\uz^0_i)}(\Zb)^{-1} \Big( \Y_i \  L_{p(\uz^0_i)}(\Zb)^{-1}\Big)^2
 +L_{p(\uz_{i-1})}(\Zb)^{-1}\Big(  \Y_i \  L_{p(\uz^0_i)}(\Zb)^{-1}\Big)^3. 
\end{multline*}
The assertion follows by taking the difference, summing over $i=1,\dots,n$ and taking the conditional expectation. 
\end{proof}

\subsection{Bounds on the operator norm}
 In order to prove Theorem \ref{theo:Lin-Lin:freeness}, we would need to control first the operator norm of the inverse of the linearization matrices appearing in the telescoping sum. Let us first define the quantity
 \begin{equation}\label{linearization-constant-2}
M_{\rho;x,y} := \|Q_0^{-1}\| \sum^{r_\rho}_{k=0} \big( c_\rho M_{x,y}  \big)^k
\end{equation}
where $c_\rho := 2d \max_{1 \leq \ell \leq d} \|Q_\ell Q_0^{-1}\|$,  $M_{x,y} := \|x\| + \|y\| + 2 \sqrt{n} \sqrt{\| E [xx^*] \|}+2\sqrt{n} \sqrt{\| E [x^*x] \|}$  and $r_\rho$ is a non-negative integer associated with $\rho$ as by Proposition \ref{prop:Q_inverse}. 
\begin{lemma}\label{lem:estimates-L-inverse}
Assume that the families $x$ and $y$ are as in Theorem \ref{theo:Lin-Lin:freeness}.  Fix $i \in [n] $ and let $\uz_i$ and $\uz^0_i$ be as defined in Proposition \ref{prop:Lin-Lin}.  Then, uniformly in $n$,
\begin{equation}\label{eq:estimates-L-inverse-1}
\max \{ \|Q(\uz_i)^{-1}\|, \|Q(\uz_i^0)^{-1}\| \} \leq M_{\rho;x,y},
\end{equation}
and moreover, for any $\Zb \in \bH^+(\B)$,
\begin{equation}\label{eq:estimates-L-inverse-2}
\max \{ \|L_{p(\uz_i)}(\Zb)^{-1}\|, \|L_{p(\uz^0_i)}(\Zb)^{-1}\| \} \leq M_{\rho;x,y} + (1 + v^*v M_{\rho;x,y}^2) \| \Im (\Zb)^{-1}\|.
\end{equation} 
\end{lemma}

\begin{proof}
In view of Lemma \ref{lem:norm-L_inverse}, applied for the given linear representation $\rho=(v^*,Q,v)$, \eqref{eq:estimates-L-inverse-2} is an immediate consequence of \eqref{eq:estimates-L-inverse-1}; thus, it suffices to verify the bound \eqref{eq:estimates-L-inverse-1}. 
By Proposition \ref{prop:Q_inverse}, there exists a finite number $r_\rho$ depending on the chosen linear representation $\rho$ such that
\[
\|Q(\uz_i)^{-1}\| \leq \|Q_0^{-1}\| \sum^{r_\rho}_{k=0} (d C_1C_2)^k\] 
where 
$$C_1 := \max_{1\leq \ell \leq d} \|Q_\ell Q_0^{-1}\|  \qquad \text{and} \qquad C_2 := \max_{1\leq k \leq d} \big\| \sum_{j=1}^i x_j^{(k)} + \sum_{j=i+1}^n y_j^{(k)} \big\|.$$
It is obvious that $C_1$ is bounded and does not depend on $n$. Now, fix $k \in [d]$ and recall that $x^{(k)}_1 , \ldots , x^{(k)}_n$ are free with amalgamation over $\B$. Hence, by \cite[Proposition 7.1]{Junge}, which is an operator-valued extension of Voiculescu's estimate \cite[Lemma 3.2]{Voi-86}, we get 
\begin{align*}
\Big\|\sum_{j=1}^i x_j^{(k)} \Big\| 
&\leq \max_{1\leq j \leq i} \|x_j^{(k)} \| + \Big\| \sum_{j=1}^i E [x_j^{(k)*}x_j^{(k)}] \Big\|^{1/2} + \Big\| \sum_{j=1}^i E [x_j^{(k)}x_j^{(k)*}] \Big\|^{1/2}
\\ & \leq \max_{1\leq j \leq n} \|x_j^{(k)} \|  + \sqrt{n}  \max_{1\leq j \leq n} \big\| E [x_j^{(k)*}x_j^{(k)}]\big\|^{1/2} +  \sqrt{n}  \max_{1\leq j \leq n} \big\| E [x_j^{(k)}x_j^{(k)*}] \big\|^{1/2}.
\end{align*}
In a similar way, we control $\big\|\sum_{j=i+1}^n y_j^{(k)} \big\|$. Finally by taking into account the matching second-order moments in Assumption \ref{A:general}, and taking the maximum over $k \in [d]$, we prove that $C_2 \leq  M_{x,y}$ with $M_{x,y} = \|x\| + \|y\| +  2\sqrt{n} \sqrt{\| E [xx^*] \|}+ 2\sqrt{n} \sqrt{\| E [x^*x] \|}$. Analogously, we control $\|Q(\uz_i^0)^{-1}\|$ and end the proof.
\end{proof}

\subsection{Proof of the main result}
We are now ready to prove Theorem \ref{theo:Lin-Lin:freeness}. We start by applying Proposition \ref{prop:Lin-Lin} which gives for any $\Zb\in \bH^+(\B)$,
\[
E [G_{p(\ux)}(\Zb)] - E [G_{p(\uy)}(\Zb)] = \sum_{i=1}^n 
\big((A_i)_{11}+(B_i)_{11}+(C_i)_{11}\big),
\]  
and we adopt the same notation therein.We assume, without loss of generality, that the families $x = \{ x^{(k)}_j \mid 1 \leq j \leq n, 1 \leq k \leq d\}$ and $y = \{ y^{(k)}_j \mid 1 \leq j \leq n, 1 \leq k \leq d\}$ are also free with amalgamation over $\B$. Fixing $i\in [n]$, we proceed to prove that the first and second order terms $A_i$ and $B_i$  are zero.  We first note that freeness with amalgamation is conserved when lifted to matrices. Hence, as $\uz_i^{0}$ is free  from $\A_{x_i}$ and $\A_{y_i}$ over $\B$, then $L_{p(\uz^0_i)}(\Zb)^{-1}$ is free from $\X_i$ and $\Y_i$ over $M_{m+1}(\B)$. Also as $E[x_i^{(\ell)} ]=E[y_i^{(\ell)} ]=0$, then $E_{m+1}[ \X_i ]=E_{m+1}[ \Y_i ]=0$. Therefore, freeness and the moment-cumulant formula \eqref{operator-moment-cumulant-for} yield
\begin{align*}
E_{m+1} \big[L_{p(\uz^0_i)}(\Zb)^{-1} \ \X_i \  L_{p(\uz^0_i)}(\Zb)^{-1}  \big] 
= E_{m+1} \big[L_{p(\uz^0_i)}(\Zb)^{-1} \ E_{m+1}[ \X_i ] \  L_{p(\uz^0_i)}(\Zb)^{-1}  \big] =0,
\end{align*}
\begin{align*}
E_{m+1} \big[L_{p(\uz^0_i)}(\Zb)^{-1} \ \Y_i \  L_{p(\uz^0_i)}(\Zb)^{-1} \big] 
= E_{m+1} \big[L_{p(\uz^0_i)}(\Zb)^{-1} \ E_{m+1}[ \Y_i ] \  L_{p(\uz^0_i)}(\Zb)^{-1}  \big] =0, 
\end{align*}
and hence $A_i=0$. Now, as for all $\ell_1, \ell_2 \in [d]$, $\epsilon_1 , \epsilon_2 \in \{1,*\}$ and  $b \in \B$,  $E [x_k^{(\ell_1),\epsilon_1} b x_k^{(\ell_2),\epsilon_2} ]$ $ =E [ y_k^{(\ell_1),\epsilon_1} b   y_k^{(\ell_2),\epsilon_2}]$ then $E_{m+1}[ \X_i B \X_i ]=E_{m+1}[ \Y_i B \Y_i ]$ for any $B \in M_{m+1}( \B)$. By freeness and the fact that $E_{m+1}[ \X_i ]=E_{m+1}[ \Y_i ]=0$, we get by the moment-cumulant formula \eqref{operator-moment-cumulant-for} that 
\begin{align*}
E_{m+1} \Big[ L_{p(\uz_i^0)}(\Zb)^{-1} &\ \X_i \  L_{p(\uz_i^0)}(\Zb)^{-1} \ \X_i \  L_{p(\uz_i^0)}(\Zb)^{-1} \Big]
\\ & = E_{m+1} \Big[ L_{p(\uz_i^0)}(\Zb)^{-1} \ E_{m+1} \big[ \X_i \  E_{m+1}[ L_{p(\uz_i^0)}(\Zb)^{-1}] \ \X_i  \big]\  L_{p(\uz_i^0)}(\Zb)^{-1}  \Big]
\\ & = E_{m+1} \Big[ L_{p(\uz_i^0)}(\Zb)^{-1} \ E_{m+1} \big[\  \Y_i \  E_{m+1}[ L_{p(\uz_i^0)}(\Zb)^{-1} ] \ \Y_i    \big]\  L_{p(\uz_i^0)}(\Zb)^{-1}  \Big]
\\& =E_{m+1} \Big[ L_{p(\uz_i^0)}(\Zb)^{-1} \ \Y_i \  L_{p(\uz_i^0)}(\Zb)^{-1} \ \Y_i \  L_{p(\uz_i^0)}(\Zb)^{-1}\Big],
\end{align*}
which yields that $B_i = 0$. Therefore, we get for any $\Zb\in \bH^+(\B)$,
\begin{equation}\label{eq:Lin-Lin:freeness-Lindeberg}
E \big[ G_{p(\ux)}(\Zb) \big] - E \big[ G_{p(\uy)}(\Zb) \big] 
= \sum_{i=1}^n \big( (E_{m+1}[W_i])_{11} - (E_{m+1}[\widetilde{W}_i])_{11} \big)
\end{equation}
with 
$
W_i(\Zb) :=
  L_{p(\uz_i)}(\Zb)^{-1} \ \big(\X_i \  L_{p(\uz_i^0)}(\Zb)^{-1} \big)^3
$ and $\widetilde{W}_i(\Zb):=  L_{p(\uz_{i-1})}(\Zb)^{-1} \ \big( \Y_i \  L_{p(\uz_i^0)}(\Zb)^{-1}\big)^3.$

\vspace{0.2cm}
\paragraph{\bf Proof of \eqref{oper:Lin-Lin:free}} We deduce from \eqref{eq:Lin-Lin:freeness-Lindeberg} that, for any $\Zb \in \bH^+(\B)$,
\[
\| E [G_{p(\ux)}(\Zb)] - E [G_{p(\uy)}(\Zb)] \| 
\leq \sum^n_{i=1} \big( \|E_{m+1} [W_i(\Zb)] \| + \|E_{m+1} [\widetilde{W}_i(\Zb)]\|\big).
\]
Noting that $ L_{p(\uz_i)}(\Zb)^{-1} \X_i L_{p(\uz^0_i)}(\Zb)^{-1}= L_{p(\uz^0_i)}(\Zb)^{-1} \X_i L_{p(\uz_i)}(\Zb)^{-1}$, we get by Lemma \ref{lemma:moment-bounds}
\begin{align*}
\|E_{m+1} [W_i(\Zb)] \| 
& \leq \|  L_{p(\uz_i)}(\Zb)^{-1} \| \|L_{p(\uz^0_i)}(\Zb)^{-1}\|^3 \sqrt{\|E_{m+1} [X_i^2]\|}\sqrt{  \sup \|m_4^{X_i} (b^*,1,b)\| +  \|E_{m+1}[X_i^2]\|^2}
\end{align*}
where the supremum is over all $b \in M_{m+1}(\B)$ such that $\|b\| \leq 1$. By Lemma \ref{lem:tensor_moments}, we have 
\begin{align*}
 \|E_{m+1}[X_i^2]\| \leq \Big( \sum^d_{\ell=1} \|Q_\ell\| \Big)^2
\beta_2(x) \quad \text{and} \quad \sup \|m_4^{X_i}(b^*,1,b)\|  \leq (m+1)^6 \Big( \sum^d_{\ell=1} \|Q_\ell\| \Big)^4 \beta_4(x),
\end{align*}
with $\beta_2(x)$ and $\beta_4(x)$ as defined in \eqref{def:betas}.
Hence, together with the bounds in Lemma \ref{lem:estimates-L-inverse}, we get
\begin{align*}
\|E_{m+1} [W_i(\Zb)] \| 
 \leq C_{\rho; x,y} (\Zb) \sqrt{\beta_2(x)}\sqrt{\beta_4(x) +\beta_2^2(x)}
\end{align*}
with $ C_{\rho; x,y} (\Zb) := 8m^3  \big( \sum^d_{\ell=1} \|Q_\ell\| \big)^3 \big( M_{\rho;x,y} + (1 + v^*v M_{\rho;x,y}^2)  \| \Im (\Zb)^{-1}\|\big)^4  $ and $M_{\rho;x,y}$ as defined in \eqref{linearization-constant-2}. Note that for sufficiently large $n$, we have
\begin{align}\label{bound:Crho}
\max \big\{ M_{\rho;x,y},  1+ v^*v M_{\rho;x,y}^2 \big\} \leq C_p   M_{x,y}^{2r_\rho}
\end{align}
for some positive constants $C_p$ and $r_\rho$ that only depend on the linearization $\rho$ of $p$. Therefore,
\begin{align*}
C_{\rho; x,y} (\Zb) \leq C_p  M_{x,y}^{8r_\rho} \big(1 + \| \Im (\Zb)^{-1}\|\big)^4.
\end{align*}
Note that $C_p$ can change from one line to the other. Hence, we get for any $\Zb \in \bH^+(\B)$,
\begin{align*}
\|E_{m+1} [W_i(\Zb)] \| 
 \leq C_p  M_{x,y}^{8r_\rho} \big(1 + \| \Im (\Zb)^{-1}\|\big)^4 \sqrt{\beta_2(x)}\sqrt{\beta_4(x) +\beta_2^2(x)}.
\end{align*}
Similarly, we treat the second term and obtain for any $\Zb \in \bH^+(\B)$,
\begin{align*}
\|E_{m+1} [\widetilde{W}_i(\Zb)] \| 
 \leq C_p  M_{x,y}^{8r_\rho} \big(1 + \| \Im (\Zb)^{-1}\|\big)^4  \sqrt{\beta_2(x)}\sqrt{\beta_4(y) +\beta_2^2(x)}.
\end{align*}
Putting the above bounds together, we prove \eqref{oper:Lin-Lin:free}. 

\vspace{0.2cm}

\paragraph{\bf Proof of \eqref{scalar:Lin-Lin:free}} Starting by \eqref{eq:Lin-Lin:freeness-Lindeberg}, we have for any $z \in \C_+$,
\begin{equation}\label{eq:Lin-Lin:freeness-Lindeberg_scalar}
\varphi [G_{p(\ux)}(z)] - \varphi [G_{p(\uy)}(z)] = \sum_{i=1}^n \big( ( \id_{m+1} \otimes \varphi)(W_i(z))_{11} + ( \id_{m+1} \otimes \varphi)(\widetilde{W}_i(z))_{11} \big).
\end{equation}
By H\"older inequality and the estimate \eqref{eq:estimates-L-inverse-2} stated in Lemma \ref{lem:estimates-L-inverse}, we get
\begin{align*}
(& \id_{m+1}  \otimes  \varphi)(W_i(z))_{11}
\\&=  \sum_{i_1, \dots, i_6=1}^{m+1} 
\varphi \big[L_{p(\uz_i)}(z)^{-1}_{1i_1} \ (\X_i )_{i_1i_2} L_{p(\uz^0_i)}(z)^{-1}_{i_2i_3} \  (\X_i )_{i_3i_4} L_{p(\uz^0_i)}(z)^{-1}_{i_4i_5} \ (\X_i )_{i_5i_6} L_{p(\uz^0_i)}(z)^{-1}_{i_61} \big]
 \\&\leq \sum_{i_1, \dots, i_6=1}^{m+1} 
\|  L_{p(\uz_i)}(\Zb)^{-1}_{1i_1} \, (\X_i )_{i_1i_2} \, L_{p(\uz^0_i)}(z)^{-1}_{i_2i_3} \|_{L^3} \| (\X_i )_{i_3i_4} \, L_{p(\uz^0_i)}(z)^{-1}_{i_4i_5} \|_{L^3}\|(\X_i )_{i_5i_6} \,L_{p(\uz^0_i)}(z)^{-1}_{i_61} \|_{L^3}
\\&  \leq  
2^3   \| L_{p(\uz_i)}(\Zb)^{-1}\| \|  L_{p(\uz^0_i)}(z)^{-1}\|^3  
 \max_{1\leq \ell \leq d} \| x_i^{(\ell)}\|_{L^3}^3 
 \sum_{i_1, \dots, i_6=1}^{m+1}  \sum_{\ell_1, \ell_2, \ell_3=1}^{d} | (\widehat{Q}_{\ell_1} )_{i_1i_2}| \, | (\widehat{Q}_{\ell_2} )_{i_3i_4}| \, |(\widehat{Q}_{\ell_3} )_{i_5i_6}|
\\& \leq  
2^3  (m+1)^6 \Big( \sum_{\ell=1}^d \|Q_\ell\|\Big)^3 \Big(M_{\rho;x,y} + (1 + v^*v M_{\rho;x,y}^2) \frac{1}{\Im(z)}\Big)^4   \| x\|_{L^3}^3 
\\ & \leq C_p  M_{x,y}^{8r_\rho} \Big(1 +\frac{1}{\Im (z)} \Big)^4 \|x\|_{L^3}^3.
\end{align*}
In the last inequality, we used the bound in \eqref{bound:Crho} which holds for sufficiently large $n$.
In an analogous way, we control the term $ ( \id_{m+1} \otimes \varphi )( \widetilde{W}_i(z) )_{11}$. Finally, by summing over $i=1 , \dots , n$ we get the bound in \eqref{scalar:Lin-Lin:free}.

\vspace{0.2cm}

\paragraph{\bf Proof of \eqref{scalar:Lin-Lin:free-Levy}} Again, we begin with \eqref{eq:Lin-Lin:freeness-Lindeberg} and proceed as follows. Consider the term
$$W_i(\Zb) = L_{p(\uz_i)}(\Zb)^{-1} (X_i  L_{p(\uz_i^0)}(\Zb)^{-1})^3 = L_{p(\uz_i^0)}(\Zb)^{-1} X_i  L_{p(\uz_i)}(\Zb)^{-1} (X_i  L_{p(\uz_i^0)}(\Zb)^{-1})^2.$$
In the sequel, we will denote by $E_{11}$ and $E_{22}$ the block matrix units
\[
E_{11} = \begin{bmatrix} \1 & 0\\ 0 & 0 \end{bmatrix} \qquad\text{and}\qquad E_{22} = \begin{bmatrix} 0 & 0\\ 0 & \1_m\end{bmatrix}
\]
in $M_{m+1}(\A)$. Note that $E_{22} X_i E_{22} = X_i$; thus
\begin{multline*}
E_{11} W_i(\Zb) E_{11}\\
= (E_{11} L_{p(\uz_i^0)}(\Zb)^{-1} E_{22}) X_i (E_{22} L_{p(\uz_i)}(\Zb)^{-1} E_{22}) X_i (E_{22} L_{p(\uz_i^0)}(\Zb)^{-1} E_{22}) X_i (E_{22} L_{p(\uz_i^0)}(\Zb)^{-1} E_{11})
\end{multline*}
Notice that according to the decomposition \eqref{eq:Schur_decomposition} 
$$E_{22} L_{p(\uz_i^0)}(\Zb)^{-1} E_{22} = \begin{bmatrix} 0 & 0\\ 0 & A_i^0(\Zb) \end{bmatrix} \qquad\text{and}\qquad E_{22} L_{p(\uz_i)}(\Zb)^{-1} E_{22} = \begin{bmatrix} 0 & 0\\ 0 & A_i(\Zb) \end{bmatrix}$$
with
\begin{align*}
A_i^0(\Zb) &:= Q(\uz_i^0)^{-1} v G_{p(\uz_i^0)}(\Zb) v^* Q(\uz_i^0)^{-1} - Q(\uz_i^0)^{-1},\\
A_i(\Zb) &:= Q(\uz_i)^{-1} v G_{p(\uz_i)}(\Zb) v^* Q(\uz_i)^{-1} - Q(\uz_i)^{-1}.
\end{align*}   
 Furthermore, we have that
$$E_{11} L_{p(\uz_i^0)}(\Zb)^{-1} E_{22} = \begin{bmatrix} 0 & - G_{p(\uz_i^0)}(\Zb) v^* Q(\uz_i^0)^{-1}\\ 0 & 0 \end{bmatrix}, \; E_{22} L_{p(\uz_i^0)}(\Zb)^{-1} E_{11} = \begin{bmatrix} 0 & 0\\ -Q(\uz_i^0)^{-1} v G_{p(\uz_i^0)}(\Zb) & 0 \end{bmatrix}.$$ 
Therefore, in summary, we get that
$$(W_i(\Zb))_{11} = -G_{p(\uz_i^0)}(\Zb) v^* Q(\uz_i^0)^{-1} \widetilde{Q}_1(x_i) A_i(\Zb) \widetilde{Q}_1(x_i) A_i^0(\Zb) \widetilde{Q}_1(x_i) Q(\uz_i^0)^{-1} v G_{p(\uz_i^0)}(\Zb),$$
with $ \widetilde{Q}_1(x_i) = Q(x_i) -Q_0 \otimes 1_\A =  \sum_{\ell=1}^d \big(Q_\ell \otimes x_i^{(\ell)} + Q_\ell^* \otimes x_i^{(\ell)*} \big)$.
Now, for any $a \in \A$, we denote by $diag(a)$ the diagonal matrix with $a$ as entries and we observe that
\begin{multline*}
\varphi\big[(W_i(\Zb))_{11}\big] = - m (\tr_m \otimes \varphi)\big[ diag(G_{p(\uz_i^0)}(\Zb)) Q(\uz_i^0)^{-1} 
\\ \cdot \widetilde{Q}_1(x_i) A_i(\Zb) \widetilde{Q}_1(x_i) A_i^0(\Zb) \widetilde{Q}_1(x_i) Q(\uz_i^0)^{-1}  v G_{p(\uz_i^0)}(\Zb) v^*\big].
\end{multline*}
We apply now Lemma \ref{lemma:moment-bounds_variant} in this setting of the operator-valued $W^\ast$-probability space $(M_m(\A), $ $ \tr_m \otimes \varphi, E_m, M_m(\B))$ with $E_m := \id_m \otimes E$. This yields
\begin{align*}
\lefteqn{\big| \big((\id \otimes \varphi)[W_i(\Zb)]\big)_{11} \big|^2 =  \big| \varphi \big[(W_i(\Zb))_{11}\big] \big|^2}\\
&\qquad\leq m^2 \| A_i(\Zb)\|^2 \| diag(G_{p(\uz_i^0)}(\Zb))Q(\uz_i^0)^{-1}\|_{L^2}^2 \|Q(\uz_i^0)^{-1} v G_{p(\uz_i^0)}(\Zb) v^*\|_{L^2}^2 \|E_m[ A_i^0(\Zb)^* A_i^0(\Zb)]\| \\
&\qquad\qquad\qquad \cdot \|E_m[\widetilde{Q}_1(x_i)^2]\| \, \Big( \sup \|m_4^{\widetilde{Q}_1(x_i)} (b^*,1,b)\| +  \|E_m[\widetilde{Q}_1(x_i)^2]\|^2 \Big),
\end{align*}
where the supremum is taken over all $b \in M_m(\B)$ such that $\|b\| \leq 1$. We proceed by controlling the above terms. We have
\begin{align*}
\|diag(G_{p(\uz_i^0)}(\Zb))Q(\uz_i^0)^{-1}\|_{L^2}^2
&\leq \|Q(\uz_i^0)^{-1}\|^2 \|G_{p(\uz_i^0)}(\Zb)\|_{L^2}^2
\end{align*}
and 
\begin{align*}
\|Q(\uz_i^0)^{-1} v G_{p(\uz_i^0)}(\Zb) v^*\|_{L^2}^2
 &\leq v^*v \|Q(\uz_i^0)^{-1}\|^2 (\tr_m \otimes \varphi)\big[v G_{p(\uz_i^0)}(\Zb)^* G_{p(\uz_i^0)}(\Zb) v^*\big]
\\& = \frac{1}{m} (v^*v)^2 \|Q(\uz_i^0)^{-1}\|^2 \|G_{p(\uz_i^0)}(\Zb)\|_{L^2}^2.
\end{align*}
Due to the bound \eqref{eq:estimates-L-inverse-1} in Lemma \ref{lem:estimates-L-inverse}, $ \|Q(\uz_i^0)^{-1}\| \leq M_{\rho;x,y}$ and hence  $\max \big\{ \|A_i(\Zb)\|, \|A_i^0(\Zb)\| \big\} \leq v^*v M_{\rho;x,y}^2 \|\Im(\Zb)^{-1}\| + M_{\rho;x,y} $. Together with the above estimates, we infer
\begin{multline*}
\big|\big((\id \otimes \varphi)[W_i(\Zb)]\big)_{11} \big|\\
\leq \tilde{c}_{\rho;x,y}(\Zb) \|G_{p(\uz_i^0)}(\Zb)\|_{L^2}^2 \|E_m[Q_1(x_i)^2]\|^{1/2} \Big(\sup \|m_4^{Q_1(x_i)}(b^\ast,1,b)\| + \|E_m[Q_1(x_i)^2]\|^2 \Big)^{1/2},
\end{multline*}
where $\tilde{c}_{\rho;x,y}(\Zb) = \frac{1}{\sqrt{m}} v^*v M_{\rho;x,y}^4 \big(1 + v^*v M_{\rho;x,y} \|\Im(\Zb)^{-1}\| \big)^2$. Then, we involve Lemma \ref{lem:tensor_moments} and \eqref{def:betas} to bound
\begin{align*}
\begin{aligned}
\|E_m[\widetilde{Q}_1(x_i)^2]\| & \leq  \Big( \sum^d_{\ell=1} \|Q_\ell\| \Big)^2 \beta_2(x),
 \quad \text{and}
\\ \sup \|m_4^{\widetilde{Q}_1(x_i)}(b^\ast,1,b)\| &\leq \|m_4^{\widetilde{Q}(x_i)}\| \leq m^6 \Big( \sum^d_{\ell=1} \|Q_\ell\| \Big)^4 \beta_4(x).
\end{aligned}
\end{align*}
Setting $c_{\rho;x,y}(\Zb)  =m^{5/2}v^*v \Big( \sum^d_{\ell=1} \|Q_\ell\| \Big)^3  M_{\rho;x,y}^4 \big(1+v^*v M_{\rho;x,y} \|\Im(\Zb)^{-1}\| \big)^2$ and collecting the above bounds, we infer that
$$\big| \big((\id \otimes \varphi)(W_i(\Zb))\big)_{11} \big| \leq c_{\rho; x,y}(\Zb) \|G_{p(\uz_i^0)}(\Zb)\|_{L^2}^2 \sqrt{\beta_2(x)} \sqrt{\beta_4(x) + \beta_2(x)^2}.$$
Now, for a sufficiently large $n$, $M_{\rho;x,y}\leq C_p M_{x,y}^{r_\rho}$ for some positive constant $C_p$. This yields that $ c_{\rho;x,y}(\Zb)  \leq C_p M_{x,y}^{6r_\rho} \big(1+ \|\Im(\Zb)^{-1}\| \big)^2$ and hence 
$$\big| \big((\id \otimes \varphi)(W_i(\Zb)\big)_{11} \big| \leq C_p M_{x,y}^{6r_\rho} \big(1+ \|\Im(\Zb)^{-1}\| \big)^2 \|G_{p(\uz_i^0)}(\Zb)\|_{L^2}^2 \sqrt{\beta_2(x)} \sqrt{\beta_4(x) +  \beta_2(x)^2}.$$  
In a similar way, we obtain the bound
$$\big| \big((\id \otimes \varphi)(\widetilde{W}_i(\Zb)\big)_{11} \big| \leq C_p M_{x,y}^{6r_\rho} \big(1+ \|\Im(\Zb)^{-1}\| \big)^2 \|G_{p(\uz_i^0)}(\Zb)\|_{L^2}^2 \sqrt{\beta_2(x)} \sqrt{\beta_4(y) +  \beta_2(x)^2}.$$
Specializing to $\Zb = (t+i\epsilon)\1$, summing over $i=1, \dots, n$, using \eqref{eq:Lin-Lin:freeness-Lindeberg_scalar}, and integrating over $t\in\R$ by using \eqref{eq:Cauchy-integral} produces then the bound asserted in \eqref{scalar:Lin-Lin:free}.

\appendix

\section{Linearizations of operator-valued noncommutative polynomials}  \label{section:linearizations}

In this section, we are going to give a brief introduction to linearization techniques. What became known in the free probability community as the ``linearization trick'' \cite{Voi-95,HT05,HST06,And13,And15,B-M-S-17,P18} turned out
some years ago to be used extensively also in other branches of mathematics, ranging from system engineering over automata theory to the theory of non-commutative rings; see \cite{HMS18,MS18} and the references collected therein.
Let us point out that, to a great extent, linearization works equally well for noncommutative rational functions. For noncommutative polynomials, however, linearizations have some additional features which are crucial for our purpose. Thus, we restrict in the following exposition from the beginning to the case of noncommutative polynomials.

Roughly speaking, the ``linearization trick'' allows a translation of non-linear problems involving noncommutative polynomials into linear problems about linear but matrix-valued polynomials. While not clear at first sight, this translation brings indeed a significant simplification, since it can be combined effectively with tools from operator-valued free probability.

For that purpose, however, it is of particular importance that such linearizations are able to preserve selfadjointness. While in the past, mostly polynomials in selfadjoint operators were considered, we will present here a more general version of the ``selfadjoint linearization trick'' that works equally well for polynomials in not necessarily selfadjoint operators. Moreover, instead of scalar-valued noncommutative polynomials, we will consider here polynomials with operator-valued coefficients.

\subsubsection*{\bf Definition and basic properties}

In this subsection, we present the selfadjoint version of the linearization trick for operator-valued noncommutative polynomials. Throughout the following, let $\B$ be a unital complex $\ast$-algebra and denote by $\B\langle x_1,\dots,x_d\rangle$ the complex unital algebra of noncommutative $\B$-valued polynomials in the non-commuting indeterminates $x=(x_1,\dots,x_d)$; more explicitly, $\B\langle x_1,\dots,x_d\rangle$ is the $\C$-linear span of all \emph{$\B$-valued monomials}, i.e., expressions of the form
$$b_0 x_{i_1} b_1 x_{i_2} b_2 \cdots b_{k-1} x_{i_k} b_k$$
for integers $k\geq 0$, indices $1\leq i_1,\dots,i_k \leq d$, and elements $b_0,b_1,\dots,b_k\in\B$.

Furthermore, we assume that $\B\langle x_1,\dots,x_d\rangle$ is endowed with an anti-linear involution $\ast$ with respect to which it forms a $\ast$-algebra. There are two prototypical instances that will be important for us:
\begin{itemize}
 \item on $\B\langle x_1,\dots,x_d\rangle$, we have the canonical involution which declares each $x_i$ to be selfadjoint; it thus satisfies for monomials $$(b_0 x_{i_1} b_1 x_{i_2} b_2 \cdots b_{k-1} x_{i_k} b_k)^\ast = b_k^\ast x_{i_k} b_{k-1}^\ast \cdots b_2^\ast x_{i_2} b_1^\ast x_{i_1}b_0^\ast.$$
 \item on $\B\langle y_1,y_1^\ast,\dots,y_d,y_d^\ast\rangle$, which is the algebra of all noncommutative polynomials in the $2d$ non-commuting variables $y_1,y_1^\ast,\dots,y_d,y_d^\ast$, we declare each $y_i^\ast$ to be the adjoint of $y_i$ and vice versa; we thus have for monomials
$$(b_0 y_{i_1}^{\epsilon_1} b_1 y_{i_2}^{\epsilon_2} b_2 \cdots b_{k-1} y_{i_k}^{\epsilon_k} b_k)^\ast = b_k^\ast y_{i_k}^{\epsilon_k\ast} b_{k-1}^\ast \cdots b_2^\ast y_{i_2}^{\epsilon_2\ast} b_1^\ast y_{i_1}^{\epsilon_1\ast}b_0^\ast,$$ where $\epsilon_1,\dots,\epsilon_n\in\{1,\ast\}$ are subjected to the rules $\ast \ast = 1$ and $1 \ast = \ast$.
\end{itemize}
We point out that of course also $\B\langle x_1, \dots, x_{d_1}, y_1, y_1^\ast, \dots, y_{d_2}, y_{d_2}^\ast \rangle$ with $d=d_1+2d_2$ non-commuting variables, which is an amalgamated free product of two algebras of the above types, fits into our framework; thus, we cover in particular the setting of \cite{EKY18}.

Note that the involution $\ast$ on $\B\langle x_1,\dots,x_d\rangle$ induces an involution on $M_m(\B\langle x_1,\dots,x_d\rangle)$ for every integer $m\geq 1$, with respect to which it forms a $\ast$-algebra as well. To be more precise, the involution on $M_m(\B\langle x_1,\dots,x_d\rangle)$, which we denote again by $\ast$, is given by $P^\ast = (p_{l,k}^\ast)_{k,l=1}^m$ for every matrix $P = (p_{k,l})_{k,l=1}^m$ in $M_m(\B\langle x_1,\dots,x_d\rangle)$.

By a \emph{$\B$-valued affine linear pencil of size $m \times m$}, we mean an element of $M_m(\B\langle x_1,\dots,x_d\rangle)$ which is of the form $Q = Q_0 + Q_1 x_1 + \dots + Q_d x_d$ with matrices $Q_0,Q_1,\dots,Q_d \in M_m(\B)$; sometimes, we identify $Q$ with $Q_0 \otimes 1 + Q_1 \otimes x_1 + \dots + Q_d \otimes x_d$ in $M_m(\B) \otimes \C\langle x_1,\dots,x_d\rangle$.

The tools that enable us to handle noncommutative $\B$-valued polynomials efficiently are the so-called linearizations. Those objects can be built out of linear representations, which we thus introduce first. 

\begin{definition}[Linear representations] 
Let $p\in \B\langle x_1,\dots,x_d\rangle$ be any noncommutative $\B$-valued polynomial. A \emph{linear representation of $p$} is a triple $\rho=(u,Q,v)$ such that 
\[p = -u Q^{-1} v,\]
where $u$ and $v$ are respectively a row vector and a column vector in $\C^{m}$ and 
  $Q \in M_{m}(\B\langle x_1,\dots,x_d\rangle \!)$  is a $\B$-valued affine linear pencil of size $m\times m$ of the particular form $Q = Q_0 + Q_1 x_1 + \dots + Q_d x_d$
with $Q_0 \in M_m(\B)$ and scalar matrices $Q_1,\dots,Q_d \in M_{m}(\C) \subseteq M_m(\B)$.
The linear representation $\rho$ of $p$ is said to be \emph{selfadjoint}, if $v=u^\ast$ holds and if $Q$ is selfadjoint with respect to the involution on $M_{m}(\B\langle x_1,\dots,x_d\rangle)$. We also refer to $m$ as the \emph{dimension of $\rho$}.
\end{definition}

It is easily seen that the existence of a selfadjoint linear representation $\rho$ of $p$ forces $p$ to be selfadjoint. On the other hand, the existence of linear representations and especially the existence of selfadjoint linear representations are by no means clear; this will be ensured by the following theorem.

\begin{theorem}[Existence of linear representations]\label{thm:linearizations_existence}
Each $p\in\B\langle x_1,\dots,x_d\rangle$ admits a linear representation $\rho$. Moreover, if $p$ is selfadjoint, then $\rho$ can be chosen to be selfadjoint as well.
\end{theorem}

The proof of Theorem \ref{thm:linearizations_existence} is of constructive nature. In order to verify that each noncommutative $\B$-valued polynomial has a linear representation, we only need to observe the following: for every $b\in\B$ and for each of the formal variables $x_j$, $j=1,\dots,d$, linear representations are given by
$$\rho_{b} := \bigg(\begin{bmatrix} 0 & 1\end{bmatrix}, \begin{bmatrix} b & -\1\\ -\1 & 0\end{bmatrix}, \begin{bmatrix} 0\\ 1 \end{bmatrix}\bigg) \quad \text{and} \quad \rho_{x_j} := \bigg(\begin{bmatrix} 0 & 1\end{bmatrix}, \begin{bmatrix} x_j & -\1\\ -\1 & 0\end{bmatrix}, \begin{bmatrix} 0\\ 1 \end{bmatrix}\bigg),$$
respectively; if $\rho_1=(u_1,Q_1,v_1)$ and $\rho_2=(u_2,Q_2,v_2)$ are linear representations of the noncommutative $\B$-valued polynomials $p_1$ and $p_2$, then
$$\rho_1 \oplus \rho_2 := \bigg(\begin{bmatrix} u_1 & u_2\end{bmatrix}, \begin{bmatrix} Q_1 & 0\\ 0 & Q_2\end{bmatrix}, \begin{bmatrix} v_1\\ v_2 \end{bmatrix}\bigg) \quad \text{and} \quad \rho_1 \odot \rho_2 := \bigg(\begin{bmatrix} 0 & u_1\end{bmatrix}, \begin{bmatrix} v_1u_2 & Q_1\\ Q_2 & 0\end{bmatrix}, \begin{bmatrix} 0\\ v_2 \end{bmatrix}\bigg)$$
give pure linear representation of $p_1 + p_2$ and $p_1 \cdot p_2$, respectively.
When $p \in \B\langle x_1,\dots,x_d\rangle$ is selfadjoint, we can build a selfadjoint linear representation $\rho=(u,Q,v)$ as follows: we choose $p_0\in\B\langle x_1,\dots,x_d\rangle$ such that $p=p_0+p_0^\ast$ and we construct a linear representation $\rho_0 = (u_0,Q_0,v_0)$ of $p_0$ with the help of the the previously discussed rules; a selfadjoint linear representation $\rho = (u,Q,v)$ is then given by
$$\rho = \bigg(\begin{bmatrix} u_0 & v_0^\ast \end{bmatrix}, \begin{bmatrix} 0 & Q_0^\ast\\ Q_0 & 0\end{bmatrix}, \begin{bmatrix} u_0^\ast \\ v_0\end{bmatrix}\bigg).$$

As announced above, we will work in the sequel with linearizations, which are canonically associated with linear representations. The precise definition reads as follows.

\begin{definition}[Linearization]
Let $p\in \B\langle x_1,\dots,x_d\rangle$ be a noncommutative $\B$-valued polynomial and let $\rho=(u,Q,v)$ be any linear representation of $p$. The associated matrix
$$\hat{p}_\rho := \begin{bmatrix} 0 & u\\ v & Q\end{bmatrix} \in M_{{m}+1}(\B\langle x_1,\dots,x_d\rangle)$$
is called the \emph{linearization of $p$ with respect to $\rho$}.
\end{definition}

The motivation to work with linearizations comes from the well-known Schur complement formula, which relates resolvents of noncommutative $\B$-valued polynomials with resolvents of their associated linearizations. The following theorem gives the precise statement; the proof is analogous to that of \cite[Proposition 3.2]{B-M-S-17} and thus omitted. Note that for any matrix $A$ with entries $A_{i,j}$, we put $[A]_{i,j} := A_{i,j}$. 

\begin{theorem}\label{thm:linearization}
Let $p\in\B\langle x_1,\dots,x_d\rangle$ be a noncommutative $\B$-valued polynomial and let $\hat{p} := \hat{p}_\rho$ be the linearization of $p$ with respect to a linear representation $\rho=(u,Q,v)$ of $p$.
Moreover, let $\A$ be a unital complex algebra which contains $\B$ as a unital subalgebra with the same unit element. Consider a tuple $X=(X_1,\dots,X_d) \in \A^d$.
Then the following statements hold true:
\begin{enumerate}
 \item The evaluation $Q(X) = Q_0 + Q_1 X_1 + \dots + Q_d X_d$ of $Q$ at the point $X$ is invertible in $M_{m}(\A)$ and we have that $p(X) = - u Q(X)^{-1} v$.
 \item\label{it:linearization-ii} For any $b\in\B$, we have that $b-p(X)$ is invertible in $\A$ if and only if the matrix $L_{p(X)}(b) := \Lambda(b)-\hat{p}_\rho(X)$ is invertible in $M_{{m}+1}(\A)$, where
$$\Lambda(b) := \begin{bmatrix} b & 0 & \hdots & 0\\ 0 & 0 & \hdots & 0\\ \vdots & \vdots & \ddots & \vdots\\ 0 & 0 & \hdots & 0\end{bmatrix} \in M_{m+1}(\B).$$
In the case when one and hence both conditions are satisfied, we have that
$$(b - p(X))^{-1} = \Big(\big(\Lambda(b) - \hat{p}_\rho(X)\big)^{-1}\Big)_{11} = \big( L_{p(X)}(b)^{-1} \big)_{11}.$$
\end{enumerate} 
\end{theorem}

\subsubsection*{\bf Algebraic and analytic properties of linearizations}

Suppose that $\B$ is a unital $C^\ast$-algebra. Consider a noncommutative $\B$-valued polynomial $p\in\B\langle x_1,\dots,x_d\rangle$. According to Theorem \ref{thm:linearizations_existence}, we find a linear representation $\rho=(u,Q,v)$ of $p$. By definition, we have that the $\B$-valued affine linear pencil $Q = Q_0 + Q_1 x_1 + \dots + Q_d x_d$ is an invertible element in $M_{m}(\B\langle x_1,\dots,x_d\rangle)$. The latter is clearly a strong requirement and it is thus natural to ask, how this is reflected by properties of the coefficient matrices $Q_0,Q_1,\dots,Q_d$.
In this subsection, we address this question in order to control the norm of $Q^{-1}(X)$ and, at each point $b\in\B$ for which $\Lambda(b) - \hat{p}_\rho(X)$ becomes invertible, the norm of its inverse $(\Lambda(b) - \hat{p}_\rho(X))^{-1}$ for $d$-tuples $X$ over any unital $C^\ast$-algebra $\A$ in which $\B$ is unitally embedded.

This requires some preparations. We observe that each $P \in M_{m}(\B\langle x_1,\dots,x_d\rangle)$ can be written as $P = \sum^r_{k=0} P_k$, where each $P_k \in M_{m}(\B\langle x_1,\dots,x_d\rangle)$ is real homogeneous of degree $k$, i.e., $P_k(tx_1,\dots,tx_d) = t^k P_k(x_1,\dots,x_d)$ for all $t\in\R$. If we suppose additionally that $P_r \neq 0$, then the latter representation of $P$ is unique.

\begin{proposition}\label{prop:Q_inverse}
Consider $Q\in M_{m}(\B\langle x_1,\dots,x_d\rangle)$ which is of the particular form $Q = \widetilde{Q}_0 + \widetilde{Q}_1$ with $\widetilde{Q}_0 \in M_m(\B)$ and a matrix $\widetilde{Q}_1 \in M_{m}(\B\langle x_1,\dots,x_d\rangle)$ which is homogeneous of degree $1$. If $Q$ is invertible, then the following statements hold true:
\begin{enumerate}
 \item The matrix $\widetilde{Q}_0$ is invertible in $M_m(\B)$.
 \item There exists an integer $r\geq 0$ such that $(\widetilde{Q}_1 \widetilde{Q}_0^{-1})^{r+1} = 0$ and $(\widetilde{Q}_0^{-1} \widetilde{Q}_1)^{r+1} = 0$.
 \item For each unital $C^\ast$-algebra $\A$ with $\1 \in \B \subseteq \A$ and for every $d$-tuple $X=(X_1,\dots,X_d)$ in $\A^d$, the inverse of $Q(X)$ in $M_m(\A)$ can be bounded in operator norm by $$\|Q^{-1}(X)\| \leq \|\widetilde{Q}_0^{-1}\| \sum^r_{k=0} \|\widetilde{Q}_1(X) \widetilde{Q}_0^{-1}\|^k.$$
\end{enumerate}
\end{proposition}

\begin{proof}
We consider the homogeneous decomposition of the inverse $Q^{-1}$ in $M_{m}(\B\langle x_1,\dots,x_d\rangle)$, say $Q^{-1} = \sum^r_{k=0} P_k$. Since
$$\1_m = (\widetilde{Q}_0 + \widetilde{Q}_1) \Big(\sum^r_{k=0} P_k\Big) = \widetilde{Q}_0 P_0 + \Big(\sum^r_{k=1} (\widetilde{Q}_0 P_k + \widetilde{Q}_1 P_{k-1})\Big) + \widetilde{Q}_1 P_r$$
is a homogeneous decomposition of the unit element $\1_m$ in $M_{m}(\B\langle x_1,\dots,x_d\rangle)$, we conclude that
$$\widetilde{Q}_0 P_0 = \1_m,\qquad \widetilde{Q}_1 P_r = 0, \qquad\text{and}\qquad \widetilde{Q}_0 P_k + \widetilde{Q}_1 P_{k-1} = 0 \quad \text{for $k=1,\dots,r$}.$$
Similarly, we deduce from
$$\1_m = \Big(\sum^r_{k=0} P_k\Big) (\widetilde{Q}_0 + \widetilde{Q}_1) = P_0 \widetilde{Q}_0 + \Big(\sum^r_{k=1} (P_k \widetilde{Q}_0 + P_{k-1} \widetilde{Q}_1)\Big) + P_r \widetilde{Q}_1$$
that
$$P_0 \widetilde{Q}_0 = \1_m,\qquad P_r \widetilde{Q}_1 = 0, \qquad\text{and}\qquad P_k \widetilde{Q}_0 + P_{k-1} \widetilde{Q}_1 = 0 \quad \text{for $k=1,\dots,r$}.$$
In summary, it follows that $\widetilde{Q}_0$ is invertible with inverse $P_0$, as claimed in (i). 

Now as $P_k = - P_{k-1} (\widetilde{Q}_1 \widetilde{Q}_0^{-1})$ for any $k=1,\dots,r$, we infer that $P_k = (-1)^k \widetilde{Q}_0^{-1} (\widetilde{Q}_1 \widetilde{Q}_0^{-1})^k$ by induction for $k=0,\dots,r$. In particular, $0 = \widetilde{Q}_1 P_r = (-1)^r (\widetilde{Q}_1 \widetilde{Q}_0^{-1})^{r+1}$, which establishes the first of the properties stated in (ii). As we have that $\widetilde{Q}_0^{-1} (\widetilde{Q}_1 \widetilde{Q}_0^{-1})^{r+1} \widetilde{Q}_0 = (\widetilde{Q}_0^{-1} \widetilde{Q}_1)^{r+1}$, the second property follows from the latter.

Thanks to the previous observations, we obtain the identity $Q^{-1} = \sum^r_{k=0} (-1)^k \widetilde{Q}_0^{-1} (\widetilde{Q}_1 \widetilde{Q}_0^{-1})^k$ in $M_{m}(\B\langle x_1,\dots,x_d\rangle)$. Under evaluation at any given $d$-tuple $X$ in $\A^d$, the latter gives that $Q(X)^{-1} = \sum^r_{k=0} (-1)^k \widetilde{Q}_0^{-1} (\widetilde{Q}_1(X) \widetilde{Q}_0^{-1})^k$, because we obviously have that $Q(X)^{-1} = Q^{-1}(X)$. By using the triangle inequality, (iii) follows.
\end{proof}

We conclude by the following lemma which generalizes \cite[Lemma 21]{Be-Bo-Ca-Ce-19} and \cite[Lemma 2.5]{EKY18}.

\begin{lemma}\label{lem:norm-L_inverse}
Let $p\in \B\langle x_1,\dots,x_d\rangle$ and let $\rho=(u,Q,v)$ be a linear representation of $p$ of dimension $m\geq 1$; put $w_\rho := \max\{v^\ast v, u u^\ast\}$. Suppose that $\A$ is a unital $C^\ast$-algebra satisfying $\1 \in \B \subseteq \A$ and that $X_1,\dots,X_d \in \A$. Then, for any $b \in \B$, the matrix $L_{p(X)}(b) = \Lambda(b) - \hat{p}_\rho(X)$ is invertible in $M_{m+1}(\A)$ if and only if $b - p(X)$ is invertible in $\A$. In this case, we have
\[
\| L_{p(X)}(b)^{-1} \| \leq \| Q(X)^{-1} \| + \big( 1 + w_\rho \| Q(X)^{-1} \|^2 \big)  \| (b - p(X))^{-1} \|.
\]
\end{lemma}

\begin{proof}
The equivalence property in the statement of the lemma is insured by Theorem \ref{thm:linearization} \ref{it:linearization-ii}. Noting that
\[
\left\|\begin{bmatrix}
\1 \\ - Q(X)^{-1} v\end{bmatrix}\right\| \leq \sqrt{1 + v^\ast v \|Q(X)^{-1}\|^2} \quad\text{and}\quad \left\| \begin{bmatrix}
\1 & -u Q(X)^{-1} \end{bmatrix} \right\| \leq \sqrt{1 + u u^\ast \|Q(X)^{-1}\|^2},
\]
we end the proof by writing
\begin{equation}\label{eq:Schur_decomposition}
\begin{aligned}
L_{p(X)}(b)^{-1} &=
 \begin{bmatrix}
\1 & 0 \\ -Q(X)^{-1}v & \1_m
\end{bmatrix}
\begin{bmatrix}
(b-p(X))^{-1} & 0 \\ 0 & -Q(X)^{-1}
\end{bmatrix}
 \begin{bmatrix}
\1
& -uQ(X)^{-1} \\ 0 & \1_m
\end{bmatrix}\\
&= \begin{bmatrix}
0 & 0\\ 0 & -Q(X)^{-1}
\end{bmatrix} +
\begin{bmatrix}
\1\\ - Q(X)^{-1} v\end{bmatrix} (b-p(X))^{-1}
\begin{bmatrix}
\1 & -u Q(X)^{-1}
\end{bmatrix}
\end{aligned}
\end{equation}
and then taking the operator norm.
\end{proof}

\section{Bound for the L\'evy distance}\label{sec:Levy_bound}

For any Borel probability measure $\mu$ on $\R$ and each $\epsilon>0$, we define a Borel probability measure $\mu_\epsilon$ by $\mathrm{d}\mu_\epsilon(t) = -\frac{1}{\pi} \Im(\G_\mu(t+i\epsilon))\, \mathrm{d}t$. Notice that $\mu_\epsilon = \mu \ast \gamma_\epsilon$ where $\gamma_\epsilon$ is the Borel probability measure given by $d\gamma_\epsilon(t) := \frac{1}{\pi}\frac{\epsilon}{\epsilon^2+t^2}\, \mathrm{d}t$.

The known Stieltjes inversion states that $\mu_\epsilon$ converges weakly to $\mu$ as $\epsilon \searrow 0$. In \cite{Salazar2020}, Salazar obtained a quantified version of the Stieltjes inversion in terms of the L\'evy distance defined in Section \ref{section:Levy-Kolmogorov}. Namely, he proved that $L(\mu_\epsilon,\mu) \leq \sqrt{2 \frac{\epsilon}{\pi}}$ for all $\epsilon>0$. In the following lemma, we provide an improved version of this bound; our approach is inspired by the derivation of a similar but also weaker bound presented in the appendix of the lecture notes ``Random matrix theory'' (2017) by M. Krishnapur.

\begin{lemma}\label{lem:Stieltjes_quantified}
Let $\mu$ be a Borel probability measure on $\R$. Then, for $\epsilon>0$, we have $L(\mu_\epsilon,\mu) \leq \sqrt{\frac{\epsilon}{\pi}}$.
\end{lemma}

\begin{proof}
We prove that for all $t_0\in\R$ and each $\delta>0$
\begin{equation}\label{eq:Levy_estimates}
\F_{\mu_\epsilon}(t_0) \geq \F_\mu(t_0-\delta) - \frac{\epsilon}{\pi\delta} \qquad\text{and}\qquad \F_\mu(t_0)  \; \geq \;  \F_{\mu_\epsilon}(t_0-\delta) - \frac{\epsilon}{\pi\delta}.
\end{equation}
Once this is shown, it follows with the particular choice $\delta = \sqrt{\frac{\epsilon}{\pi}}$ that $L(\mu_\epsilon,\mu) \leq \delta$ and hence the assertion.

First of all, we use the fact $\mu_\epsilon = \mu \ast \gamma_\epsilon$ to write
$$\F_{\mu_\epsilon}(t_0) = (\mu \otimes \gamma_\epsilon)\big(\{(t,s)\in\R^2 \mid t + s \leq t_0\}\big).$$

For $\delta>0$, we consider the two sets
\begin{multline*}
A(t_0,\delta) := \{(t,s) \in\R^2 \mid (t \leq t_0-\delta)\, \lor\, (t + s > t_0)\} \qquad\text{and}\\
B(t_0,\delta) := \{(t,s) \in\R^2 \mid (t \leq t_0-\delta)\, \land\, (t + s > t_0)\}.
\end{multline*}
First, we notice that $B(t_0,\delta) \subseteq \R \times (\delta,\infty)$, so that
$$(\mu\otimes\gamma_\epsilon)(B(t_0,\delta)) \leq \frac{1}{\pi} \int_{\delta}^\infty \frac{\epsilon}{\epsilon^2+s^2}\, \mathrm{d}s \leq \frac{\epsilon}{\pi\delta}.$$
Next, we observe that
\begin{align*}
1 &\geq (\mu\otimes\gamma_\epsilon)(A(t_0,\delta))\\
  &= (\mu\otimes\gamma_\epsilon)\big((-\infty,t_0-\delta] \times \R\big) + (\mu\otimes\gamma_\epsilon)\big(\{(t,s)\in\R^2 \mid t + s > t_0\}\big) - (\mu\otimes\gamma_\epsilon)(B(t_0,\delta))\\
	&\geq \F_\mu(t_0-\delta) + (\mu\otimes\gamma_\epsilon)\big(\{(t,s)\in\R^2 \mid t + s > t_0\}\big) - \frac{\epsilon}{\pi\delta},
\end{align*}
from which we infer that
$$\F_{\mu_\epsilon}(t_0) = 1 - (\mu\otimes\gamma_\epsilon)\big(\{(t,s)\in\R^2 \mid t + s > t_0\}\big) \geq \F_\mu(t_0-\delta) - \frac{\epsilon}{\pi\delta},$$
which is the first inequality in \eqref{eq:Levy_estimates}.

In order to prove the second inequality, we consider the two sets
\begin{multline*}
A'(t_0,\delta) := \{(t,s) \in\R^2 \mid (t > t_0)\, \lor\, (t + s \leq t_0-\delta)\} \qquad\text{and}\\
B'(t_0,\delta) := \{(t,s) \in\R^2 \mid (t > t_0)\, \land\, (t + s \leq t_0 - \delta)\}.
\end{multline*}
First, we notice that $B'(t_0,\delta) \subseteq \R \times (-\infty,-\delta)$, so that
$$(\mu\otimes\gamma_\epsilon)(B'(t_0,\delta)) \leq \frac{1}{\pi} \int_{-\infty}^{-\delta} \frac{\epsilon}{\epsilon^2+s^2}\, \mathrm{d}s \leq \frac{\epsilon}{\pi\delta}.$$
Next, we observe that
\begin{align*}
1 &\geq (\mu\otimes\gamma_\epsilon)(A'(t_0,\delta))\\
  &= (\mu\otimes\gamma_\epsilon)\big((t_0,\infty) \times \R\big) + (\mu\otimes\gamma_\epsilon)\big(\{(t,s)\in\R^2 \mid t + s \leq t_0-\delta\}\big) - (\mu\otimes\gamma_\epsilon)(B'(t_0,\delta))\\
	&\geq (1-\F_\mu(t_0)) + (\mu\otimes\gamma_\epsilon)\big(\{(t,s)\in\R^2 \mid t + s \leq t_0-\delta\}\big) - \frac{\epsilon}{\pi\delta},
\end{align*}
from which we infer that
$$\F_\mu(t_0) \geq (\mu\otimes\gamma_\epsilon)\big(\{(t,s)\in\R^2 \mid t + s \leq t_0-\delta\}\big) - \frac{\epsilon}{\pi\delta} = \F_{\mu_\epsilon}(t_0-\delta) - \frac{\epsilon}{\pi\delta},$$
which is the second inequality in \eqref{eq:Levy_estimates}. 
\end{proof}

As observed in \cite{Salazar2020}, Lemma \ref{lem:Stieltjes_quantified} allows us to derive bounds for the L\'evy distance in terms of Cauchy transforms. Suppose that $\mu$ and $\nu$ are Borel probability measures on $\R$ and let $\epsilon>0$ be given. It is easy to verify that one has
$$L(\mu_\epsilon,\nu_\epsilon) \leq \frac{1}{\pi} \int^\infty_{-\infty} \big|\Im(\G_\mu(t+i\epsilon)) - \Im(\G_\nu(t+i\epsilon))\big|\, \mathrm{d}t.$$
In combination with the Lemma \ref{lem:Stieltjes_quantified}, the latter bound yields that
$$L(\mu,\nu) \leq L(\mu,\mu_\epsilon) + L(\mu_\epsilon,\nu_\epsilon) + L(\nu_\epsilon,\nu) \leq 2 \sqrt{\frac{\epsilon}{\pi}} + \frac{1}{\pi} \int^\infty_{-\infty} \big|\Im(\G_\mu(t+i\epsilon)) - \Im(\G_\nu(t+i\epsilon))\big|\, \mathrm{d}t,$$
which is the bound asserted in \eqref{eq:Levy_bound}.

\bigskip
\nocite{*}

\bibliographystyle{abbrv}
\bibliography{ref}

\end{document}